\documentclass[a4paper,11pt]{amsart}
\usepackage{mathrsfs}
\usepackage{amssymb,enumerate}
\usepackage{amsrefs}

\newcommand{\Complex}{\mathbb C}
\newcommand{\Real}{\mathbb R}
\newcommand{\ddbar}{\overline\partial}
\newcommand{\pr}{\partial}
\newcommand{\ol}{\overline}
\newcommand{\Td}{\widetilde}
\newcommand{\norm}[1]{\left\Vert#1\right\Vert}
\newcommand{\abs}[1]{\left\vert#1\right\vert}
\newcommand{\set}[1]{\left\{#1\right\}}
\newcommand{\To}{\rightarrow}

\newcommand{\dbarb}{\ddbar_b}

\newcommand{\Zbar}{\overline{Z}}

\newcommand{\m}{\theta \wedge d\theta}
\newcommand{\teta}{\tilde{\eta}}
\theoremstyle{plain}
\newtheorem{thm}{Theorem}[section]
\newtheorem{cor}[thm]{Corollary}
\newtheorem{lem}[thm]{Lemma}
\newtheorem{prop}[thm]{Proposition}

\theoremstyle{definition}

\theoremstyle{remark}

\numberwithin{equation}{section}

\begin{document}
\title[]{Solving the Kohn Laplacian on asymptotically flat CR manifolds of dimension 3} 
\author[]{Chin-Yu Hsiao}
\address{Institute of Mathematics, Academia Sinica, 6F, Astronomy-Mathematics Building, No.1, Sec.4, Roosevelt Road, Taipei 10617, Taiwan}
\email{chsiao@math.sinica.edu.tw}
\author[]{Po-Lam Yung}
\address{Mathematical Institute, University of Oxford, OX2 6GG, United Kingdom}
\email{yung@maths.ox.ac.uk}
\thanks{The first author is supported in part by the DFG funded project MA 2469/2-1.}
\thanks{The second author is supported in part by NSF grant DMS 1201474.} 
\date{\today}

\begin{abstract} 
Let $(\hat X,T^{1,0}\hat X)$ be a compact orientable CR embeddable three dimensional strongly pseudoconvex CR manifold, where $T^{1,0}\hat X$ is a CR structure on $\hat X$. Fix a point $p\in\hat X$ and take a global contact form $\hat\theta$ so that $\hat\theta$ is asymptotically flat near $p$. Then $(\hat X,T^{1,0}\hat X,\hat\theta )$ is a pseudohermitian $3$-manifold. Let $G_p\in C^\infty(\hat X\setminus\set{p})$, $G_p > 0$, with $G_p(x)\sim\vartheta(x,p)^{-2}$ near $p$, where $\vartheta(x,y)$ denotes the natural pseudohermitian distance on $\hat X$. Consider the new pseudohermitian $3$-manifold with a blow-up of contact form $(\hat X\setminus\set{p},T^{1,0}\hat X,G^2_p\hat\theta)$ and let $\Box_{b}$ denote the corresponding Kohn Laplacian on $\hat X\setminus\set{p}$.

In this paper, we prove that the weighted Kohn Laplacian $G^2_p\Box_b$ has closed range in $L^2$ with respect to the weighted volume form $G^2_p\hat\theta\wedge d\hat\theta$, and that the associated partial inverse and the Szeg\"{o} projection enjoy some regularity properties near $p$. As an application, we prove the existence of some special functions in the kernel of $\Box_{b}$ that grow at a specific rate at $p$. The existence of such functions provides an important ingredient for the proof of a positive mass theorem in 3-dimensional CR geometry by Cheng-Malchiodi-Yang \cite{CMY}.
\end{abstract}  

\maketitle \tableofcontents

\section{Introduction} 

\subsection{Motivation from CR geometry}

The study described in this paper was motivated by that on a positive mass theorem in 3-dimensional CR geometry by Cheng-Malchiodi-Yang~\cite{CMY}, where one needs to find some special functions in the kernel of the Kohn Laplacian that grow at a specific rate at a given point on an asymptotically flat pseudohermitian $3$-manifold. We begin by giving a brief description of the relevance of our result with their work below. 

Consider a compact orientable 3-dimensional strongly pseudoconvex CR manifold $(\hat X,T^{1,0}\hat X)$ with CR structure $T^{1,0}\hat X$. We assume throughout that it is CR embeddable in some $\mathbb{C}^N$. By choosing a contact form $\hat \theta_0$ on $\hat X$ that is compatible with its CR structure, one can make $(\hat X,T^{1,0}\hat X,\hat\theta_0)$ a pseudohermitian 3-manifold; in particular, one can define a Hermitian inner product on $T^{1,0} \hat X$, by
$$
\langle Z_1 | Z_2 \rangle_{\hat \theta_0} = \frac{1}{2} d\hat \theta_0(Z_1, i\overline{Z}_2).
$$ 
Now fix $p \in\hat X$. By conformally changing the contact form, we may find another contact form $\hat\theta$ (which is a multiple of $\hat \theta_0$ by a positive smooth function), so that near $p$, there exists CR normal coordinates $(z,t)$. In other words, the contact form $\hat \theta$ and the coordinates $(z,t)$ are chosen, so that
\begin{enumerate}[(i)]
\item the point $p$ corresponds to $(z,t) = (0,0)$;
\item one can find a local section $\hat Z_1$ of $T^{1,0} \hat X$ near $p$, with $\langle \hat Z_1 | \hat Z_1 \rangle_{\hat \theta} = 1$, such that $\hat Z_1$ admits an expansion near $p$ as described in \eqref{s1-e5b} below; and
\item the Reeb vector field $\hat T$ with respect to $\hat \theta$ admits an expansion as described in \eqref{s1-e5c}. 
\end{enumerate} 
Then $(\hat X,T^{1,0}\hat X,\hat\theta )$ is another pseudohermitian $3$-manifold. Assume that this pseudohermitian $3$-manifold is of positive Tanaka-Webster class: this means that the lowest eigenvalue of the conformal sublaplacian 
$$L_b:=-4\Delta_b+R$$ is strictly positive. Here $R = R_{\hat \theta}$ is the Tanaka-Webster curvature of $\hat X$, and $\Delta_b$ is the sublaplacian on $\hat X$. (The above assumption on $L_b$ will hold when e.g. $R_{\hat \theta}$ is strictly positive on $\hat X$.) Then $L_b$ is invertible, so one can write down the Green's function $G_p$ of $L_b$ with pole at $p$. We normalize $G_p$ so that
$$L_b G_{p}=16\delta _{p}.$$
Then near $p$, $G_p$ admits the following expansion
\begin{equation} \label{eq:Gpexpand0}
G_p=\frac{1}{2\pi}\hat\rho^{-2}+A+f,\ \ f\in\mathcal{E}(\hat\rho^{1}),
\end{equation}
where $A$ is some real constant, $\hat\rho^4(z,t)=\abs{z}^4+t^2$, and for $m\in\Real$, $\mathcal{E}(\hat\rho^m)$ denotes, roughly speaking, the set of all smooth functions $g\in C^\infty(\hat X\setminus\set{p})$ such that $\abs{
g(z,t)}\leq \hat\rho(z,t)^{m-p-q-2r}$ near $p$,
along with some suitable control of the growth of derivatives near $p$ (see \eqref{s1-e6} for the precise meaning of the Fr\'echet space $\mathcal{E}(\hat\rho^m)$). 

Let now $$X = \hat X \setminus \{p\} \quad \text{and} \quad \theta = G_p^2 \hat \theta$$
and let $T^{1,0} X$ be the restriction of $T^{1,0} \hat X$ to $X$. Then $(X,T^{1,0} X,\theta)$ is a new non-compact pseudohermitian manifold, which we think of as the blow-up of our original $(\hat X, T^{1,0} \hat X, \hat \theta)$ at $p$. We say that this pseudohermitian manifold is asymptotically flat, since under an inversion of coordinates, $X$ has asymptotically the geometry of the Heisenberg group at infinity. We note that the Tanaka-Webster scalar curvature $R_{\theta}$ of $X$ is identially zero, since the conformal factor $G_p$ we used is the Green's function for the conformal sublaplacian on $\hat X$. Let $\ddbar_b$ and $\Box_b$ denote the associated tangential Cauchy-Riemann operator and Kohn Laplacian respectively. In~\cite{CMY}, Cheng-Malchiodi-Yang introduced the pseudohermitian $p$-mass for $(X,T^{1,0} X, \theta)$, given by $$m(\theta):=\lim_{\Lambda\To0}i\int_{\set{\hat\rho=\Lambda}}\omega^1_1\wedge\theta,$$ where $\omega^1_1$ stands for the connection form of the given pseudohermitian structure. Then they proved that there is a specific $\Td\beta\in\mathcal{E}(\hat\rho^{-1})$, with 
$\Box_b\Td\beta=\mathcal{E}(\hat\rho^{4})$, such that 
\begin{equation}\label{e-intrIImi}
m(\theta) = -\frac{3}{2} \int_X\abs{\Box_b\Td\beta}^2\theta\wedge d\theta + 3 \int_X |\Td \beta_{,\bar{1} \bar{1}}|^2 \theta\wedge d\theta + \frac{3}{4} \int_X \overline{\Td \beta} \cdot P \Td \beta \theta\wedge d\theta,
\end{equation}
where $\Td \beta_{,\bar{1} \bar{1}}$ is some derivative of the function $\Td \beta$, and $P$ is the CR Paneitz operator of $(X,T^{1,0}X,\theta)$.
(Note $R_{\theta} = 0$ in our current set-up, so the term involving $R_{\theta}$ in the corresponding identity of mass in~\cite{CMY} is not present above.) Moreover, it was shown that \eqref{e-intrIImi} holds for any $\beta$ in place of $\Td\beta$, as long as $\Td\beta-\beta\in\mathcal{E}(\hat\rho^{1+\delta})$, and $\Box_b\beta=\mathcal{E}(\hat\rho^{3+\delta})$ for some $\delta > 0$. Thus, if we could find such a $\beta$ in the \emph{kernel} of $\Box_b$, then under the assumption that the CR Paneitz operator $P$ is non-negative, one can conclude that the mass $m(\theta)$ is non-negative. The construction of such $\beta$ is the motivation of the current paper. (See Corollary~\ref{cor1} in the next subsection.)

Classically, if one wants to solve $\Box_b$ on say a compact CR manifold, one proceeds by showing first that $\Box_b$ extends to a closed linear operator on $L^2$, and that this extended $\Box_b$ has closed range on $L^2$. Then one solves $\Box_b$ in a weak sense, and shows that the solution is classical if the right hand side of the equation is smooth. This strategy does not directly apply in our situation, since our CR manifold $X$ is non-compact. In fact, the natural volume form on $X$ is given by $\theta \wedge d\theta$, and even if we extend $\Box_b$ so that it becomes a closed linear operator on $L^2(\theta \wedge d\theta)$, this operator may not have closed range in $L^2(\theta \wedge d\theta)$, as is seen in some simple examples (e.g. when $\hat X$ is the unit sphere in $\mathbb{C}^2$).

To overcome this difficulty, we introduce in this paper a weighted volume form
$$m_1:=G^{-2}_p\theta\wedge d\theta,$$
as well as a weighted Kohn Laplacian, namely $$\Box_{b,1}:=G^2_p\Box_b.$$ We will show that $\Box_{b,1}$ extends as a densely defined closed linear operator
$$
\Box_{b,1} \colon {\rm Dom\,} \Box_{b,1} \subset L^2(m_1) \to L^2(m_1),
$$
and this extended operator has closed range in $L^2(m_1)$. (See Theorem~\ref{thm2} below.) Here $L^2(m_1)$ is the space of $L^2$ functions with respect to the volume form $m_1$. As a result, we have the following $L^2$ decomposition:
\[\Box_{b,1}N+\Pi=I\ \ \mbox{on}\ \ L^2(m_1)\]
Here $N:L^2(m_1)\To{\rm Dom\,}\Box_{b,1}$ is the partial inverse of $\Box_{b,1}$ and $\Pi:L^2(m_1)\To({\rm Ran\,}\Box_{b,1})^\bot$ is the orthogonal projection onto the orthogonal complement of the range of $\Box_{b,1}$ in $L^2(m_1)$. We will show that 
for every $0<\delta<2$, $N$ and $\Pi$ can be extended continuously to
\begin{equation}\label{i-main-e2}
\begin{split}
&N:\mathcal{E}(\hat\rho^{-2+\delta})\To\mathcal{E}(\hat\rho^{\delta}),\\
&\Pi:\mathcal{E}(\hat\rho^{-2+\delta})\To\mathcal{E}(\hat\rho^{-2+\delta})
\end{split}
\end{equation} 
(see Theorem~\ref{thm3} below.)
Hence 
\begin{equation}\label{i-main-e1ad}
\Box_{b,1}N+\Pi=I\ \ \mbox{on}\ \ \mathcal{E}(\hat\rho^{-2+\delta}),
\end{equation}  
for every $0<\delta<2$. Now, let $\Td\beta$ be as in \eqref{e-intrIImi}. Put $$f:=\Box_{b,1}\Td\beta=G^2_p\Box_b\Td\beta.$$ Then by the expansion of $G_p$ and the assumption on $\Box_b \Td\beta$, we have $f \in\mathcal{E}(\hat\rho^{-1+\delta})$, for every $0<\delta<1$. From \eqref{i-main-e2}, we know that $\Pi f$ is well-defined and $Nf\in\mathcal{E}(\hat\rho^{1+\delta})$, for every $0<\delta<1$. Moreover, we will show, in Theorem~\ref{thm4} below, that $$\Pi f=0.$$ Thus from \eqref{i-main-e1ad}, we have $$\Box_{b,1}(Nf)=f=\Box_{b,1}\Td\beta.$$ If we put $$\beta:=\Td\beta-Nf,$$ then $$\Box_b\beta=0 \quad \text{and} \quad \beta-\Td\beta\in\mathcal{E}(\hat\rho^{1+\delta}),$$ for every $0<\delta<1$. With this we have achieved our goal.

It turns out that a large part of our analysis does not depend on the fact that $G_p$ is the Green's function of a conformal sublaplacian. All that we need is that $G_p$ admits an expansion as in (\ref{eq:Gpexpand0}), that it is smooth on $X$, and that it is positive everywhere on $X$. We will formulate our result in this framework in the next subsection.

We expect that it is possible to approach the same problem by proceeding via $L^p$ spaces rather than weighted $L^2$ spaces. In~\cite{HY13}, we solved the $\Box_b$ equation in some $L^p$ spaces in a special case. 

The operator $\Box_{b,1}$ can be seen as the Kohn Laplacian on the non-compact CR manifold $X=\hat X\setminus{\set{p}}$ defined with respect to the natural CR structure
$T^{1,0}\hat X$ and the ``singular" volume form $m_1$. The coefficients of $\Box_{b,1}$ are smooth on $X$ but singular at $p$. This work can be seen as a first study of this kind of ``singular Kohn Laplacians''. It will be quite interesting to develop some kind of ``singular'' functional calculus for pseudodifferential operators and Fourier integral operators and establish a completely microlocal analysis for $\Box_{b,1}$ along the lines of Beals-Greiner~\cite{BG88}, Boutet de Monvel-Sj\"{o}strand~\cite{BouSj76} and~\cite{Hsiao08}. We hope that the ``singular Kohn Laplacians'' will be interesting for analysts. 

\subsection{Our main result}

Let us now formulate our main results in their full generality. Consider a compact orientable 3-dimensional strongly pseudoconvex pseudohermitian manifold $\hat{X}$, with CR structure $T^{1,0}\hat X$ and contact form $\hat\theta_0$. We assume throughout that it is CR embeddable in some $\mathbb{C}^N$. By conformally changing the contact form $\hat\theta_0$, we may find another contact form $\hat\theta$, so that near $p$, one can find CR normal coordinates $(z,t)$ as described in the previous subsection. We will write 
$$\hat\rho(z,t)=(\abs{z}^4+t^2)^{1/4},$$
and for every $m \in \Real$, we can define a Fr\'echet function space $\mathcal{E}(\hat\rho^m)$ as in \eqref{s1-e6}.

Now fix a point $p \in \hat{X}$, and let $$X = \hat X \setminus \{p\}.$$ We fix from now on an everywhere positive function $G_p \in C^{\infty}(X)$, such that $G_p$ admits an expansion
\begin{equation} \label{eq:Gpexpand}
G_p=\frac{1}{2\pi}\hat\rho^{-2}+A+f,\ \ f\in\mathcal{E}(\hat\rho^{1}),
\end{equation}
where $A$ is some real constant. (Again $G_p$ need not be the Green's function of the conformal sublaplacian any more.) Define now
$$
\theta = G_p^2 \hat \theta.
$$
Then $(X, T^{1,0}\hat X, \theta)$ is a non-compact pseudohermitian 3-manifold, which we think of as the blow-up of the original $\hat X$. The $\theta$ defines for us a pointwise Hermitian inner product on $T^{1,0} X$, given by
$$
\langle Z_1 | Z_2 \rangle_{\theta} = \frac{1}{2} d\theta (Z_1, i\overline{Z}_2);
$$
we denote the dual pointwise inner product on the space $(0,1)$ forms on $X$ by the same notation $\langle \cdot | \cdot \rangle_{\theta}$.
Let $\dbarb$ be the tangential Cauchy-Riemann operator on $X$. This is defined depending only on the CR structure on $X$. Now there is a natural volume form on $X$, given by 
$$m := \theta \wedge d\theta.$$
This induces an inner product on functions on $X$, given by
$$(f|g)_m = \int_X f \overline{g} \, m,$$
and an inner product on $(0,1)$ forms on $X$, given by
$$(\alpha|\beta)_{m,\theta} = \int_X \langle \alpha| \beta \rangle_{\theta} \, m.$$
We write $\ol{\pr}^{*,f}_b$ for the formal adjoint of $\dbarb$ under these two inner products. In other words, $\ol{\pr}^{*,f}_b$ satisfies 
\[(\ddbar_bu|v)_{m,\theta}=(u|\ol{\pr}^{*,f}_bv)_m\]
for all functions $u$ and $(0,1)$ forms $v$ on $X$ that are smooth with compact support.
We can now define the Kohn Laplacian on $X$, namely
$$ 
\Box_b:=\ddbar^{*,f}_b\ddbar_b,
$$ 
at least on smooth functions with compact support on $X$; then 
$$(\Box_bu\ |\ f)_{m}=(u\ |\ \Box_bf)_{m}$$ for all functions $u$, $f$ on $X$ that are smooth with compact support, so we can extend $\Box_b$ to distributions on $X$ by duality.

Our goal is then to solve a specific equation involving $\Box_b$. First, let $\chi(z,t)$ be a smooth function with compact support on $\hat X$, so that its support is contained in the local coordinate chart given by the CR normal coordinates $(z,t)$, and that it is identically 1 in a neighborhood of $p$. Let
$$
\beta_0 = \chi(z,t) \frac{i\overline{z}}{|z|^2-it} \in \mathcal{E}(\hat \rho^{-1}).
$$
Then $\Box_b \beta_0 \in \mathcal{E}(\hat\rho^{3}).$
Furthermore, as was shown in \cite{CMY}, there exists $\beta_1 \in \mathcal{E}(\hat\rho^{1})$, such that if
$$ 
\tilde{\beta} := \beta_0 + \beta_1,
$$
then
$$F := \Box_b \tilde{\beta} \in \mathcal{E}(\hat\rho^{4}).$$
Actually, in what follows, all we will use is that $F \in \mathcal{E}(\hat\rho^{3+\delta})$ for all $0 < \delta < 1$. Our main theorem can now be stated as follows:

\begin{thm} \label{thm1} 
Let $F$ be as defined above. Then there exists a smooth function $u$ on $X$, such that $u \in \mathcal{E}(\hat\rho^{1+\delta})$ for any $0 < \delta < 1$, and $$\Box_b u = F.$$
\end{thm}

By taking $\beta := \Td \beta - u$, we then have:

\begin{cor} \label{cor1}
There exists $\beta \in \mathcal{E}(\hat \rho^{-1})$ with $\beta - \Td \beta \in \mathcal{E}(\hat \rho^{1+\delta})$ for any $0 < \delta < 1$, such that $$\Box_b \beta = 0.$$ 
\end{cor}

This provides a key tool in the proof of a positive mass theorem in 3-dimensional CR geometry in the work of Cheng-Malchiodi-Yang \cite{CMY}, as was explained in the last subsection.

Some remarks are in order. The first is about numerology. Considerations of homogenity shows that $\Box_b$ takes a function in $\mathcal{E}(\hat \rho^k)$ to $\mathcal{E}(\hat \rho^{k+2})$. Thus the homogeneity above works out right; the only small surprise is that while $\beta_0$ is in $\mathcal{E}(\hat \rho^{-1})$, $\Box_b \beta_0$ is in $\mathcal{E}(\hat \rho^3)$, which is 2 orders better than expected. But that is a reflection of the fact that our $\beta_0$ has been chosen such that $\ddbar_b \beta_0$ is almost annihilated by $\ddbar^{*,f}_b$. 

Next, $\beta_1$ above is an explicit correction in $\mathcal{E}(\hat \rho^1)$ such that $\Box_b \beta_1 \in \mathcal{E}(\hat \rho^3)$ cancels out the main contribution of $\Box_b \beta_0 \in \mathcal{E}(\hat \rho^3)$. This ensures that $F \in \mathcal{E}(\hat \rho^{3+\delta})$ for all $0 < \delta < 1$, which in turn guarantees that the $\beta$ we construct in Corollary~\ref{cor1} is determined explicitly up to $\mathcal{E}(\hat \rho^{1+\delta})$ for all (in particular, for some) $0 < \delta < 1$. The latter is important in the proof of the positive mass theorem in \cite{CMY}, since any term in the expansion of $\beta$ that is in $\mathcal{E}(\hat \rho^1)$ would enter into the calculation of mass in \eqref{e-intrIImi}. But for the purpose of solving $\Box_b$ in the current paper, the correction term $\beta_1$ is not essential; in particular, if $F_0 := \Box_b \beta_0 \in \mathcal{E}(\hat \rho^3),$ then our proof below carries over, and shows that there exists $u_0 \in \mathcal{E}(\hat \rho^1)$ such that $\Box_b u_0 = F_0$.

Finally, in Corollary~\ref{cor1}, note that we do not claim $\ddbar_b \beta = 0$. It is only $\Box_b \beta$ that vanishes, as can be shown by say the example when $\hat X$ is the standard CR sphere in $\mathbb{C}^2$.

\subsection{Our strategy} \label{subsect:strategy}

As we mentioned earlier, the difficulty in establishing the above theorem is that the CR manifold we are working on, namely $X$, is non-compact; also, the natural measure on $X$, namely $m = \theta \wedge d\theta$, has infinite volume on $X$. Let $L^2(m)$ be the space of $L^2$ functions on $X$ with respect to $m$. Even if we extend $\Box_b$ to be a closed linear operator on $L^2(m) \to L^2(m)$, in general the extended $\Box_b$ may not have closed range in $L^2(m)$. Thus the classical methods of solving $\Box_b$ fail in our situation. 

We thus proceed by introducing a weighted $L^2$ space, and a weighted Kohn Laplacian. Let $$m_1:=G^{-2}_p\theta\wedge d\theta.$$
We define $L^2(m_1)$ to be the space of $L^2$ functions on $X$ with respect to the inner product
$$
(f|g)_{m_1} := \int_X f \overline g \, m_1,
$$
and define $L^2_{(0,1)}(m_1,\hat\theta)$ to be the space of $L^2$ $(0,1)$ forms on $X$ with respect to the inner product
$$
(\alpha|\beta)_{m_1,\hat\theta} := \int_X \langle \alpha | \beta \rangle_{\hat\theta}  \, m_1.
$$
We extend the tangential Cauchy-Riemann operator so that 
$$
\text{Dom} \ddbar_{b,1} := \{ u \in L^2(m_1) \colon \text{the distributional $\ddbar_b$ of $u$ is in $L^2_{(0,1)}(m_1,\hat\theta)$} \},
$$
and define
$$
\ddbar_{b,1} u := \text{the distributional $\ddbar_b$ of $u$}
$$
if $u \in \text{Dom} \ddbar_{b,1}$.
Then
$$
\ddbar_{b,1} \colon {\rm Dom\,} \ddbar_{b,1} \subset L^2(m_1)\To L^2_{(0,1)}( m_1,\hat\theta),
$$
is a densely defined closed linear operator. Let 
$$
\ddbar_{b,1}^* \colon {\rm Dom\,} \ddbar_{b,1}^* \subset L^2_{(0,1)}(m_1,\hat\theta) \To L^2(m_1)
$$
be its adjoint. Let $\Box_{b,1}$ denote the Gaffney extension of the singular Kohn Laplacian given by 
$$ 
{\rm Dom\,}\Box_{b,1}
=\set{s\in L^2(m_1)\colon s\in{\rm Dom\,}\ddbar_{b,1},\ \ddbar_{b,1}s\in{\rm Dom\,}\ol{\pr}^{*}_{b,1}}\, ,
$$ 
and $\Box_{b,1}s=\ol{\pr}^{*}_{b,1}\ddbar_{b,1}s$ for $s\in {\rm Dom\,}\Box_{b,1}$. By a result of Gaffney, 
$\Box_{b,1}$ is a positive self-adjoint operator (see~\cite[Prop.\,3.1.2]{MM07}). We extend $\Box_{b,1}$ to distributions on $X$ by 
\begin{align*}
&\Box_{b,1}:\mathscr D'(X)\To\mathscr D'(X),\\
&(\Box_{b,1}u\ |\ f)_{m_1}=(u\ |\ \Box_{b,1}f)_{m_1},\ \ u\in\mathscr D'(X),\ f\in C^\infty_0(X).
\end{align*}
This is well-defined, since if $f$ is a test function on $X$, then so is $\Box_{b,1} f$. One can show
$$ 
\Box_{b,1}u=G^2_p\Box_bu,\ \ \forall u\in\mathscr D'(X).
$$ 
Thus solving $\Box_b$ is essentially the same as solving for $\Box_{b,1}$, and it is the latter that forms the heart of our paper.

The key here is then three-fold, as is represented by the next three theorems. 
First we will show that 
\begin{thm} \label{thm2}
$\Box_{b,1} \colon {\rm Dom\,} \Box_{b,1}\subset L^2(m_1)\To L^2(m_1)$ has closed range in $L^2(m_1)$. 
\end{thm} 
Once this is shown, we have the following $L^2$ decomposition:
\[\Box_{b,1}N+\Pi=I\ \ \mbox{on}\ \ L^2(m_1).\]
Here $N:L^2(m_1)\To{\rm Dom\,}\Box_{b,1}$ is the partial inverse of $\Box_{b,1}$ and $\Pi:L^2(m_1)\To({\rm Ran\,}\Box_{b,1})^\bot$ is the orthogonal projection onto $({\rm Ran\,}\Box_{b,1})^\bot$. 
We now need:
\begin{thm} \label{thm3}
For every $0<\delta<2$, $\Pi$ and $N$ can be extended continuously to
\[ 
\begin{split}
&\Pi:\mathcal{E}(\hat\rho^{-2+\delta})\To\mathcal{E}(\hat\rho^{-2+\delta}),\\
&N:\mathcal{E}(\hat\rho^{-2+\delta})\To\mathcal{E}(\hat\rho^{\delta}).
\end{split}
\] 
\end{thm}
It then follows that
$$ 
\Box_{b,1}N+\Pi=I\ \ \mbox{on}\ \ \mathcal{E}(\hat\rho^{-2+\delta}),
$$ 
for every $0<\delta<2$. Now, let $F$ be as in Theorem~\ref{thm1}. Put $$f:=G^2_p F.$$ Then $f \in\mathcal{E}(\hat\rho^{-1+\delta})$, for every $0<\delta<1$. From Theorem~\ref{thm3}, we know that $\Pi f$ is well-defined; we will show that
\begin{thm} \label{thm4}
$$\Pi f = 0.$$
\end{thm} 
It then follows that $u:=Nf$ satisfies $$u \in\mathcal{E}(\hat\rho^{1+\delta}), \quad \Box_{b,1} u = f = G_p^2 F.$$ From the relation between $\Box_{b,1}$ and $\Box_b$, we obtain the desired conclusion in Theorem~\ref{thm1}.

\subsection{Outline of proofs} \label{subsect:outlineofproof}

To prove the theorems in the previous subsection, we need to introduce two other Kohn Laplacians, which we denote by $\hat\Box_b$ and $\Td\Box_b$, as follows.

First, $\hat\Box_b$ is the natural Kohn Laplacian on $(\hat X, T^{1,0} \hat X, \hat \theta)$. There we have the natural measure
$$
\hat m := \hat \theta \wedge d \hat \theta.
$$
One can then define $L^2(\hat m)$ to be the space of $L^2$ functions on $\hat X$ with respect to the inner product
$$
(f|g)_{\hat m} := \int_{\hat X} f \overline g \, \hat m,
$$
and define $L^2_{(0,1)}(\hat m,\hat\theta)$ to be the space of $L^2$ $(0,1)$ forms on $\hat X$ with respect to the inner product
$$
(\alpha|\beta)_{\hat m,\hat\theta} := \int_{\hat X} \langle \alpha | \beta \rangle_{\hat\theta}  \, \hat m.
$$
We extend the tangential Cauchy-Riemann operator so that 
$$
\text{Dom} \hat \ddbar_{b}  := \{ u \in L^2(\hat m) \colon \text{the distributional $\ddbar_b$ of $u$ is in $L^2_{(0,1)}(\hat m,\hat\theta)$}\},
$$
and define
$$
\hat \ddbar_{b}  u := \text{the distributional $\ddbar_b$ of $u$}
$$
if $u \in \text{Dom} \hat \ddbar_{b}$.
Then 
$$
\hat \ddbar_{b} \colon {\rm Dom\,} \hat \ddbar_{b}\subset L^2(\hat m)\To L^2_{(0,1)}(\hat m,\hat\theta)
$$
is a densely defined closed linear operator. Let 
$$
\hat \ddbar_{b}^* \colon {\rm Dom\,} \hat \ddbar_{b}^* \subset L^2_{(0,1)}(\hat m,\hat\theta) \To L^2(\hat m)
$$
be its adjoint. Let $\hat \Box_{b}$ denote the Gaffney extension of the Kohn Laplacian given by 
$$ 
{\rm Dom\,} \hat \Box_{b}
=\{ s\in L^2(\hat m)\colon s\in{\rm Dom\,}\hat \ddbar_{b},\ \hat \ddbar_{b}s\in{\rm Dom\,} \hat \ddbar_{b}^* \},
$$
$$\hat\Box_{b}s=\hat \ddbar_{b}^* \hat \ddbar_{b}s \quad \text{for $s\in {\rm Dom\,} \hat\Box_{b}$}.$$
It is then a positive self-adjoint operator on $L^2(\hat m)$ (see e.g. \cite[Prop.\,3.1.2]{MM07}). The analysis of this $\hat\Box_b$ is very well-understood; see work of Kohn \cite{Koh85}, \cite{Koh86}, Boas-Shaw \cite{BoSh}, Christ \cite{Ch88I}, \cite{Ch88II} and Fefferman-Kohn \cite{FeKo88} in the CR embeddable case, and work of Kohn-Rossi \cite{KoRo}, Folland-Stein \cite{FoSt}, Rothschild-Stein \cite{RoSt}, Greiner-Stein \cite{GrSt}, Nagel-Stein \cite{NaSt}, Fefferman \cite{F}, Boutet de Monvel-Sjostrand \cite{BouSj76}, Nagel-Stein-Wainger \cite{NSW}, Nagel-Rosay-Stein-Wainger \cite{NRSW1}, \cite{NRSW2} and Machedon \cite{Ma1}, \cite{Ma2} for some earlier work or related results. On the other hand, it is not very straightforward to reduce the analysis of $\Box_{b,1}$ to the analysis of this $\hat\Box_b$; we go through an intermediate Kohn Laplacian, which we denote by $\tilde \Box_b$. 

To introduce $\tilde \Box_b$, first we need to construct a special CR function $\psi$ on $X$, such that $\ddbar_b \psi = 0$ on $\hat X$, $\psi \ne 0$ on $X$, and near $p$, we have
$$
\psi(z,t) = 2\pi (|z|^2 + it) + \text{error},
$$ 
where the error vanishes like $\hat\rho^4$ near $p$.
(The precise construction is given in Section~\ref{sect:CR}.) One can then define the following volume form on $\hat X$:
$$
\tilde m := G_p^2 |\psi|^2 \hat \theta \wedge d \hat \theta.
$$
Note that $\hat X$ has finite volume with respect to this volume form, since $G_p^2 |\psi|^2$ is bounded near $p$. However, this volume form does not have a smooth density against $\hat{m}$; this is one of the biggest sources of difficulties in what we do below. The key turns out to be the following: the asymptotics (\ref{eq:Gpexpand}) we assumed of $G_p$ allows us to obtain some crucial asymptotics of the density of $\tilde m$ against $\hat m$ near $p$. This in turn implies a crucial relation between the $\Td \Box_b$ we will introduce, and the $\hat \Box_b$ we defined above (see (\ref{eq:Boxbhattotilde2}) below).

Now let $L^2(\Td m)$ be the space of $L^2$ functions on $\hat X$ with respect to the inner product
$$
(f|g)_{\tilde m } := \int_{\hat X} f \overline g \, \tilde m ,
$$
and define $L^2_{(0,1)}(\tilde m ,\hat\theta)$ to be the space of $L^2$ $(0,1)$ forms on $\hat X$ with respect to the inner product
$$
(\alpha|\beta)_{\tilde m,\hat\theta} := \int_{\hat X} \langle \alpha | \beta \rangle_{\hat\theta}  \, \tilde m.
$$
We extend the tangential Cauchy-Riemann operator so that 
$$
\text{Dom} \Td \ddbar_{b}  := \{ u \in L^2(\Td m) \colon \text{the distributional $\ddbar_b$ of $u$ is in $L^2_{(0,1)}(\Td m,\hat\theta)$}\},
$$
and define
$$
\Td \ddbar_{b}  u := \text{the distributional $\ddbar_b$ of $u$}
$$
if $u \in \text{Dom} \Td \ddbar_{b}$.
Then 
$$
\Td \ddbar_{b} \colon {\rm Dom\,} \Td \ddbar_{b}\subset L^2(\Td m)\To L^2_{(0,1)}(\Td m,\hat\theta)
$$
is a densely defined closed linear operator. Let 
$$
\tilde  \ddbar_{b}^* \colon {\rm Dom\,} \tilde  \ddbar_{b}^* \subset L^2_{(0,1)}(\tilde  m,\hat\theta) \To L^2(\tilde  m)
$$
be its adjoint. Let $\tilde \Box_{b}$ denote the Gaffney extension of the Kohn Laplacian given by 
$$ 
{\rm Dom\,} \tilde  \Box_{b}
=\{ s\in L^2(\tilde  m)\colon s\in{\rm Dom\,}\tilde \ddbar_{b},\ \tilde  \ddbar_{b}s\in{\rm Dom\,} \tilde  \ddbar_{b}^* \},
$$
$$\tilde \Box_{b}s=\tilde \ddbar_{b}^* \tilde  \ddbar_{b}s \quad \text{for $s\in {\rm Dom\,} \tilde \Box_{b}$}.$$ It is then a positive self-adjoint operator on $L^2(\tilde m)$. The analysis of $\tilde \Box_b$ is not as well-understood, since this Kohn Laplacian (in particular, the operator $\tilde \ddbar_b^*$) is defined with respect to a non-smooth measure $\tilde{m}$. Nonetheless, it is this Kohn Laplacian that can be related to our operator of interest, namely $\Box_{b,1}$, in a simple manner. We will prove that since $\Td m = |\psi|^2 m_1$, 
\begin{eqnarray}
\mbox{$u\in{\rm Dom\,}\Box_{b,1}$ if and only if $\frac{u}{\psi}\in{\rm Dom\,}\Td\Box_b$}, \label{eq:Boxb1totildeBoxb1} \\
\Box_{b,1}u=\psi\Td\Box_b(\psi^{-1} u),\ \ \forall u\in{\rm Dom\,}\Box_{b,1}. \label{eq:Boxb1totildeBoxb2}
\end{eqnarray}
Thus we can understand the solutions of $\Box_{b,1}$, once we understand the solutions of $\tilde{\Box}_b$. In order to carry out the latter, we relate $\tilde \Box_b$ to $\hat \Box_b$: we will show that 
\begin{equation} \label{eq:Boxbhattotilde1}
{\rm Dom\,}\tilde \Box_{b} = {\rm Dom\,}\hat \Box_b,
\end{equation}
and there exists some $g \in \mathcal{E}(\hat \rho^1, T^{0,1}\hat X)$ (possibly non-smooth at $p$) such that 
\begin{equation} \label{eq:Boxbhattotilde2}
\Td \Box_b u= \hat \Box_b u + g \hat \ddbar_b u ,\ \ \forall u\in{\rm Dom\,}\tilde \Box_{b}.
\end{equation}
(Here $g\hat \ddbar_b u$ is the pointwise pairing of the $(0,1)$ vector $g$ with the $(0,1)$ form $\hat \ddbar_b u$; see the discussion in Section~\ref{subsect:def} for the precise meaning of $\mathcal{E}(\hat \rho^1, T^{0,1}\hat X)$.)

We can now outline the proofs of Theorems~\ref{thm2}, \ref{thm3} and \ref{thm4}.

First, from (\ref{eq:Boxb1totildeBoxb1}) and (\ref{eq:Boxb1totildeBoxb2}), it is clear that $\Box_{b,1}$ has closed range in $L^2(m_1)$, if and only if $\Td \Box_b$ has closed range in $L^2(\tilde m)$. On the other hand, one can check that $\tilde \ddbar_b \colon {\rm Dom\,}\Td \ddbar_{b}\subset L^2(\tilde m) \to L^2_{(0,1)} (\tilde  m, \hat \theta)$ is the identical as an operator to $\hat \ddbar_b \colon {\rm Dom\,}\hat\ddbar_{b}\subset L^2(\hat m) \to L^2_{(0,1)} (\hat m, \hat \theta)$. The latter is known to have closed range in $L^2(\hat m)$ by the CR embeddability of $\hat X$; see \cite{Koh86}. Hence the same holds for the former, and it follows that $\tilde \Box_b$ has closed range in $L^2(\tilde m)$. This proves Theorem~\ref{thm2}.

Now from the above argument, we see that not only $\tilde \Box_b$ has closed range in $L^2(\tilde m)$, but also $\hat \Box_b$ has closed range in $L^2(\hat m)$. Thus there exist partial inverses
$$
\tilde N \colon L^2(\tilde m) \to \text{Dom} (\tilde \Box_b) \subset L^2(\Td m),
$$
and 
$$
\hat N \colon L^2(\hat m) \to \text{Dom} (\hat \Box_b) \subset L^2(\hat m)
$$
to $\tilde \Box_b$ and $\hat \Box_b$ respectively, so that if 
$$
\tilde \Pi \colon L^2(\tilde m) \to L^2(\tilde m)
$$
is the orthogonal projection onto the kernel of $\tilde \Box_b$ in $L^2(\tilde m)$, and 
$$
\hat \Pi \colon L^2(\hat m) \to L^2(\hat m)
$$
is the orthogonal projection onto the kernel of $\hat \Box_b$ in $L^2(\hat m)$, then
$$
\Td \Box_b \Td N + \Td \Pi = I \quad \text{on $L^2(\Td m)$},
$$
and
$$
\hat \Box_b \hat N + \hat \Pi = I \quad \text{on $L^2(\hat m)$}.
$$
From the relation 
$$
m_1 = |\psi|^{-2} \tilde m
$$ 
and (\ref{eq:Boxb1totildeBoxb2}),
it is easy to see that 
\begin{equation} \label{eq:soltoTdsol}
\Pi = \psi \Td \Pi \psi^{-1} \quad \text{and} \quad N = \psi \Td N \psi^{-1},
\end{equation}
at least when applied to functions in $L^2(m_1)$; thus to prove Theorem~\ref{thm3}, it suffices to prove instead that $\Td \Pi$ and $\Td N$ extend as continuous operators
\begin{equation} \label{eq:Pimapprop}
\Td \Pi \colon \mathcal{E}(\hat \rho^{-4+\delta}) \to \mathcal{E}(\hat \rho^{-4+\delta})
\end{equation}
\begin{equation} \label{eq:Nmapprop}
\Td N \colon \mathcal{E}(\hat \rho^{-4+\delta}) \to \mathcal{E}(\hat \rho^{-2+\delta})
\end{equation}
for every $0 < \delta < 2$. In order to do so, we relate $\Td \Pi$ to $\hat \Pi$, and $\Td N$ to $\hat N$, since $\hat \Pi$ and $\hat N$ are much better understood. We will show, on $L^2(\Td m)$, that
\begin{equation} \label{eq:TdPitohatPi}
\Td \Pi (I + \hat R) = \hat \Pi
\end{equation}
\begin{equation} \label{eq:TdNtohatN}
\Td N (I + \hat R) = (I - \Td \Pi) \hat N,
\end{equation}
where
$$
\hat R := g \hat \ddbar_b \hat N \colon L^2(\Td m) \to L^2(\Td m)
$$
is a continuous linear operator; in fact, these identities are almost immediate from (\ref{eq:Boxbhattotilde1}) and (\ref{eq:Boxbhattotilde2}). Furthermore, one can show that $\hat N$ and $\hat \Pi$ extend as continuous operators 
\begin{equation} \label{eq:hPimapprop}
\hat \Pi \colon \mathcal{E}(\hat \rho^{-4+\delta}) \to \mathcal{E}(\hat \rho^{-4+\delta})
\end{equation}
\begin{equation} \label{eq:hNmapprop}
\hat N \colon \mathcal{E}(\hat \rho^{-4+\delta}) \to \mathcal{E}(\hat \rho^{-2+\delta})
\end{equation}
for every $0 < \delta < 2$; 
thus if one can show that $(I+\hat R)$ is invertible on $L^2(\Td m)$, and that the inverse extends to a continuous operator $\mathcal{E}(\hat \rho^{-4+\delta}) \to \mathcal{E}(\hat \rho^{-4+\delta})$, then from (\ref{eq:TdPitohatPi}), at least one can conclude the assertion about $\Td \Pi$ in Theorem~\ref{thm3}. It turns out that it is unclear whether or not the latter can be done; so we choose to proceed differently, by some bootstrap argument. It is this argument that we explain below.

First, so far $\Td \Pi$ and $\Td N$ are defined only on $L^2(\Td m)$. Since $\mathcal{E}(\hat \rho^{-4+\delta})$ is not a subset of $L^2(\Td m)$ when $0 < \delta < 2$, we need to first extend $\Td \Pi$ and $\Td N$ to $\mathcal{E}(\hat \rho^{-4+\delta})$,  $0 < \delta < 2$. This is done by rewriting (\ref{eq:TdPitohatPi}) and (\ref{eq:TdNtohatN}) as
\begin{equation} \label{eq:Tdtohatdist}
\Td \Pi = \hat \Pi - \Td \Pi \hat R , \quad \Td N = (I-\Td \Pi) \hat N - \Td N \hat R.
\end{equation}
Note that $\hat R$ extends to a continuous operator
$$
\hat R \colon \mathcal{E}(\hat \rho^{-4+\delta}) \to \mathcal{E}(\hat \rho^{-2+\delta}) \subset L^2(\Td m)
$$
for every $0 < \delta < 2$, since $\hat R = g \hat \ddbar_b \hat N$, and $\hat N$ satisfies the analogous property. Thus the second term on the right hand sides of the equations in (\ref{eq:Tdtohatdist}) map $\mathcal{E}(\hat \rho^{-4+\delta})$ continuously to $L^2(\Td m)$. It follows that the domains of definition of $\Td \Pi$ and $\Td N$ can be extended to $\mathcal{E}(\hat \rho^{-4+\delta})$, $0 < \delta < 2$.

To proceed further, let's write $\hat{\Pi}^{*,\Td m}$, $\hat{N}^{*,\Td m}$ and $\hat R^{*,\Td m}$ for the adjoints of $\hat \Pi$, $\hat N$ and $\hat R$ with respect to the inner product of $L^2(\Td m)$. Since $\Td \Pi$ and $\Td N$ are self-adjoint operators on $L^2(\Td m)$, we have, by (\ref{eq:TdPitohatPi}) and (\ref{eq:TdNtohatN}), that
\begin{equation} \label{eq:TdtohatPiadjoint}
(I + \hat{R}^{*,\Td m}) \Td \Pi = \hat \Pi^{*,\Td m}
\end{equation}
\begin{equation} \label{eq:TdtohatNadjoint}
(I + \hat{R}^{*,\Td m}) \Td N = \hat N^{*,\Td m} (I - \Td \Pi)
\end{equation}
on $L^2(\Td m)$. Now we need to understand some mapping properties of $\hat \Pi^{*,\Td m}$, $\hat N^{*,\Td m}$ and $\hat R^{*,\Td m}$; to do so, we note that
$$
\hat \Pi^{*,\Td m} = \frac{\hat m}{\Td m} \hat \Pi \frac{\Td m}{\hat m}, 
$$
$$
\hat N^{*,\Td m} = \frac{\hat m}{\Td m} \hat N \frac{\Td m}{\hat m}, 
$$
$$
\hat R^{*,\Td m} = \frac{\hat m}{\Td m} \hat N \hat \ddbar_b^* (g^{*} \frac{\Td m}{\hat m}), 
$$
where $\Td m / \hat m := G_p^2 |\psi|^2$ is the density of $\Td m$ with respect to $\hat m$, and similarly $\hat m / \Td m := G_p^{-2} |\psi|^{-2}$. Here $g^*$ is the $(0,1)$ form dual to $g$. Note that $\Td m / \hat m$, $\hat m / \Td m \in \mathcal{E}(\hat \rho^0)$. Hence one can show that
\begin{equation} \label{eq:Pi*}
\hat \Pi^{*,\Td m} \colon \mathcal{E}(\hat \rho^{-4+\delta}) \to \mathcal{E}(\hat \rho^{-4+\delta})
\end{equation}
\begin{equation} \label{eq:N*}
\hat N^{*,\Td m} \colon \mathcal{E}(\hat \rho^{-4+\delta}) \to \mathcal{E}(\hat \rho^{-2+\delta})
\end{equation}
\begin{equation} \label{eq:R*}
\hat R^{*,\Td m} \colon \mathcal{E}(\hat \rho^{-4+\delta}) \to \mathcal{E}(\hat \rho^{-2+\delta})
\end{equation}
for every $0 < \delta < 2$; these are easy consequences of the analogous properties of $\hat \Pi$, $\hat N$ and $\hat R$. It then follows that (\ref{eq:TdtohatPiadjoint}) and (\ref{eq:TdtohatNadjoint}) continue to hold on $\mathcal{E}(\hat \rho^{-4+\delta})$ for all $0 < \delta < 2$.

Now we return to (\ref{eq:TdtohatPiadjoint}) and (\ref{eq:TdtohatNadjoint}). The problem facing us there is that we do not know whether $I + \hat R^{*,\Td m}$ is invertible on $\mathcal{E}(\hat \rho^{-4+\delta})$; if it is, then we can conclude, say from (\ref{eq:TdtohatPiadjoint}), that at least $\Td \Pi$ satisfies the conclusion of Theorem~\ref{thm3}. In order to get around this problem, we have to proceed differently; the trick here is to introduce a suitable cut-off function, as follows. 

Let $\chi$ be a smooth function on $\hat X$, such that $\chi$ is identically 1 in a neighborhood of $p$, and vanishes outside a small neighborhood of $p$. Then by (\ref{eq:TdtohatPiadjoint}) and (\ref{eq:TdtohatNadjoint}), we have
\begin{equation} \label{eq:TdtohatPiadjointchi}
(I + \hat{R}^{*,\Td m} \chi) \Td \Pi = \hat \Pi^{*,\Td m} - \hat R^{*,\Td m} (1-\chi) \Td \Pi,
\end{equation}
\begin{equation} \label{eq:TdtohatNadjointchi}
(I + \hat{R}^{*,\Td m} \chi) \Td N = \hat N^{*,\Td m} (I - \Td \Pi) -  \hat R^{*,\Td m} (1-\chi) \Td N
\end{equation}
on $\mathcal{E}(\hat \rho^{-4+\delta})$ for all $0 < \delta < 2$. The upshot here is the following: if the support of $\chi$ is sufficiently small, then $(I + \hat{R}^{*,\Td m} \chi)$ is invertible on $L^2(\Td m)$, and extends to a linear map 
\begin{equation} \label{eq:I+hatRchi}
(I + \hat{R}^{*,\Td m} \chi)^{-1} \colon \mathcal{E}(\hat \rho^{-4+\delta}) \to \mathcal{E}(\hat \rho^{-4+\delta})
\end{equation}
for every $0 < \delta < 4$. Roughly speaking, this is possible, because
$$I + \hat{R}^{*,\Td m} \chi = I + \frac{\hat m}{\Td m}  \hat{N} \hat \ddbar_b^* (\chi g^* \frac{\Td m}{\hat m}),$$ and because
$$
\chi g^* \in \mathcal{E}(\hat \rho^1, \Lambda^{0,1} T^*\hat X)
$$
has compact support in a sufficiently small neighborhood of $p$. (In particular, $\chi g^*$ is small.)
Furthermore, for any $0 < \delta < 2$, one can show that
\begin{equation} \label{eq:TdPipseudolocal}
(1-\chi) \Td \Pi \colon \mathcal{E}(\hat \rho^{-4+\delta}) \to C^{\infty}_0(X)
\end{equation}
\begin{equation} \label{eq:TdNpseudolocal}
(1-\chi) \Td N \colon \mathcal{E}(\hat \rho^{-4+\delta}) \to C^{\infty}_0(X)
\end{equation}
where $C^{\infty}_0(X)$ is the space of all smooth functions on $X$ that has compact support in $X$. These can be used to control the last term on the right hand side of (\ref{eq:TdtohatPiadjointchi}) and (\ref{eq:TdtohatNadjointchi}). By (\ref{eq:Pi*}) and (\ref{eq:R*}), one then concludes that the right hand side of (\ref{eq:TdtohatPiadjointchi}) maps $\mathcal{E}(\hat \rho^{-4+\delta})$ into itself for $0 < \delta < 2$; thus (\ref{eq:I+hatRchi}) shows that (\ref{eq:Pimapprop}) holds as desired. This in turn controls the first term of (\ref{eq:TdtohatNadjointchi}); by (\ref{eq:N*}), (\ref{eq:R*}) and (\ref{eq:TdNpseudolocal}), one concludes that the right hand side of (\ref{eq:TdtohatNadjointchi}) maps $\mathcal{E}(\hat \rho^{-4+\delta})$ into $\mathcal{E}(\hat \rho^{-2+\delta})$ for $0 < \delta < 2$. Finally, another application of (\ref{eq:I+hatRchi}) shows that $\Td N$ satisfies (\ref{eq:Nmapprop}) as desired. Thus Theorem~\ref{thm3} is established, modulo (\ref{eq:I+hatRchi}), (\ref{eq:TdPipseudolocal}) and (\ref{eq:TdNpseudolocal}).

It may help to reiterate here the reason for the introduction of the cut-off $\chi$: that was introduced so that one can invert $I+\hat R^{*,\Td m} \chi$. In fact, since the coefficient of $g^* \in \mathcal{E}(\hat \rho^1,\Lambda^{0,1}T^*\hat X)$, $\chi g^*$ has sufficiently small $L^{\infty}$ norm, if the support of $\chi$ is chosen sufficiently small. As a result, one could make $\|\hat{R}^{*,\Td m} \chi \|_{L^2(\Td m) \to L^2(\Td m)} \leq 1/2$, by controlling the support of $\chi$. This allows one to invert $I+\hat{R}^{*,\Td m} \chi$ on $L^2(\Td m)$ via a Neumann series. We note in passing that it is only because we have taken the adjoint of $\hat R$ that the product $\chi g^*$ appears in the expression for $I+\hat{R}^{*,\Td m} \chi$; that is essentially why we want to take the adjoints of (\ref{eq:TdPitohatPi}) and (\ref{eq:TdNtohatN}). One can then proceed to extend $(I + \hat{R}^{*,\Td m} \chi)^{-1}$ so that it satisfies (\ref{eq:I+hatRchi}); the precise detail is rather involved, and we leave this until Section~\ref{sect:aux}.

It then remains to prove (\ref{eq:TdPipseudolocal}) and (\ref{eq:TdNpseudolocal}). To do so, we need yet to introduce yet another Kohn Laplacian, namely $\hat \Box_{b,\varepsilon}$. Let $\eta \in C^{\infty}_0(\mathbb{R}^3)$ be a non-negative function such that $\eta(z,t) = 1$ if $\hat \rho(z,t) \leq 1/2$, and $\eta(z,t) = 0$ if $\hat \rho(z,t) \geq 1$. For $0 < \varepsilon < 1$, let $\eta_{\varepsilon}(z,t) = \eta(\varepsilon^{-1} z, \varepsilon^{-2}t)$, and let
$$
\hat m_{\varepsilon} := \eta_{\varepsilon} \hat m + (1-\eta_{\varepsilon}) \Td m.
$$
This volume form has a smooth density against $\hat m$, and the volume of $\hat X$ with respect to this volume form is finite. So if we extend the Cauchy-Riemann operator such that 
$$
\text{Dom} \hat \ddbar_{b,\varepsilon}  := \{ u \in L^2(\hat m_{\varepsilon}) \colon \text{the distributional $\ddbar_b$ of $u$ is in $L^2_{(0,1)}(\hat m_{\varepsilon},\hat\theta)$}\},
$$
and define
$$
\hat \ddbar_{b,\varepsilon}  u := \text{the distributional $\ddbar_b$ of $u$}
$$
if $u \in \text{Dom} \hat \ddbar_{b,\varepsilon}$, then 
$$
\hat \ddbar_{b,\varepsilon} \colon {\rm Dom\,} \hat \ddbar_{b,\varepsilon}\subset L^2(\hat m_\varepsilon)\To L^2_{(0,1)}(\hat m_\varepsilon,\hat\theta)
$$
is a densely defined closed linear operator. Let
$$
\hat \ddbar_{b,\varepsilon}^* \colon {\rm Dom\,} \hat \ddbar_{b,\varepsilon}^* \subset L^2_{(0,1)}(\hat m_{\varepsilon},\hat\theta) \To L^2(\hat m_{\varepsilon})
$$
be its adjoint. Let $\hat \Box_{b,\varepsilon}$ be the Gaffney extension of the Kohn Laplacian $\hat \ddbar_{b,\varepsilon}^* \hat \ddbar_{b,\varepsilon}$. Then $\hat \Box_{b,\varepsilon}$ is almost as well-behaved as $\hat \Box_b$. In particular, $\hat \Box_{b,\varepsilon} \colon {\rm Dom\,} \hat \Box_{b,\varepsilon} \subset L^2(\hat m_{\varepsilon}) \to L^2(\hat m_{\varepsilon}) $ has closed range in $L^2(\hat m_{\varepsilon})$, and if $\hat \Pi_{\varepsilon}$ denotes the orthogonal projection of $L^2(\hat m_{\varepsilon})$ onto the kernel of $\hat \Box_{b,\varepsilon}$, and $\hat N_{\varepsilon} \colon L^2(\hat m_{\varepsilon}) \to L^2(\hat m_{\varepsilon})$ is the partial inverse of $\hat \Box_{b,\varepsilon}$, then 
$$
\hat \Box_{b,\varepsilon} \hat N_{\varepsilon} + \hat \Pi_{\varepsilon} = I.
$$
Furthermore, $${\rm Dom\,}\tilde \Box_{b} = {\rm Dom\,}\hat \Box_{b,\varepsilon},$$ and there will exist some $g_{\varepsilon} \in \mathcal{E}(\hat \rho^1, T^{0,1} \hat X)$ (possibly non-smooth near $p$) such that 
\begin{equation} \label{eq:Boxbhattotildeep}
\Td \Box_b u= \hat \Box_{b,\varepsilon} u + g_{\varepsilon} \hat \ddbar_b u ,\ \ \forall u\in{\rm Dom\,}\tilde \Box_{b}.
\end{equation} 
The upshot here is that $g_{\varepsilon}$ will be compactly supported in the support of $\eta_{\varepsilon}$, whereas our previous $g$ may not be compactly supported near $p$. One can repeat the proof of (\ref{eq:TdtohatPiadjoint}) and (\ref{eq:TdtohatNadjoint}), and show that 
$$
(I + \hat{R}_{\varepsilon}^{*,\Td m}) \Td \Pi = \hat \Pi_{\varepsilon}^{*,\Td m}
$$
$$
(I + \hat{R}_{\varepsilon}^{*,\Td m}) \Td N = \hat N_{\varepsilon}^{*,\Td m} (I - \Td \Pi)
$$
on $L^2(\Td m)$, where 
$$
\hat{R}_{\varepsilon} := g_{\varepsilon} \hat \ddbar_{b,\varepsilon} \hat N_{\varepsilon}.
$$
It follows that
\begin{equation} \label{eq:1-chiTdPi}
(1-\chi) \Td \Pi = (1-\chi) \hat \Pi_{\varepsilon}^{*,\Td m} - (1-\chi) \hat{R}_{\varepsilon}^{*,\Td m} \Td \Pi,
\end{equation}
\begin{equation} \label{eq:1-chiTdN}
(1-\chi) \Td N = (1-\chi) \hat N_{\varepsilon}^{*,\Td m} (I - \Td \Pi) - (1-\chi) \hat{R}_{\varepsilon}^{*,\Td m} \Td N.
\end{equation}
Now
$$
(1-\chi) \hat{R}_{\varepsilon}^{*,\Td m} = \frac{\hat m_{\varepsilon}}{\Td m} (1-\chi) \hat N_{\varepsilon} \hat \ddbar_{b,\varepsilon}^* g_{\varepsilon}^* \frac{\Td m}{\hat m_{\varepsilon}},
$$
and if $\varepsilon$ is chosen sufficiently small (so that the support of $g_{\varepsilon}$ is disjoint from that of $1-\chi$), then $(1-\chi) \hat N_{\varepsilon} \hat \ddbar_{b,\varepsilon}^* g_{\varepsilon}^*$ is an infinitely smoothing pseudodifferential operator, by pseudolocality of $\hat N_{\varepsilon}$. Hence the last term of (\ref{eq:1-chiTdPi}), and also the last term of (\ref{eq:1-chiTdN}), map $\mathcal{E}(\hat \rho^{-4+\delta})$ into $C^{\infty}_0(X)$. Since for every $0 < \varepsilon < 1$ and every $0 < \delta < 2$, $$\hat \Pi_{\varepsilon}^{*,m} \colon \mathcal{E}(\hat \rho^{-4+\delta}) \to \mathcal{E}(\hat \rho^{-4+\delta}),$$ and $$\hat N_{\varepsilon}^{*,m} \colon \mathcal{E}(\hat \rho^{-4+\delta}) \to \mathcal{E}(\hat \rho^{-2+\delta}),$$ it follows that $(1-\chi)\Td \Pi$ and $(1-\chi)\Td N$ satisfies (\ref{eq:TdPipseudolocal}) and (\ref{eq:TdNpseudolocal}), and we are done with the proof of Theorem~\ref{thm3}.

Finally, to prove Theorem~\ref{thm4}, the key is the following fact, which we prove in Lemma~\ref{lem:PiErho0}: if $\alpha$ is a $(0,1)$ form with coefficients in $\mathcal{E}(\hat \rho^0)$, then $\ddbar_{b,1}^* \alpha \in \mathcal{E}(\hat \rho^{-1})$ satisfies $\Pi \ddbar_{b,1}^* \alpha = 0$. To compute $\Pi f = \Pi \Box_{b,1} \Td \beta = \Pi \ddbar_{b,1}^* (\ddbar_{b,1} \Td \beta)$, we will then decompose $\ddbar_{b,1} \Td \beta$ into a sum
$$
\ddbar_{b,1} \Td \beta = \alpha_0 + E,
$$
where the main term $\alpha_0$ has coefficients in $\mathcal{E}(\hat \rho^{-2})$, and the error $E$ has coefficients in $\mathcal{E}(\hat \rho^0)$. Then by Lemma~\ref{lem:PiErho0}, $\Pi \ddbar_{b,1}^* E = 0$. Furthermore, we will construct by hand an explicit family of $(0,1)$ forms $\alpha_{\varepsilon}$, with coefficients in $\mathcal{E}(\hat \rho^0)$, such that 
$$
\ddbar_{b,1}^* \alpha_{\varepsilon} \to \ddbar_{b,1}^* \alpha_0 \quad \text{in $\mathcal{E}(\hat \rho^{-1})$ as $\varepsilon \to 0$}.
$$
Thus by continuity of $\Pi$ on $\mathcal{E}(\hat \rho^{-1})$, we have $\Pi \ddbar_{b,1}^* \alpha_0 = \lim_{\varepsilon \to 0} \Pi \ddbar_{b,1}^* \alpha_{\varepsilon} = 0$ as well, the last equality following from Lemma~\ref{lem:PiErho0}. Together we get $\Pi f = 0$, as desired.

\subsection{Definitions and Notations} \label{subsect:def}

We shall now recall some basic definitions, and introduce some basic notations.

A 3-dimensional smooth manifold $X$ is said to be a CR manifold, if there exists a 1-dimensional subbundle $L$ of the complexified tangent bundle $\mathbb{C}TX$ such that $L \cap \overline{L} = \{0\}$; such subbundle $L$ is then denoted as $T^{1,0}X$, and $\overline{L}$ denoted as $T^{0,1}X$. The dual bundles to $T^{1,0}X$ and $T^{0,1}X$ will be denoted by $\Lambda^{1,0}T^*X$ and $\Lambda^{0,1}T^*X$ respectively. A typical example is a 3-dimensional smooth submanifold $X$ of $\mathbb{C}^N$; there one has a natural CR structure induced from $\mathbb{C}^N$, given by the bundle of all $(1,0)$ vectors in $\mathbb{C}^N$ that are tangent to $X$.

A 3-dimensional CR manifold $X$ is said to be strongly pseudoconvex, if at every point on $X$ there exists a local section $Z$ of $T^{1,0}X$ such that $[Z,\overline{Z}]$ is transverse to $T^{1,0}X \oplus T^{0,1}X$. It is said to be CR embeddable in $\mathbb{C}^N$, if there exists a smooth embedding $\Phi \colon X \to \Phi(X) \subset \mathbb{C}^N$, such that $d\Phi(T^{1,0}X)$ agrees with the natural CR structure of $\Phi(X)$ induced from $\mathbb{C}^N$.

We shall write $C^{\infty}(X)$ for the space of smooth functions on $X$, and $\Omega^{0,1}(X)$ for the space of smooth sections of $\Lambda^{0,1}T^*X$. We shall also write $C^{\infty}_0(X)$ and $\Omega_0^{0,1}(X)$ for the subspaces of $C^{\infty}(X)$ and $\Omega^{0,1}(X)$ which consist of elements that have compact support in $X$.

Suppose $X$ is a 3-dimensional CR manifold. If there exists a real contact form $\theta$ (i.e. a global real 1-form $\theta$ with $\theta \wedge d\theta \ne 0$ everywhere) such that $T^{1,0} X \oplus T^{0,1} X$ is given by the kernel of $\theta$, then $(X,T^{1,0}X,\theta)$ is called a pseudohermitian 3-manifold. In that case, $X$ is strongly pseudoconvex, and one can define a Hermitian inner product on $T^{1,0}X$, by
$$
\langle Z_1 | Z_2 \rangle_{\theta} := \frac{1}{2} d\theta(Z_1, i \overline{Z}_2).
$$
This allows one to define various geometric quantities on $X$, like the connection form $\omega^1_1$, and the Tanaka-Webster scalar curvature $R$. One can also define the sublaplacian $\Delta_b$, the conformal sublaplacian $L_b := -4\Delta_b + R$, and the CR Paneitz operator $P$. We refer the reader to say \cite{CMY} for the precise definitions of such.

We note that the above Hermitian inner product on $T^{1,0}X$ induces naturally a Hermitian inner product on $\Lambda^{0,1}T^*X$, which we still denote by $\langle \cdot | \cdot \rangle_{\theta}$. For $\alpha \in \Omega^{0,1}X$, we write $|\alpha|_{\theta}^2 := \langle \alpha|\alpha \rangle_{\theta}$.

In what follows, we will need the Reeb vector field $T$ on a contact manifold $(X,\theta)$, which is the unique vector field such that 
$$
\theta(T) \equiv 1, \quad d\theta(T,\cdot) \equiv 0.
$$

If $\hat \rho$ is a non-negative smooth function defined near a point $p$, then we write $\varepsilon(\hat\rho^k)$ for the set of smooth functions $f\in C^\infty(\hat X)$ such that $|f|\leq C\hat\rho^k$ near $p$ for some $C>0$.

Suppose now $(\hat X,T^{1,0}\hat X,\hat\theta_0)$ a pseudohermitian 3-manifold, and we fix $p \in\hat X$. Then as is known, there exists another contact form  $\hat\theta$ on $\hat X$, which is a multiple of $\hat \theta_0$ by a positive smooth function, so that near $p$, there exists CR normal coordinates $(z,t)$. In other words, the contact form $\hat \theta$ and the coordinates $(z,t)$ are chosen, so that
\begin{enumerate}[(i)]
\item the point $p$ corresponds to $(z,t) = (0,0)$;
\item one can find a local section $\hat Z_1$ of $T^{1,0} \hat X$ near $p$, with $\langle \hat Z_1 | \hat Z_1 \rangle_{\hat \theta} = 1$, such that $\hat Z_1$ admits the following expansion near $p$:
\begin{equation} \label{s1-e5b}
\hat Z_1=\frac{\pr}{\pr z}-i\ol z\frac{\pr}{\pr t}+\varepsilon(\hat\rho^4)\frac{\pr}{\pr z}+\varepsilon(\hat\rho^4)\frac{\pr}{\pr\ol z}+\varepsilon(\hat\rho^5)\frac{\pr}{\pr t};
\end{equation}
\item the Reeb vector field $\hat T$ with respect to $\hat \theta$ admits an expansion
\begin{equation} \label{s1-e5c}
\hat T=\frac{\pr}{\pr t}+\varepsilon(\hat\rho^3)\frac{\pr}{\pr z}+\varepsilon(\hat\rho^3)\frac{\pr}{\pr\ol z}+\varepsilon(\hat\rho^4)\frac{\pr}{\pr t}.
\end{equation}
\end{enumerate} 
Here \[\hat\rho(z,t)=(\abs{z}^4+t^2)^{\frac{1}{4}}\] for $(z,t)$ in a neighborhood of $(0,0)$. For later convenience, from now on we will fix a positive smooth extension of $\hat\rho$ to the whole manifold $X$. We will also write $\hat Z_{\bar 1} := \overline{\hat Z_1}$. Note that in CR normal coordinates we have 
\begin{equation} \label{eq:hthetaexpansion}
\hat \theta = dt - i(z d\overline{z} - \overline{z} dz) + \varepsilon(\hat \rho^5) dz + \varepsilon(\hat \rho^5) d\overline{z} + \varepsilon(\hat \rho^4) dt.
\end{equation} 

Next, for $m \in \mathbb{R}$, we will introduce a Fr\'{e}chet space $\mathcal{E}(\hat\rho^m)$, with which our results are formulated. We pause and introduce some notations first. Let $k\in\mathbb N$. We denote by $\hat \nabla^k_b$ any differential operator of the form $L_1\ldots L_k$, where $L_j\in C^\infty(\hat X,T^{1,0}\hat X\oplus T^{0,1}\hat X)$, $\langle\,L_j\,|\,L_j\,\rangle_{\hat\theta}\leq1$, $j=1,\ldots,k$. Let $\mathcal{O}(\hat\rho^m)=\mathcal{O}^{(0)}(\hat\rho^m)$, $m\in\Real$, denote the set of all $f\in C^\infty(X)$ such that $\abs{f}\leq C\hat\rho^m$ near $p$, for some $C>0$. 
Let $\mathcal{O}^{(1)}(\hat\rho^m)$ denote the set of all functions $f\in\mathcal{O}(\hat\rho^m)$ such that 
$\hat \nabla_b f\in\mathcal{O}(\hat\rho^{m-1})$. Similarly, for $k\in\mathbb N$, $k\geq2$, let $\mathcal{O}^{(k)}(\hat\rho^m)$ denote the set of all functions $f\in\mathcal{O}(\hat\rho^m)$ such that $\hat \nabla_b f\in\mathcal{O}^{(k-1)}(\hat\rho^{m-1})$. Put 
\begin{equation}\label{s1-e6} 
\mathcal{E}(\hat\rho^m)=\bigcap_{k\in\mathbb N\bigcup\set{0}}\mathcal{O}^{(k)}(\hat\rho^m).
\end{equation} 
Let $\Omega\subset\hat X$ be an open set. For $f\in C^\infty(\Omega)$, define \[\norm{f}_{L^\infty(\Omega)}:=\sup_{x\in\Omega}\abs{f(x)}.\]
$\mathcal{E}(\hat\rho^m)$ is a Fr\'{e}chet space with the semi-norms:
\begin{equation}\label{s-smmiI}
u\To\norm{\hat\nabla^k_b(\hat\rho^{-m+k}u)}_{L^\infty(X)},\ \ u\in\mathcal{E}(\hat\rho^m), 
\end{equation}
for $k\in\mathbb N_0$. These semi-norms then define the topology of $\mathcal{E}(\hat\rho^m)$. 

There is a version of this space for smooth vector bundles over $X$. Let $E$ be a smooth vector bundle over $\hat X$ of rank $r$. Let $f_1,\ldots,f_r$ be any local frame in some small neighbourhood $U$ of $p$. For $m\in\Real$, let $\mathcal{E}(\hat\rho^m,E)$ be the set of all $u\in C^\infty(X,E)$ such that $u=u_1f_1+\cdots+u_rf_r$ on $U$, and $\chi u_j\in\mathcal{E}(\hat\rho^m)$ for every $\chi\in C^\infty_0(U)$ and every $j=1,\ldots,r$.

Note that $\mathcal{E}(\hat\rho^{m})\subset\mathcal{E}(\hat\rho^{m'})$ if $m'<m$. We also notice that for every $m\in\Real$, $C^\infty_0(X)$ is dense in $\mathcal{E}(\hat\rho^m)$ for the topology of $\mathcal{E}(\hat\rho^{m'})$, for every $m'<m$. Similarly for $\mathcal{E}(\hat\rho^{m},E)$ for any smooth vector bundle $E$.

Distributions on $\hat X$ will be denoted $\mathcal{D}'(\hat X)$.

Finally, suppose $T \colon \text{Dom}(T) \subset H_1 \to H_2$ is a densely defined closed linear operator between two Hilbert spaces $H_1$ and $H_2$, and suppose the range of $T$ is closed in $H_2$. Then the partial inverse of $T$ is the unique linear operator $S \colon H_2 \to H_1$, such that if $\Pi_1 \colon H_1 \to H_1$ and $\Pi_2 \colon H_2 \to H_2$ are the orthogonal projections onto the kernels of $T$ and $T^*$ respectively, then 
$$
TS + \Pi_2 = I, \quad S\Pi_2 = 0, \quad \text{and} \quad \Pi_1 S = 0.
$$
(Here $T^* \colon \text{Dom}(T^*) \subset H_2 \to H_1$ is the adjoint of $T$, which is also densely defined and closed; and $I$ is the identity operator.) It follows that $S \colon H_2 \to H_1$ is bounded, and  
$$
ST + \Pi_1 = I \quad \text{on the domain of $T$}.
$$

As we saw in Section~\ref{subsect:outlineofproof}, the operators 
$$
\hat \ddbar_b \colon \text{Dom} \hat \ddbar_b \subset L^2(\hat m) \to L^2_{(0,1)}(\hat m),
$$
$$
\hat \Box_b \colon \text{Dom} \hat \Box_b \subset L^2(\hat m) \to L^2(\hat m),
$$
$$
\Td \Box_b \colon \text{Dom} \Td \Box_b \subset L^2(\Td m) \to L^2(\Td m),
$$
and
$$
\Box_{b,1} \colon \text{Dom} \Box_{b,1} \subset L^2(m_1) \to L^2(m_1)
$$
all have closed ranges. (See Section~\ref{sect:closedrange} for more details.) Their partial inverses will be denoted by $\hat K$, $\hat N$, $\Td N$ and $N$ respectively. 

A piece of convention here: recall that $\hat K$ is an operator that takes $(0,1)$ forms to functions. By identifying the space of $(0,1)$ forms locally with functions, we would sometimes like to think of $\hat K$ as a map from functions to functions. To do so rigorously, we proceed as follows. At every point $x \in \hat X$, there exists a non-isotropic ball $B(x,r_x)$ such that $T^{0,1} \hat X$ has a non-zero section on $B(x,r_x)$. The sets $\{B(x,r_x/2) \colon x \in \hat X\}$ covers $\hat X$; one can thus take a finite subcover that covers $\hat X$. Denote this finite subcover by $\{B_1, \dots, B_N\}$, and the dual of a non-zero local section of $T^{0,1} \hat X$ on $2B_i$ by $\hat \omega_i$; here $2B_i$ is the non-isotropic ball that has the same center as $B_i$, but twice the radius. We further normalize $\hat \omega_i$ so that $\langle \hat \omega_i | \hat \omega_i \rangle_{\hat \theta} = 1$ on $2B_i$. One then has the following property: there is some $r_0 > 0$ such that if $B$ is a non-isotropic ball of radius $< r_0$ on $\hat X$ that intersects some of the $B_i$ above, then $\hat \omega_i$ is defined and of norm 1 on $B$. By taking a partition of unity $\sum \eta_i = 1$ subordinate to the open cover $\{B_1, \dots, B_N\}$, we can define maps $\hat K_i$, $1 \leq i \leq N$, that map from functions to functions, by the following formula:
\begin{equation} \label{eq:hatKdefn}
\hat K_i \varphi := \hat K (\eta_i \varphi \hat \omega_i).
\end{equation} 
Then since $\hat K = \sum_i \hat K \eta_i$, and $\eta_i \phi = \eta_i \langle \phi | \hat \omega_i \rangle_{\hat \theta} \hat \omega_i$ for all $(0,1)$ forms $\phi$, we have
$$
\hat K \phi = \sum_i \hat K_i [\langle \phi|\hat \omega_i \rangle_{\hat \theta}].
$$
It will be slightly more convenient to consider properties of $\hat K_i$ instead of $\hat K$ at a number of places below. The result will always be independent (up to constants) of the choices of the cut-offs $\eta_i$, and of the choice of frames $\hat \omega_i$.

The plan of the paper is as follows. In Section~\ref{sect2}, we gather together some properties of the Szeg\"o projection $\hat \Pi$, as well as the partial inverse $\hat N$ of the smooth Kohn Laplacian $\hat \Box_b$. In Section~\ref{sect:aux}, we develop tools to establish the key mapping property (\ref{eq:I+hatRchi}), that involves the weighted space $\mathcal{E}(\hat \rho^{-4+\delta})$.  In Section~\ref{sect:CR}, we construct the  CR function $\psi$ that is crucial for us. Sections~\ref{sect5} and \ref{sect:TdBoxbtohatBoxb} clarifies the relations between the various Kohn Laplacians. Sections~\ref{sect:closedrange} to \ref{sect:app} contain the proofs of Theorems~\ref{thm2} to \ref{thm4}, which implies Theorem~\ref{thm1} and Corollary~\ref{cor1} as we have explained above. Finally, in Appendix 1, we establish some properties of the Green's function of the conformal Laplacian $L_b$, which allows us to apply our results towards the study of the CR positive mass theorem as was laid out in \cite{CMY}. In Appendix 2, we prove a subelliptic estimate for $\hat \Box_b$, which should be known to the experts, but which has not appeared explicitly in literature.

\noindent
{\small\emph{\textbf{Acknowledgements.} The authors would like to express their gratitude to Jih-Hsin Cheng, Andrea Malchiodi and Paul Yang for suggesting to us this beautiful problem and for several useful conversations. The second author would like to thank Kenneth Koenig for his interest in our work, and for some useful discussion.}}

\section{Some properties of the smooth Kohn Laplacian $\hat\Box_b$} \label{sect2}

We collect in this section some results from subelliptic analysis and several complex variables. The key is to introduce a class of non-isotropic smoothing operators on our pseudohermitian manifold $\hat X$, and show that the Szeg\"o projection $\hat \Pi$, as well as the partial inverse $\hat N$ of $\hat \Box_b$, are examples of such; we will deduce, as a result, mapping properties of $\hat \Pi$ and $\hat N$ with respect to the weighted spaces $\mathcal{E}(\hat \rho^{\delta})$. Many of these are known; we refer the reader to Nagel-Stein-Wainger~\cite{NSW}, Kohn~\cite{Koh85}, Christ~\cite{Ch88I},~\cite{Ch88II}, Nagel-Rosay-Stein-Wainger~\cite{NRSW2}, Koenig~\cite{Koe02} and the references therein for further details. 

First, we recall the Carnot-Caratheodory metric on our CR manifold $\hat X$. For $\delta > 0$, let $C(\delta)$ be the class of all absolutely continuous mappings $\varphi \colon [0,1] \to \hat X$ such that for a.e. $t$,
$$
\varphi'(t) = a_1(t) X_1(\varphi(t)) + a_2(t) X_2(\varphi(t)), \quad |a_j(t)| < \delta,\quad j=1,2.
$$ 
Here $X_1$ and $X_2$ are the real and imaginary parts of $\hat Z_1$ respectively. The Carnot-Caratheodory metric on $\hat X$ is then defined by
$$
\vartheta(x,y) = \inf \{ \delta > 0 \colon \text{ there exists } \varphi \in C(\delta) \text{ such that } \varphi(0) = x, \varphi(1) = y \}
$$
for $x, y \in \hat X$. From Theorem 4 of \cite{NSW}, coupled with the representations \eqref{s1-e5b} and \eqref{s1-e5c} of $\hat{Z}_1$ and $\hat{T}$ in CR normal coordinates, it is easy to show that for points $x$ sufficiently close to $p$, we have
$$
\vartheta(x,p) \simeq \hat \rho(x).
$$
(See also Theorem 3.5 and Remark 3.3 of Jean \cite{Jean12}.) We write $B(x,r)$ for the non-isotropic ball $\{y \in \hat X \colon \vartheta(x,y) < r\}$ of radius $r$ centered at $x$.

Next, we proceed to define on $\hat X$ a class of (non-isotropic) smoothing operators of order $j$. For our purposes, it suffices to restrict our attention to the case when $0 \leq j < 4$.

Recall that a function $\phi$ on $\hat X$ is said to be a  \emph{normalized bump function} on a ball $B(x,r)$, if it is smooth with compact support on $B(x,r)$, and satisfies 
\begin{equation}
\norm{\hat{\nabla}_b^k \phi}_{L^{\infty}(B(x,r))} \leq C_k r^{-k}
\end{equation} 
for all $k \geq 0$; here $C_k>0$ are absolute constants independent of $r$. 

Usually we only require the above derivative estimate to be satisfied for all $0 \leq k \leq N$ for some large integer $N$. In that case, we say that $\phi$ is a normalized bump function of order $N$ in $B(x,r)$.

Suppose now $T$ is a continuous linear operator $T \colon C^{\infty}(\hat X) \to C^{\infty}(\hat X)$, and its adjoint $T^*$ (with respect to the inner product of $L^2(\hat m)$) is also a continuous map  $T^* \colon C^{\infty}(\hat X) \to C^{\infty}(\hat X)$. We say that $T$ is a smoothing operator of order $j$,  $0 \leq j < 4$, if 
\begin{enumerate}[(a)]
\item there exists a kernel $T(x,y)$, defined and smooth away from the diagonal in $\hat X \times \hat X$, such that 
\begin{equation} \label{eq:Tkernelrep}
Tf(x) = \int_{\hat X} T(x,y) f(y) \hat m(y)
\end{equation}
for any $f \in C^{\infty}(\hat X)$, and every $x$ not in the support of $f$,
\item the kernel $T(x,y)$ satisfies the following differential inequalities when $x \ne y$:
$$
|(\hat \nabla_b)_x^{\alpha_1} (\hat \nabla_b)_y^{\alpha_2} T(x,y)| \lesssim_{\alpha} \vartheta(x,y)^{-4+j-|\alpha|}, \quad |\alpha| = |\alpha_1| + |\alpha_2|;
$$
\item the operators $T$ and $T^*$ satisfy the following cancellation conditions: if $\phi$ is a normalized bump function in some ball $B(x,r)$, then $$\|\hat \nabla_b^{\alpha} T\phi \|_{L^{\infty}(B(x,r))} \lesssim_{\alpha} r^{j-|\alpha|},$$
and
$$\|\hat \nabla_b^{\alpha} T^*\phi \|_{L^{\infty}(B(x,r))} \lesssim_{\alpha} r^{j-|\alpha|}.$$
\end{enumerate}

It is then clear that $T$ is smoothing of order $j$, if and only if $T^*$ is smoothing of order $j$. We also have the following proposition:

\begin{prop}
If $T$ is a smoothing operator of order 0, then $T$ is bounded on $L^p(\hat m)$ for $1 < p < \infty$.
\end{prop} 

\begin{proof}
The boundedness of $T$ on $L^2(\hat m)$ follows from a version of $T(1)$ theorem. In fact, suppose $f$ is a normalized bump function on a ball $B(x_0,r)$. If $T$ is a smoothing operator of order 0, then by the cancellation condition on $T$,
$$
\|Tf\|_{L^2(B(x_0,2r))} \lesssim r^2,
$$
and by the kernel representation of $T$, when $x \notin B(x_0,2r)$,
$$
|Tf(x)| 
\lesssim \int_{y \in B(x_0,r)} \frac{|f(y)|}{\vartheta(x,y)^4} \hat m(y) 
\lesssim \vartheta(x,x_0)^{-4} r^4.
$$
Hence
$$
\int_{x \notin B(x_0,2r)} |Tf(x)|^2 \hat m(x)
\lesssim \int_{x \notin B(x_0,2r)} \vartheta(x,x_0)^{-8} r^8 \hat m(x) 
\lesssim r^4.
$$
Altogether,
$$
\|Tf\|_{L^2(\hat X)} \lesssim r^2;
$$
similiarly for $\|T^*f\|_{L^2(\hat X)}$. Hence both $T$ and $T^*$ are restrictedly bounded, and by the $T(1)$ theorem (see e.g. Chapter 7 of \cite{Ste93}), $T$ is bounded on $L^2(\hat m)$. By the Calderon-Zygmund theory of singular integrals, it then follows that such operators are bounded on $L^p(\hat m)$, $1 < p < \infty$. (See e.g. Chapter 1 of \cite{Ste93}.) 
\end{proof}

Next we have the following theorem:

\begin{thm} \label{thm2.1}
The Szeg\"o projection $\hat \Pi$, and the partial inverse $\hat N$ of $\hat \Box_b$, are smoothing operators of orders $0$ and $2$ respectively. Furthermore, $\hat K$ is smoothing of order $1$, in the sense that the local representations $\hat K_i$ defined by (\ref{eq:hatKdefn}) are smoothing of order $1$.
\end{thm}

We defer its proof until the end of this section.

We will need two further key facts about this class of smoothing operators:

\begin{thm} \label{thm2.2}
If $T_1$ and $T_2$ are smoothing operators of orders $j_1$ and $j_2$ respectively, with $j_1, j_2 \geq 0$ and $j_1 + j_2 < 4$, then $T_1 \circ T_2$ is a smoothing operator of order $j_1 + j_2$.
\end{thm}

\begin{thm} \label{thm2.3}
If $T$ is a smoothing operator of order $j$, $0 \leq j < 4$, then $T$ extends to a continuous linear map $$T \colon \mathcal{E}(\hat \rho^{-\gamma}) \to \mathcal{E}(\hat \rho^{-\gamma+j}),$$ as long as $j < \gamma < 4$.
\end{thm}

In particular, in proving Theorem~\ref{thm2.1}, it suffices to prove the statements for $\hat \Pi$ and $\hat K$, since the statement for $\hat N$ follows from Theorem~\ref{thm2.2} and the well-known fact that 
$$
\hat N = \hat K \hat K^*.
$$
Also, combining Theorems~\ref{thm2.1} and \ref{thm2.3}, $\hat \Pi$ and $\hat N$ extend to continuous linear maps
$$
\hat \Pi \colon \mathcal{E}(\hat \rho^{-4+\delta}) \to \mathcal{E}(\hat \rho^{-4+\delta}), \quad 0 < \delta < 4,
$$ 
$$
\hat N \colon \mathcal{E}(\hat \rho^{-4+\delta}) \to \mathcal{E}(\hat \rho^{-2+\delta}), \quad 0 < \delta < 2,
$$ 
and (\ref{eq:hPimapprop}), (\ref{eq:hNmapprop}) follows.

\begin{proof}[Proof of Theorem~\ref{thm2.2}]
We will only need the case when $T_1$ is smoothing of order 1, and $T_2$ is smoothing of order 0 or 1. Thus we will focus on these cases. 

Suppose first both $T_1$ and $T_2$ are smoothing of order 1. Then $T:=T_1 \circ T_2$ is continuous on $C^{\infty}(\hat X)$, and so is $T^*$. Furthermore, when $f \in C^{\infty}(\hat X)$, we have
$$
T_j f(x) = \int_{\hat X} T_j(x,y) f(y) \hat m(y), \quad j = 1,2
$$
for all $x$ not in the support of $f$. Using the cancellation conditions, one can show that this integral representation actually holds for all $x \in \hat X$ (not just for all $x$ outside the support of $f$; c.f. Chapter 1.7 of Stein~\cite{Ste93}). This is typical of operators that are smoothing of positive orders. 

As a result, by Fubini's theorem, when $f \in C^{\infty}(\hat X)$, (\ref{eq:Tkernelrep}) holds for all $x \in \hat X$, where
\begin{equation} \label{eq:TkernelrepT1T2}
T(x,y) = \int_{\hat X} T_1(x,z) T_2(z,y) \hat m(z).
\end{equation}
Fix now $x, y \in \hat X$, and let $r = \vartheta(x,y)/4$. We pick  normalized bump functions $\chi_1$, $\chi_2$ in $B(x,r)$ and $B(y,r)$ respectively, such that $\chi_1 \equiv 1$ on $B(x,r/2)$, and $\chi_2 \equiv 1$ on $B(y,r/2)$. Then inserting $1 = \chi_1(z) + \chi_2(z) + (1-\chi_1-\chi_2)(z)$ into the integral defining $T(x,y)$, we have
$$
T(x,y) = T^{(1)}(x,y) + T^{(2)}(x,y) + T^{(3)}(x,y),
$$
and we estimate these one by one.

First, 
$$
T^{(1)}(x,y) = \int_{\hat X} T_1(x,z) \chi_1(z) T_2(z,y) \hat m(z).
$$
We can differentiate under the integral, and obtain
$$
(\hat \nabla_b)_y^{\alpha_2} T^{(1)}(x,y) = (T_1 f_{y}^{(\alpha_2)})(x)
$$
where
$$
f_{y}^{(\alpha_2)}(z) := \chi_1(z) (\hat \nabla_b)_y^{\alpha_2} T_2(z,y)
$$
is $r^{-3-|\alpha_2|}$ times a normalized bump function in $B(x,r)$. Thus by the cancellation condition for $T_1$, we obtain
$$
|(\hat \nabla_b)_x^{\alpha_1} (\hat \nabla_b)_y^{\alpha_2} T^{(1)}(x,y) | \lesssim r^{-3-|\alpha_2|} r^{1-|\alpha_1|} = r^{-2-|\alpha|}.
$$
This proves the desired differential inequalities for $T^{(1)}(x,y)$. A similar argument, using the cancellation conditions for $\ol T_2^*$ instead, shows that
$$
|(\hat \nabla_b)_x^{\alpha_1} (\hat \nabla_b)_y^{\alpha_2} T^{(2)}(x,y) | \lesssim r^{-2-|\alpha|}.
$$
Finally, the integral defining $T^{(3)}(x,y)$ is supported for $z$ outside the balls $B(x,r/2)$ and $B(y,r/2)$. As a result, $\vartheta(x,z) \simeq \vartheta(y,z)$ for $z$ in the support of the integral defining $T^{(3)}(x,y)$. One can now differentiate under the integral, and obtain
$$
|(\hat \nabla_b)_x^{\alpha_1} (\hat \nabla_b)_y^{\alpha_2} T^{(3)}(x,y) | \lesssim \int_{\vartheta(x,z) > r/2} \vartheta(x,z)^{-3-|\alpha_1|} \vartheta(x,z)^{-3-|\alpha_2|} dz 
\lesssim r^{-2-|\alpha|}.
$$
This proves the desired differential inequalities for $T(x,y)$.

Next, to prove the cancellation conditions for $T$, suppose $\phi$ is a normalized bump function in a ball $B(x_0,r)$. Then we let $\chi$ be a normalized bump function supported in $B(x_0,4r)$, that is identically 1 on $B(x_0,3r)$, and write
$$
T\phi = T_1 (\chi T_2 \phi) + T_1 ((1-\chi)T_2 \phi).
$$
Now by cancellation conditions for $T_2$, one sees that $r^{-1} \chi T_2 \phi$ is a normalized bump function on $B(x_0,4r)$. Hence $T_1 (\chi T_2 \phi)$ obeys the desired bound, namely
$$
\|\hat \nabla_b^{\alpha} T_1 (\chi T_2 \phi)\|_{L^{\infty}(B(x_0,r))} \lesssim_{\alpha} r^{2-|\alpha|}.
$$
Furthermore, for $x \in B(x_0,r)$,
\begin{align*}
&\hat \nabla_b^{\alpha} T_1[(1-\chi)T_2 \phi](x) \\
=& \int_{z \notin B(x,3r)} (\hat \nabla_b)_x^{\alpha} T_1(x,z) (1-\chi)(z) T_2 \phi(z) \hat m(z) \\
=& \int_{z \notin B(x,3r)} \int_{y \in B(x,2r)} (\hat \nabla_b)_x^{\alpha} T_1(x,z) (1-\chi)(z) T_2(z,y) \phi(y) \hat m(y) \hat m(z).
\end{align*}
Putting absolute values,
\begin{align*}
&|\hat \nabla_b^{\alpha} T_1[(1-\chi)T_2 \phi](x)| \\
\leq & \int_{z \notin B(x,3r)} \int_{y \in B(x,2r)} \vartheta(x,z)^{-3-|\alpha|} \vartheta(z,y)^{-3} \hat m(y) \hat m(z) \\
\lesssim & \int_{z \notin B(x,3r)} \int_{y \in B(x,2r)} \vartheta(x,z)^{-6-|\alpha|}  \hat m(y) \hat m(z) \\
\lesssim & r^{2-|\alpha|},
\end{align*}
the second to last line following since $\vartheta(x,z) \simeq \vartheta(z,y)$ on the support of the integrals. This provides the desired bound for $\|(\hat \nabla_b)^{\alpha} T\phi\|_{L^{\infty}(B(x,r))}$. A similar argument establishes the bound for $\|(\hat \nabla_b)^{\alpha} T^* \phi\|_{L^{\infty}(B(x,r))}$. This completes our proof when both $T_1$ and $T_2$ are smoothing of order 1.

Finally, suppose $T_1$ is smoothing of order 0, and $T_2$ is smoothing of order 1. Then $T := T_1 \circ T_2$ maps $C^{\infty}(\hat X)$ continuously into itself, and so does $T^* = T_2^* \circ T_1^*$; one can repeat the above argument to show that both $T$ and $T^*$ satisfy the cancellation conditions for an operator of order 1. Thus it remains to compute the kernel of $T$, and to establish differential inequalities for the kernel of $T$, to which we now turn. 

For $x, y \in \hat X$ with $x \ne y$, let $r := \vartheta(x,y)/4$, and
$$
T(x,y):= (T_1 k_y)(x) \quad \text{where $k_y(z):= T_2(z,y)$}.
$$
We first show that $T(x,y)$ is smooth away from the diagonal, and that it satisfies the differential inequalities
$$
| (\hat \nabla_b)_x^{\alpha_1} (\hat \nabla_b)_y^{\alpha_2} T(x,y) | \lesssim r^{-3-|\alpha|}.
$$
To do so, fix $x \ne y$, and let $\chi_1, \chi_2$ be normalized bump functions in $B(x,r)$ and $B(y,r)$ respectively, such that $\chi_1 \equiv 1$ on $B(x,r/2)$, $\chi_2 \equiv 1$ on $B(y,r/2)$. Then 
$$
T(x,y) = T_1 (\chi_1 k_y)(x) + T_1(\chi_2 k_y)(x) + \int_{\hat X} T_1(x,z) (1-\chi_1-\chi_2)(z) T_2(z,y) \hat m(z).
$$
The last term can be differentiated in both $x$ and $y$ under the integral, and the desired estimates follow. Thus it remains to consider the first two terms. But in the first term, by continuity of $T$ on $C^{\infty}(\hat X)$, one can differentiate with respect to $y$, and obtain
$$
(\hat \nabla_b)_y^{\alpha_2} [T_1 (\chi_1 k_y)(x)] = T_1 [\chi_1 (\hat \nabla_b)_y^{\alpha_2} k_y](x);
$$
the latter is $T_1$ acting on $r^{-3-|\alpha_2|}$ times a normalized bump function in $B(x,r)$. Thus by cancellation condition on $T_1$, 
$$
|(\hat \nabla_b)_x^{\alpha_1} (\hat \nabla_b)_y^{\alpha_2} [T_1 (\chi_1 k_y)(x)]| \lesssim r^{-3-|\alpha|}.
$$
Similarly,
\begin{align*}
T_1(\chi_2 k_y)(x)
&= \int_{\hat X} T_1(x,z) \chi_2(z) T_2(z,y) \hat m(z) \\
&= \ol T_2^* [ T_1(x,\cdot) \chi_2 (\cdot) ](y).
\end{align*}
By continuity of $\ol T_2^*$ on $C^{\infty}(\hat X)$, one can differentiate with respect to $x$, and obtain
$$
(\hat \nabla_b)_x^{\alpha_1} [T_1(\chi_2 k_y)(x)]
=\ol T_2^*[ (\hat \nabla_b)_x^{\alpha_1} T_1(x,\cdot) \chi_2 (\cdot) ](y).
$$
The latter is $\ol T_2^*$ acting on a $r^{-4-|\alpha_1|}$ times a normalized bump function in $B(y,r)$, so by cancellation condition on $\ol T_2^*$, we have
$$
|(\hat \nabla_b)_x^{\alpha_1} (\hat \nabla_b)_y^{\alpha_2} [T_1(\chi_2 k_y)(x)]| \lesssim r^{-3-|\alpha|}.
$$
This proves our desired estimates.

It remains to show that $T(x,y)$ is the kernel of the operator $T$, in the sense that (\ref{eq:Tkernelrep}) holds for all $f \in C^{\infty}(\hat X)$ and all $x$ not in the support of $f$. In fact, fix such an $f$, and a closed set $K$ disjoint from the support of $f$. Let $\chi \in C^{\infty}(\hat X)$ such that $\chi = 1$ on a neighborhood of $K$, and $\chi = 0$ on the support of $f$. Then for $x \in K$, $$Tf(x) = T_1(\chi T_2 f)(x) + T_1 ((1-\chi) T_2 f)(x).$$ The second term is equal to
\begin{align*}
&\int_{\hat X} T_1(x,z) (1-\chi)(z) T_2 f(z) \hat m(z) \\
=& \int_{\hat X} \left( \int_{\hat X} T_1(x,z) (1-\chi)(z) T_2(z,y)  \hat m(z) \right) f(y) \hat m(y),
\end{align*}
the last equality following from Fubini's theorem. We claim that for almost every $x \in \hat X$, the first term is equal to 
\begin{equation} \label{eq:T1chikyfy}
\int_{\hat X} T_1 (\chi k_y)(x) f(y) \hat m(y),
\end{equation}
where $k_y(z) := T_2(z,y)$; if this were true, then (\ref{eq:Tkernelrep}) holds for almost every $x \in K$. Since $K$ is an arbitrary compact set disjoint from the support of $f$, (\ref{eq:Tkernelrep}) holds for almost every $x$ not in the support of $f$. But then by continuity of $T(x,y)$, and bounded convergence theorem, (\ref{eq:Tkernelrep}) holds for every $x$ not in the support of $f$. Our theorem then follows. 

To prove our claim, we approximate $$\chi(z) T_2 f(z) = \int_{\hat X} \chi(z) T_2(z,y) f(y) \hat m(y)$$ by Riemann sums; 
since $T(z,y)$ is smooth away from the diagonal, and $f$ is smooth, by uniform continuity, the Riemann sums converge uniformly to $\chi T_2 f$. By continuity of $T$ in $L^2(\hat m)$, we have $T_1$ of the Riemann sums converging in $L^2$ to $T_1(\chi T_2f)$. Thus by passing to a subsequence, $T_1$ of the Riemann sums converge almost everywhere to $T_1 (\chi T_2 f)$. On the other hand, $T_1$ of the Riemann sums is the Riemann sums of (\ref{eq:T1chikyfy}); by continuity of $T_1(\chi k_y)(x)$ for $(x,y) \in K \times \text{supp}\, f$, the Riemann sums of (\ref{eq:T1chikyfy}) converges uniformly to (\ref{eq:T1chikyfy}). This establishes our claim.
\end{proof}

\begin{proof}[Proof of Theorem~\ref{thm2.3}]
Suppose $T$ is a smoothing operator of order $j$, $0 \leq j < 4$, and $g \in \mathcal{E}(\hat \rho^{-\gamma})$, $j < \gamma < 4$. Fix $k\in\mathbb N_0$ and fix a point $x_0 \neq p$ sufficiently close to $p$. Let $r=\frac{1}{4}\vartheta(x_0,p)$, and $\eta$ be a normalized bump function supported in $B(x_0,r)$, with $\eta=1$ on $B(x_0,r/2)$. Then, 
\begin{align*}
|\hat{\nabla}_b^k T g(x_0)| \leq |\hat{\nabla}_b^k T (\eta g)(x_0)| + |\hat{\nabla}_b^k T ((1-\eta)g)(x_0)|
\end{align*}
and $r^{\gamma}(\eta g)(x)$ is a normalized bump function on $B(x_0,r)$. So by the cancellation condition for $T$, we see that 
\[|\hat{\nabla}_b^k T(\eta g)(x_0)|\leq C_k r^{j-\gamma-k},\]
where $C_k>0$ is a constant independent of $x_0$ and $r$. By using the kernel estimates, $\hat{\nabla}_b^k T((1-\eta)g)(x_0)$ can be estimated by writing out the integrals directly: 
\[\hat{\nabla}_b^k T((1-\eta)g)(x_0) = \int (\hat{\nabla}_b^k T)(x_0,y) (1-\eta)(y) g(y)\hat m(y),\]
which can be split into two pieces. The first is over where $\vartheta(y,p)\leq r$; this piece is dominated by 
\[D_k \int_{\vartheta(y,p)<r} r^{-4+j-k}\vartheta(y,p)^{-\gamma}\hat m(y)\leq D_k r^{j-\gamma-k},\]
where $D_k>0$ is a constant independent of $x_0$ and $r$. The second piece is over where $\vartheta(y,p) > r$; note since we have cut off those $y$ near $x_0$ with $1-\eta$ already, we can assume that $\vartheta(y,x_0) > r/2$ on this piece of integral as well. As a result, $\vartheta(x_0,y) \simeq \vartheta(y,p)$; it follows that this piece is bounded by 
\[\int_{\vartheta(y,p) > r} \vartheta(y,p)^{-4+j-k} \vartheta(y,p)^{-\gamma} \hat m(y)\leq E_k r^{j-\gamma-k},\] 
where $E_k>0$ is a constant independent of $x_0$ and $r$. Altogether, 
\[|\hat{\nabla}_b^k Tg(x_0)|\leq\Td C_k r^{j-\gamma-k}\] 
as desired, where $\Td C_k>0$ is a constant independent of $x_0$ and $r$. This completes our proof.
\end{proof}

We now turn to the proof of Theorem~\ref{thm2.1}. The key is the following $L^2$ estimate, which can be proved by microlocalization and integration by parts (see e.g. Kohn~\cite{Koh85}). Suppose $\hat{\overline{Z}}$ is a local section of $T^{0,1} \hat X$ with $\langle \hat{\overline{Z}} | \hat{\overline{Z}} \rangle_{\hat \theta} = 1$ on some ball $B(x,2r) \subset \hat X$, and $\hat{\overline{Z}}^*$ be its formal adjoint under $L^2(\hat m)$.

\begin{prop}\label{p-keyest}
If $\hat{\overline{Z}}^* v = u$ on $B(x,2r)$, where $u, v \in C^\infty(\hat X)$, then for every $k\in\mathbb N_0$, there is a constant $C_k>0$ independent of $r$ and $x$ such that
\begin{equation} \label{basicL2est}
\begin{split}
&\|\hat{\nabla}_b^k u\|_{L^2(B(x,r))}\\
&\quad\leq C_k ( \|\hat{\nabla}_b^{k-1} \hat{\overline{Z}} u\|_{L^2(B(x,2r))} + r^{-k} \|u\|_{L^2(B(x,2r))} + r^{-(k+1)} \|v\|_{L^2(B(x,2r))}).
\end{split}
\end{equation}
\end{prop} 

\begin{prop}\label{p-keyest2}
If $\hat{\overline{Z}} v = u$ on $B(x,2r)$, where $u, v \in C^{\infty}(\hat X)$, then for every $k\in\mathbb N_0$, there is a constant $C_k>0$ independent of $r$ and $x$ such that
\begin{equation} \label{basicL2est2}
\begin{split}
&\|\hat{\nabla}_b^k u\|_{L^2(B(x,r))}\\
&\quad\leq C_k ( \|\hat{\nabla}_b^{k-1} \hat{\overline{Z}}^* u\|_{L^2(B(x,2r))} + r^{-k} \|u\|_{L^2(B(x,2r))} + r^{-(k+1)} \|v\|_{L^2(B(x,2r))}).
\end{split}
\end{equation}
\end{prop}

Here the $L^2$ norms are taken using the norms of $L^2(\hat m)$. Various variants and refinements of these estimates are very well-known; however, we have not been able to locate a precise reference for these estimates. For completeness and the convenience of the reader, we present the proofs of these estimates in an appendix.

Using these $L^2$ estimates, one can prove that $\hat \Pi$ and $\hat K$ (more precisely, the local representations $\hat K_i$'s defined by (\ref{eq:hatKdefn})) maps $C^{\infty}(\hat X)$ continuously into $C^{\infty}(\hat X)$, and that their kernels satisfy differential inequalities of the correct order; c.f. Christ \cite{Ch88I}, \cite{Ch88II}. In particular, if $\hat \Pi(x,y)$ and $\hat K_i(x,y)$ denote the Schwartz kernels of $\hat \Pi$ and $\hat K_i$ respectively, then they are smooth away from the diagonal, and 
$$
|(\hat \nabla_b)_x^{\alpha_1} (\hat \nabla_b)_y^{\alpha_2} \hat \Pi (x,y)| \lesssim_{\alpha} \vartheta(x,y)^{-4-|\alpha|}, 
$$
$$
|(\hat \nabla_b)_x^{\alpha_1} (\hat \nabla_b)_y^{\alpha_2} \hat K_i (x,y)| \lesssim_{\alpha} \vartheta(x,y)^{-3-|\alpha|}.
$$
It thus remains to prove cancellation conditions for $\hat \Pi$ and $\hat K_i$. The proofs will be based on strategies similar to those used in the proofs of the above kernel estimates.

Let $\phi$ be a normalized bump function in $B(x,r)$. We claim that
\begin{equation} \label{cancondhatS}
\|\hat{\nabla}_b^k \hat{\Pi} \phi\|_{L^{\infty}(B(x,r))} \leq C_k r^{-k}.
\end{equation} 
This will follow from the continuity of $\hat \Pi$ on $C^{\infty} (\hat X)$ if $r$ is sufficiently large. Therefore, without loss of generality, we assume that $r < r_0/2$, where $r_0$ is some small absolute constant, so that one can find a section $\hat{\overline{Z}}$ of $T^{1,0} \hat X$ that does not vanish on $B(x,2r)$ (c.f. discussion before (\ref{eq:hatKdefn})). We further normalize $\hat{\overline{Z}}$ so that $\langle \hat{\overline{Z}} | \hat{\overline{Z}} \rangle_{\hat \theta} = 1$ on $B(x,2r)$. Now \eqref{cancondhatS} is the same as showing $\|\hat{\nabla}_b^k (I- \hat{\Pi}) \phi\|_{L^{\infty}(B(x,r))} \leq C_k r^{-k}$. Let $v$ be such that $\hat{\ol{\pr}}^*_b v = (I-\hat{\Pi})\phi$, with $v$ orthogonal to the kernel of $\hat{\ol{\pr}}^*_b$. Then by \eqref{basicL2est}, we have
\[\begin{split}
&\|\hat{\nabla}_b^k (I-\hat{\Pi}) \phi \|_{L^2(B(x,r))}\\
&\leq C_k ( \|\hat{\nabla}_b^{k-1} \hat{\overline{Z}} \phi\|_{L^2(B(x,2r))}+ r^{-k} \| (I-\hat{\Pi}) \phi\|_{L^2(B(x,2r))}\\
&\quad+r^{-(k+1)} \| \langle v | \hat{\omega} \rangle_{\hat \theta} \|_{L^2(B(x,2r))}),\end{split}\]
where $\hat{\omega}$ is the dual $(0,1)$ form to $\hat{\overline{Z}}$ on $B(x,2r)$.
The first term on the right hand side is bounded by $$C_k r^{-k} |B(x,2r)|^{1/2} = C_k r^{2-k}.$$ In the second term, we estimate $\| (I-\hat{\Pi}) \phi\|_{L^2(B(x,2r))}$ trivially by $\|\phi\|_{L^2(\hat{X})} \leq C |B(x,r)|^{1/2} = C r^2$ (recall $\phi$ is normalized in $B(x,r)$), so that the second term is bounded by $C_k r^{2-k}$ as well. In the last term, we estimate using Poincar\'{e}-type inequality (see Corollary 11.5* of Christ \cite{Ch88I}): 
\[\|\langle v | \hat{\omega} \rangle_{\hat \theta}\|_{L^2(B(x,2r))} \leq C r \|(I-\hat{\Pi})\phi\|_{L^2(\hat{X})} \leq C r \|\phi \|_{L^2(\hat{X})} \leq C r r^2.\] 
Thus altogether, 
\[\|\hat{\nabla}_b^k (I-\hat{\Pi}) \phi \|_{L^2(B(x,r))} \leq C_k r^{2-k},\]
and since this holds for all $k$, by Sobolev embedding, 
\[\|\hat{\nabla}_b^k (I-\hat{\Pi}) \phi \|_{L^{\infty}(B(x,r))} \leq C_k r^{-k}\] 
as desired. \eqref{cancondhatS} follows, and since $\hat \Pi$ is self-adjoint on $L^2(\hat m)$, this completes the proof that $\hat \Pi$ is smoothing of order 0.

Let now $\hat \Pi_1 \colon L^2_{(0,1)}(\hat m, \hat \theta) \to L^2_{(0,1)}(\hat m, \hat \theta)$ be the Szeg\"o projection on $(0,1)$ forms, i.e. the orthogonal projection onto the kernel of $\hat \ddbar_b^*$ in $L^2_{(0,1)}(\hat m, \hat \theta)$. Using the partition of unity given just before (\ref{eq:hatKdefn}), we can define local representations of $\hat \Pi_1$, by letting
$$
(\hat \Pi_1)_{ij} \varphi := \langle \eta_i \hat \Pi_1 (\eta_j \varphi \hat \omega_j) | \hat \omega_i \rangle_{\hat \theta}.
$$ 
Then $(\hat \Pi_1)_{ij}$ sends functions to functions for all $1 \leq i,j \leq N$, and 
$$
\hat \Pi_1 \phi = \sum_{i,j} \left( (\hat \Pi_1)_{ij} [\langle \phi | \hat \omega_j \rangle_{\hat \theta} ] \right) \hat \omega_i
$$
for any $(0,1)$ form $\phi$ on $\hat X$.
A proof similar to the above shows that $\hat \Pi_1$ is a smoothing operator of order 0, in the sense that the local representations $(\hat \Pi_1)_{ij}$ are all smoothing of order 0; for instance, to prove that $(\hat \Pi_1)_{ij}$ satisfies the desired cancellation conditions, if $\varphi$ is a normalized bump function on a sufficiently small ball $B(x,2r)$ that intersects the support of $\eta_i$, one would apply Proposition~\ref{p-keyest2} with $\hat{\overline{Z}}$ being the $(0,1)$ vector field dual to $\hat \omega_i$, $u = \eta_j \varphi \langle \hat \omega_j | \hat \omega_i \rangle_{\hat \theta} - \langle \Pi_1 (\eta_j \varphi \hat \omega_j) | \hat \omega_i \rangle_{\hat \theta}$, and $v$ being a function that solves $\hat \ddbar_b v = (I - \hat \Pi_1) (\eta_j \varphi \hat \omega_j)$. In fact then $\hat{\overline{Z}} v = u$ on $B(x,2r)$. We omit the details. 

Now, let $\varphi$ be a normalized bump function in a ball $B(x,r)$ that intersects the support of $\eta_i$, with $r < r_0/4$ as before. We prove cancellation properties for $\hat{K}_i$ and $\hat{K}_i^*$, namely
\begin{equation} \label{cancondhatK}
\|\hat{\nabla}_b^k \hat{K}_i \varphi\|_{L^{\infty}(B(x,r))} \leq C_k r^{1-k},
\end{equation}
 and
\begin{equation} \label{cancondhatK*}
\|\hat{\nabla}_b^k \hat K_i^* \varphi\|_{L^{\infty}(B(x,r))} \leq C_k r^{1-k}.
\end{equation} 
To prove the former, let 
\[u = (I - \hat\Pi) (\psi\hat{K}_i \varphi),\] 
where $\psi\equiv 1$ on $B(x,2r)$, and is a normalized bump function on $B(x,4r)$. We apply estimate (\ref{basicL2est}) for this $u$. On $B(x,2r)$, $\hat\ddbar_b u = (I-\hat\Pi_1)(\eta_i \varphi \hat \omega_i)$, since $\psi$ is identically $1$ there. In other words, writing $\hat{\overline{Z}}$ for the dual of $\hat \omega_i$, we have $$\hat{\overline{Z}} u = \langle (I-\hat\Pi_1)(\eta_i \varphi \hat \omega_i) | \hat \omega_i \rangle_{\hat \theta} = \eta_i \varphi - \sum_j [(\hat \Pi_1)_{ji} \varphi] \langle \hat \omega_j | \hat \omega_i \rangle_{\hat \theta}$$ on $B(x,2r)$. So the first term on the right hand side of \eqref{basicL2est} is bounded by $Cr^{1-k} |B(x,2r)|^{1/2} = C r^{1-k} r^2$, by the cancellation property of $\hat \Pi_1$ we just proved above. On the other hand, by Proposition B of Christ~\cite{Ch88II}, since $u$ is orthogonal to the kernel of $\hat\ddbar_b$, 
\[\|u\|_{L^2(B(x,2r))} \leq Cr \|\hat\ddbar_bu\|_{L^2_{(0,1)}(\hat{X})},\] 
which implies 
\[\|u\|_{L^2(B(x,2r))} \leq Cr \left( \|\psi (I-\hat\Pi_1)(\eta_i \varphi \hat \omega_i)\|_{L^2_{(0,1)}(\hat{X})} + \|(\hat{\nabla}_b\psi)\hat{K}_i\varphi\|_{L^2(\hat{X})}\right).\]
But the first term in the bracket is bounded by $\|\eta_i \varphi\|_{L^2(\hat{X})} \leq C |B(x,r)|^{1/2} = C r^2$, and the second term is bounded by $C r^2$ by the kernel estimates on $\hat{K}_i$ (note $\varphi$ is supported on $B(x,r)$, while $\hat{\nabla}_b \psi$ is supported in an annulus $B(x,4r)\setminus B(x,2r)$.) This in turn implies
\[\|u\|_{L^2(B(x,2r))} \leq C r r^2,\] and the second term on the right hand side of \eqref{basicL2est} is bounded by $C r^{1-k} r^2$.  Finally, let $v$ be such that $\hat{\ol{\pr}}_b^*v = u$, so that $\hat{\overline{Z}}^* [\langle v | \hat \omega_i \rangle_{\hat \theta}] = u$ on $B(x,2r)$. Then 
\[\|\langle v | \hat \omega_i \rangle_{\hat \theta}\|_{L^2(B(x,2r))} \leq C r \|u\|_{L^2(\hat{X})} \leq C r \|\psi \hat{K}_i \varphi\|_{L^2(\hat{X})} \leq C r^2 \|\varphi\|_{L^2(\hat{X})} \leq C r^2 r^2.\]
(The first and the fourth inequality are both applications of Proposition B of Christ \cite{Ch88II} again.) Thus altogether, 
$\|\hat{\nabla}_b^k u \|_{L^2(B(x,r))} \leq C_k r^{1-k} r^2$ for all $k$, and by Sobolev embedding, this implies
\[\|\hat{\nabla}_b^k u \|_{L^{\infty}(B(x,r))} \leq C_k r^{1-k}\] 
for all $k$. 

Remember we want the same estimate for $\hat{K}_i\varphi$ in place of $u$, so as to prove \eqref{cancondhatK}. But $\hat{K}_i \varphi - u$ can be computed on $B(x,r)$ fairly easily. In fact, since $\hat{K}_i \varphi = \hat{K}(\eta_i \varphi \hat \omega_i)$ is orthogonal to the kernel of $\hat\ddbar_b$, we have $\hat{K}_i \varphi =(I-\hat{\Pi})\hat{K}_i \varphi$. So 
\[\hat{K}_i \varphi - u = (I-\hat{\Pi})(1-\psi)\hat{K}_i \varphi.\] 
Since $\psi\equiv1$ on $B(x,2r)$, we have 
\[\hat{K}_i\varphi-u =-\hat{\Pi}(1-\psi)\hat{K}_i\varphi \quad \mbox{on $B(x,r)$}.\]
It follows that for $y\in B(x,r)$, 
\[
\hat{\nabla}_b^k (\hat{K}_i\varphi-u)(y)=-\sum_{j=1}^{\infty}\int_{2^jr\leq\vartheta(z,x)\leq 2^{j+1} r}(\hat{\nabla}_b)_y^k \hat{\Pi}(y,z) (1-\psi)(z)\hat{K}_i\varphi(z) \hat m(z),\] 
so
\[\begin{split}
&\|\hat{\nabla}_b^k(\hat{K}_i\varphi-u)\|_{L^{\infty}(B(x,r))}\\
&\leq C_k\sum_{j=1}^{\infty} (2^jr)^{-2-k}\|\hat{K}_i\varphi\|_{L^2(B(x,2^{j+1}r))}\\
&\leq C_k \sum_{j=1}^{\infty} (2^j r)^{-2-k} (2^j r) \|\hat\ddbar_b \hat{K}_i\varphi\|_{L^2(\hat{X})}.\end{split}\] 
(The last inequality is Proposition B of Christ \cite{Ch88II}.) By estimating the term $\|\hat{\ddbar_b}\hat{K}_i\varphi\|_{L^2(\hat{X})}$ by $\|(I-\hat\Pi_1) (\eta_i \varphi \hat \omega_i)\|_{L^2(\hat{X})} \leq \|\varphi\|_{L^2(\hat{X})} \leq C r^2$, we get 
\[\|\hat{\nabla}_b^k (\hat{K}_i\varphi-u) \|_{L^{\infty}(B(x,r))} \leq C_k r^{-2-k} r r^2 = C_k r^{1-k}.\] 
By combining with the previous estimate on $\hat{\nabla}_b^k u$, we get 
\[\|\hat{\nabla}_b^k \hat{K}_i\varphi \|_{L^{\infty}(B(x,r))} \leq C_k r^{1-k},\] as desired. \eqref{cancondhatK} follows then. A similar argument proves \eqref{cancondhatK*}, since $\hat K^*$ is the partial inverse of $\hat \ddbar_b^*$. This shows that $\hat K$ and $\hat K^*$ are smoothing of order 1, and it follows now that $\hat N = \hat K \hat K^*$ (more precisely, $\hat N = \sum_{i,j} \hat K_i \langle \hat \omega_i | \hat \omega_j \rangle_{\hat \theta} \hat K_j^*$) is smoothing of order 2. This completes the proof of Theorem~\ref{thm2.1}. 

We conclude this section by making the following useful observation: as was demonstrated in Folland-Stein \cite{FoSt} (see Theorem 15.15 there), on $\hat X$ one can construct some smoothing operators $T_0$, $T_1$ of orders $1$, such that schematically,
$$
I = T_1 \hat \nabla_b + T_0.
$$
Thus if $A$ is a smoothing operator of order $1$, then for any positive integers $j$, one can then find smoothing operators $A_0, A_1, \dots, A_j$ of orders $1$ such that 
\begin{equation} \label{eq:commute}
\hat \nabla_b^j A = \sum_{i=0}^j A_i \hat \nabla_b^i;
\end{equation}
in fact, e.g. when $j = 1$, one just needs to observe
$$
\hat \nabla_b A = \hat \nabla_b A (T_1 \hat \nabla_b + T_0),
$$
and the desired equality follows by letting $A_i = \hat \nabla_b A T_i$, $i=0,1$. ($A_i$ is smoothing of order 1 by Theorem~\ref{thm2.2} above.) The general case for (\ref{eq:commute}) follows by induction on $j$. (\ref{eq:commute}) can be thought of as a way of commuting derivatives past smoothing operators. In particular, if $NL^{k,p}$ denotes the non-isotropic Sobolev space, given by the set of all functions whose $\hat \nabla_b^j$ is in $L^p(\hat m)$ for $j=0,1,2,\dots,k$, then (\ref{eq:commute}) implies the first part of the following proposition:

\begin{prop} \label{prop:nonisosmoothing}
\begin{enumerate}[(a)]
\item Any smoothing operator of order 1 maps $NL^{k,p}$ continuously into $NL^{k+1,p}$, for all $k \geq 0$ and all $1 < p < \infty$.
\item Any smoothing operator of order 1 maps $L^{\infty}(\hat m)$ continuously into $L^{\infty}(\hat m)$.
\end{enumerate}
\end{prop}

The last part of this proposition then follows from the case $k = 0$ of the first part by noting that $L^{\infty}(\hat m)$ embeds into $L^p(\hat m)$ for any $p > 4$, and that $NL^{1,p}$ embeds into $L^{\infty}(\hat m)$ by Sobolev embedding.

\section{A key mapping property} \label{sect:aux}

\subsection{The main theorem} In this section, we prove the following theorem, which allows one to establish the important mapping property (\ref{eq:I+hatRchi}) of $(I+\hat{R}^{*,\Td m} \chi)^{-1}$. We use the notion of smoothing operators of order $j$ we introduced in the last section.

\begin{thm}\label{t-slalkeyI}
Suppose $A$ is a smoothing operator of order 1 on $\hat X$, and $h$ is a function in $\mathcal{E}(\hat \rho^1)$ supported in a sufficiently small neighborhood of $p$. We write $h$ also for the operator that is multiplication by $h$. Then the bounded linear operator $I-Ah \colon L^2(\hat m) \to L^2(\hat m)$ is invertible, and its inverse extends to a continuous linear map
$$
(I-Ah)^{-1} \colon \mathcal{E}(\hat \rho^{-4+\delta}) \to \mathcal{E}(\hat \rho^{-4+\delta})
$$ 
for every $0<\delta<4$.
\end{thm} 

The key of the proof is the following 

\begin{lem} \label{lem:hrhokderiv}
Suppose $A$ and $h$ are as in the above theorem. Then for any non-negative integer $k$, and any function $u \in C^{\infty}(X) \cap L^{\infty}(\hat X)$, we have 
$$\|\hat \rho^k \hat \nabla_b^k (Ah)^{k+1} u\|_{L^{\infty}(\hat m)} \lesssim_k \|u\|_{L^{\infty}(\hat m)}.$$
\end{lem}

Assuming the lemma, we first prove Theorem~\ref{t-slalkeyI}.

\begin{proof}[Proof of Theorem~\ref{t-slalkeyI}]
First, if $A$ is smoothing of order 1, then since $\hat X$ is compact, the kernel of $A$ satisfies 
$$
\sup_{x \in \hat X} \int_{\hat X} |A(x,y)| \hat m(y) +
\sup_{y \in \hat X} \int_{\hat X} |A(x,y)| \hat m(x) < \infty.
$$
It follows that $A$ is bounded on $L^2(\hat m)$. If the support of $h$ is a sufficiently small neighborhood of $p$, then since $h \in \mathcal{E}(\hat \rho^1)$, one can make $\|h\|_{L^{\infty}}$ sufficiently small. Thus the norm of $Ah$, as a bounded linear operator on $L^2(\hat m)$, can be made smaller than $1/2$. This in turn allows one to invert $I-Ah$ by a Neumann series: for $u \in L^2(\hat m)$, one has
$$
u + (Ah) u + (Ah)^2 u + (Ah)^3 u + \dots
$$
converging to a limit $v$ in $L^2(\hat m)$, and $(I-Ah)v = u$. Thus $(I-Ah)$ is invertible on $L^2(\hat m)$, and its inverse is given by the Neumann series
$$
(I-Ah)^{-1} = I + Ah + (Ah)^2 + (Ah)^3 + \dots.
$$

Now we extend $(I-Ah)^{-1}$ to $\mathcal{E}(\hat \rho^{-4+\delta})$, $0 < \delta < 4$. In order to do so, we need to further assume that the norm of $Ah$, as a bounded linear operator on $L^{\infty}(\hat m)$, is smaller than $1/2$. That can be achieved if the support of $h$ is sufficiently small. 

Suppose now $u \in \mathcal{E}(\hat \rho^{-4+\delta})$, $0 < \delta < 4$. Let $v = [I+(Ah)+(Ah)^2+\dots] u$. We want to show that $v \in \mathcal{E}(\hat \rho^{-4+\delta})$. To do so, suppose $k$ is a non-negative integer. To show that
\begin{equation} \label{eq:hrhokderiv}
\hat \rho^{k + (4-\delta)} \hat \nabla_b^k v \in L^{\infty}(\hat m),
\end{equation}
we split the sum defining $v$ into two parts: let $v_1 = [I+(Ah)+\dots+(Ah)^{k+2}] u$, and $v_2 = [(Ah)^{k+3} + (Ah)^{k+4} + \dots ] u$. Then $\hat \rho^{k + (4-\delta)} \hat \nabla_b^k  v_1 \in L^{\infty}(\hat m)$ by Theorem~\ref{thm2.3}, since each term of $v_1$ is in $\mathcal{E}(\hat \rho^{-4+\delta})$. Furthermore, note that 
\begin{equation} \label{eq:smoothtoLinfty}
A \colon \mathcal{E}(\hat \rho^{-1+\delta}) \to L^{\infty}(\hat m) \quad \text{for all $0 < \delta < 1$.}
\end{equation}  
This holds because $\mathcal{E}(\hat \rho^{-1+\delta}) \subset L^p(\hat m)$ for some $p > 4$, and $A \colon L^p(\hat m) \to L^{\infty}(\hat m)$ whenever $p > 4$. Thus from $u \in \mathcal{E}(\hat \rho^{-4+\delta})$, $0 < \delta < 4$, we conclude, from Theorem~\ref{thm2.3} and (\ref{eq:smoothtoLinfty}), that $(Ah)^2 u \in C^{\infty}(X) \cap L^{\infty}(\hat X)$. As a result, by Lemma~\ref{lem:hrhokderiv}, and our bound of $Ah$ on $L^{\infty}(\hat m)$, we have
\begin{align*}
\|\hat \rho^{k} \hat \nabla_b^k v_2 \|_{L^{\infty}(\hat m)} &\leq C_k \sum_{\ell = 0}^{\infty} \|(Ah)^{2 + \ell} u\|_{L^{\infty}(\hat m)} \\
&\leq C_k \sum_{\ell = 0}^{\infty} 2^{-\ell} \|(Ah)^2 u\|_{L^{\infty}(\hat m)} \leq C_k.
\end{align*}
Combining this with the bound for $v_1$, (\ref{eq:hrhokderiv}) follows, and this shows $v \in \mathcal{E}(\hat \rho^{-4+\delta})$ as desired.
\end{proof}

\subsection{An auxiliary family of operators}

Now we need to detour into a discussion of a two-parameter family of operators, that will be indexed by two non-negative integers $j$ and $\ell$. Suppose $A$ is a smoothing operator of order 1. Suppose also that $q_{\ell}(x,y)$ is a function in $C^{\infty}(\hat X \times \hat X)$ that vanishes to non-isotropic order $\ell$ along the diagonal, i.e.
$$
|q_{\ell}(x,y)| \lesssim \vartheta(x,y)^{\ell}
$$
for some non-negative integer $\ell$. We will write $q_{\ell}^x(y) := q_{\ell}(x,y)$; by abuse of notation, we will also denote by $q_{\ell}^x$ the multiplication operator $v(y) \mapsto q_{\ell}^x(y) v(y)$. Given a non-negative integer $j$, for $v \in C^{\infty}(\hat X)$ and $x \in \hat X$, we define
\begin{equation} \label{eq:Tjl}
T v(x) = \left. (\hat \nabla_b)_z^j \right|_{z = x} [A q_{\ell}^x v](z).
\end{equation}
This is well defined, since $q_{\ell}^x v$ is a $C^{\infty}$ function on $\hat X$ for each fixed $x$, and $A$ maps $C^{\infty}(\hat X)$ into $C^{\infty}(\hat X)$. We will see below that this assignment $v \mapsto Tv$ defines a continuous map from $C^{\infty}(\hat X)$ to $C^{\infty}(\hat X)$. Since the properties of this map depend mainly only on the integers $j$ and $\ell$, we will denote any operator of this form by $T_{j,\ell}$. In other words, if $v \in C^{\infty}(\hat X)$ and $x \in \hat X$, then $T_{j,\ell}v(x)$ is given by the right hand side of (\ref{eq:Tjl}) for some smoothing operator $A$ of order 1, and some $q_{\ell} \in C^{\infty}(\hat X \times \hat X)$ that vanishes to non-isotropic order $\ell$ along the diagonal.

\begin{lem}
Suppose $v \in C^{\infty}(\hat X)$, and for all $x \in \hat X$, we have
\begin{equation} \label{eq:Tjldef}
T_{j,\ell} v(x) := \left. (\hat \nabla_b)_z^j \right|_{z = x} [A q_{\ell}^x v](z)
\end{equation}
for some $A$ and $q_{\ell}$ as above. Then $T_{j,\ell} v$ is a $C^{\infty}$ function on $\hat X$. Furthermore, for any $r \geq 0$,
\begin{equation} \label{eq:Tjlderiv}
(\hat \nabla_b)_x^r T_{j,\ell}v(x) = \sum_{s=0}^r \binom{r}{s} \left. (\hat \nabla_b)_z^{j+r-s} \right|_{z = x} [A (\hat \nabla_b)_x^{s} q_{\ell}^x v](z)
\end{equation}
where $(\hat \nabla_b)_x^{s} q_{\ell}^x$ denotes the multiplication operator $v(y) \mapsto (\hat \nabla_b)_x^{s} q_{\ell}(x,y) v(y)$. In addition, the adjoint $T_{j,\ell}^*$ of $T_{j,\ell}$ with respect to $L^2(\hat m)$ maps $C^{\infty}(\hat X)$ into itself, and is given by
\begin{equation} \label{eq:Tjl*}
T_{j,\ell}^* w (z) = A^* [(\hat \nabla_b^*)^j(\ol q_{\ell,z}w)](z)
\end{equation}
where $q_{\ell,z}(x):= q_{\ell}(x,z)$.
\end{lem}

\begin{proof}
Suppose $v \in C^{\infty}(\hat X)$, and $x \in \hat X$. First we show that $T_{j,\ell}$ is differentiable at $x$, and that (\ref{eq:Tjlderiv}) holds when $r = 1$. In fact, for any smooth curve $\gamma \colon (-1,1) \to \hat X$ with $\gamma(0) = x$, $\gamma'(0) = Y$, we have 
\begin{align}
& \frac{1}{\varepsilon} \left[ T_{j,\ell} v(\gamma(\varepsilon)) - T_{j,\ell} v(\gamma(0)) \right] \notag \\
=&  \left. (\hat \nabla_b)_z^j \right|_{z = \gamma(\varepsilon)}  \left[A \left( \frac{q_{\ell}^{\gamma(\varepsilon)} v - q_{\ell}^{\gamma(0)} v}{\varepsilon} - Y_x q_{\ell}^x \right) v\right](z) \notag \\
& \quad +  \frac{1}{\varepsilon} \left(\left. (\hat \nabla_b)_z^j \right|_{z = \gamma(\varepsilon)} - \left. (\hat \nabla_b)_z^j \right|_{z = \gamma(0)} \right) [A q_{\ell}^x v](z) \label{eq:derivTjlexpans} \\
& \quad + \left. (\hat \nabla_b)_z^j \right|_{z = \gamma(\varepsilon)}  [A (Y_x q_{\ell}^x) v](z). \notag
\end{align}
(We wrote $Y_x$ to emphasize that the derivative is with respect to $x$.) Now the first term on the right hand side of (\ref{eq:derivTjlexpans}) is bounded by
$$
\left\|(\hat \nabla_b)_z^j A \left( \frac{q_{\ell}^{\gamma(\varepsilon)} v - q_{\ell}^{\gamma(0)}}{\varepsilon} - Y_x q_{\ell}^x \right) v(z) \right\|_{L^{\infty}(\hat{m}(z))},
$$
which tends to zero as $\varepsilon \to 0$ since 
$$
\left( \frac{q_{\ell}^{\gamma(\varepsilon)} v - q_{\ell}^{\gamma(0)}}{\varepsilon} - Y_x q_{\ell}^x \right) v(z) \to 0 \quad \text{in $C^{\infty}(\hat X)$ as a function of $z$},
$$
and $A \colon C^{\infty}(\hat X) \to C^{\infty}(\hat X)$ is continuous. Next, the second term on the right hand side of (\ref{eq:derivTjlexpans}) converges to
$$
Y_z \left. (\hat \nabla_b)_z^j  [A q_{\ell}^x v](z) \right|_{z = x}
$$
as $\varepsilon \to 0$, since $A q_{\ell}^x v(z)$ is $C^{\infty}$ as a function of $z \in \hat X$. Finally, the last term on the right hand side of (\ref{eq:derivTjlexpans}) converges to
$$
\left. (\hat \nabla_b)_z^j \right|_{z = x}  [A (Y_x q_{\ell}^x) v](z)
$$
as $\varepsilon \to 0$, since $(\hat \nabla_b)_z^j [A (Y_x q_{\ell}^x) v](z)$ is a continuous function of $z \in \hat X$. This proves $T_{j,\ell}v$ is differentiable at $x$, and that 
\begin{equation} \label{eq:Tjlderiv1}
[Y T_{j,\ell} v](x) = \left.(Y_z (\hat \nabla_b)_z^j )\right|_{z = x} [A q_{\ell}^x v](z) + \left. (\hat \nabla_b)_z^j \right|_{z = x} [A (Y_x q_{\ell}^x) v](z).
\end{equation}
In particular, (\ref{eq:Tjlderiv}) holds when $r = 1$. 

By successive differentiation of (\ref{eq:Tjlderiv1}), using the case $r = 1$ of (\ref{eq:Tjlderiv}), then shows that $T_{j,\ell}v \in C^{\infty}(\hat X)$, and that (\ref{eq:Tjlderiv}) holds for all $r \geq 1$.

Now we prove that $T_{j,\ell}^*$ is given by the expression (\ref{eq:Tjl*}). The crux of the matter is to show that this holds when $j = 0$. Recall that $T_{0,\ell} v(x) =  A (q_{\ell}^x v) (x).$ We cover $\hat X$ by balls of radius $\varepsilon$, select a finite subcover, and construct a partition of unity $\sum_j \zeta_j = 1$ subordinate to it. Then we pick a point $y_j$ in the support of $\zeta_j$ for each $j$, and let $$q_{\ell,(\varepsilon)}^x(y) = \sum_j q_{\ell}^x (y_j) \zeta_j(y).$$ (Note $y_j$ and $\zeta_j$ depends implicitly on $\varepsilon$.) We then have
$$
\sup_{x,y \in \hat X \times \hat X} |q_{\ell,(\varepsilon)}^x(y) - q_{\ell}^x(y)| \to 0
$$
as $\varepsilon \to 0$. Hence by part (b) of Proposition~\ref{prop:nonisosmoothing}, for $v \in C^{\infty}(\hat X)$,
$$
\sup_{x,z \in \hat X} |A [q_{\ell,(\varepsilon)}^x v](z) - A[q_{\ell}^{x} v](z)| \to 0
$$
as $\varepsilon \to 0$, which in turn implies that 
$$
A [q_{\ell,(\varepsilon)}^x v](x) \to T_{0,\ell}v(x)
$$
uniformly for $x \in \hat X$. Now 
$$
A [q_{\ell,(\varepsilon)}^x v](x)
= \sum_j q_{\ell}^x(y_j) A [\zeta_j v](x),
$$
and
\begin{align}
\int_{\hat X} \sum_j q_{\ell}^x(y_j) A [\zeta_j v](x) \overline{w(x)} \hat m(x)
&= \int_{\hat X} \sum_j \zeta_j(y) v(y) \overline{ A^* [ \overline{q}_{\ell,y_j} w](y) } \hat m(y) \notag \\
&= \int_{\hat X} v(y) \overline{ A^* [ \sum_j \zeta_j(y)\overline{q}_{\ell,y_j} w](y) } \hat m(y)  \label{eq:dualineqT0l}
\end{align}
Recall that $\zeta_j$ and $y_j$ depends on $\varepsilon$. As $\varepsilon \to 0$,
$$
\sup_{x, y \in \hat X} |\sum_j \zeta_j(y) \overline{q}_{\ell,y_j}(x) w(x) - q_{\ell,y}(x) w(x)| \to 0.
$$
Hence by part (b) of Proposition~\ref{prop:nonisosmoothing},
$$
\sup_{z, y \in \hat X} |A^* [\sum_j \zeta_j(y) \overline{q}_{\ell,y_j} w](z) - A^* [\overline{q}_{\ell,y} w](z)| \to 0
$$
as $\varepsilon \to 0$. In particular,
$$
A^* [\sum_j \zeta_j(y) \overline{q}_{\ell,y_j} w](y) \to A^* [\overline{q}_{\ell,y} w](y)
$$
uniformly for $y \in \hat X$. Hence by (\ref{eq:dualineqT0l}), 
$$
\int_{\hat X} T_{0,\ell} v(x) \overline{w(x)} \hat m(x) = \int_{\hat X} v(y) \overline{  A^* [\overline{q}_{\ell,y} w](y) } \hat m(y).
$$
This shows $T_{0,\ell}^* w(y) = A^* [\overline{q}_{\ell,y} w](y)$, as desired. 

We assume that \eqref{eq:Tjl*} holds for all $T^*_{k,\ell}$ with $0\leq k\leq j-1$. Take $r=1$ and replace $j$ to $j-1$ in \eqref{eq:Tjlderiv}, we have 
\begin{equation}\label{e-guTj}
T_{j,\ell}=\hat\nabla_bT_{j-1,\ell}-T_{j-1,\ell_0},
\end{equation}
where $\ell_0=\min\set{\ell-1,0}$ and
$$T_{j-1,\ell_0} v(x) := \left. (\hat \nabla_b)_z^{j-1}\right|_{z = x} [A(\hat\nabla_b)_xq_{\ell}^x v](z).$$
By taking adjoint of \eqref{e-guTj} in the sense of distribution with respect to $L^2(\hat m)$, we deduce 
\begin{equation}\label{e-guTjI}
T^*_{j,\ell}= T^*_{j-1,\ell}(\hat\nabla_b)^*-T^*_{j-1,\ell_0}.
\end{equation}
From \eqref{e-guTjI} and the induction assumptions, we can check that 
\[\begin{split}
T^*_{j,\ell}w(z)&=A^*[(\hat\nabla^*_b)^{j-1}(\overline{q}_{\ell,z}\hat\nabla^*_bw)](z)-A^*[(\hat\nabla^*_b)^{j-1}(\ol{\hat\nabla_b{q}_{\ell,z}})w)](z)\\
&=A^*[(\hat\nabla^*_b)^{j}(\overline{q}_{\ell,z}w)](z).
\end{split}\]
\eqref{eq:Tjl*} follows. As a result, by repeating the proof of (\ref{eq:Tjlderiv}), $T_{j,\ell}^*$ maps $C^{\infty}(\hat X)$ into $C^{\infty}(\hat X)$, with
$$
(\nabla_b)_z^r T_{j,\ell}^* w(z) = \sum_{s=0}^r \binom{r}{s} \left. (\hat \nabla_b)_y^{r-s} \right|_{y = z} [A (\hat \nabla_b)_z^{s}(\hat \nabla_b^*)^j(\ol q_{\ell,z}w)](y).
$$
\end{proof}

\begin{lem}
For any $j, \ell \geq 0$, the linear operators
$$
T_{j,\ell} \colon C^{\infty}(\hat X) \to C^{\infty}(\hat X)
$$
and
$$
T_{j,\ell}^* \colon C^{\infty}(\hat X) \to C^{\infty}(\hat X)
$$
considered in the previous lemma are both continuous.
\end{lem}

\begin{proof}
Suppose $v_m$ converges to $v$ in $C^{\infty}(\hat X)$ as $m \to \infty$. Then for any $k, s \geq 0$,
$$
\|(\hat \nabla_b)_z^k [(\hat \nabla_b)_x^s q_{\ell}^x(z) (v_m - v)(z)] \|_{L^2(\hat m(z))} \to 0 \quad \text{uniformly in $x \in \hat X$},
$$
and by continuity of $A$ on $NL^{k,2}$ (see Proposition~\ref{prop:nonisosmoothing}(a)), we see that for any $k, s \geq 0$,
$$
\|(\hat \nabla_b)_z^k [A (\hat \nabla_b)_x^s q_{\ell}^x (v_m - v)](z)\|_{L^2(\hat m(z))} \to  0 \quad \text{uniformly in $x \in \hat X$}.
$$
By Sobolev embedding, it follows that the same is true for all $k, s \geq 0$ if the $L^2$ norm is replaced by $L^{\infty}$ in the above equation. Hence by (\ref{eq:Tjlderiv}), we conclude that
$$
\|(\hat \nabla_b)^r T_{j,\ell} (v_m - v)\|_{L^{\infty}(\hat X)} \to 0
$$
for all $r \geq 0$. Since this is true for all $r$, we proved $T_{j,\ell} v_m \to T_{j,\ell} v$ in $C^{\infty}(\hat X)$ as $m \to \infty$. A similar argument, based on (\ref{eq:Tjl*}) instead, proves that $T_{j,\ell}^* \colon C^{\infty}(\hat X) \to C^{\infty}(\hat X)$ is continuous.
\end{proof}

\begin{lem} \label{lem:Tl+1lsmoothing}
For any $\ell \geq 0$, the operator $T_{\ell+1, \ell}$ is smoothing of order 0.
\end{lem}

\begin{proof}
By the previous lemma, both $T_{\ell+1,\ell}$ and its adjoint $T_{\ell+1,\ell}^*$ map $C^{\infty}(\hat X)$ into $C^{\infty}(\hat X)$ continuously. 

Given $v \in C^{\infty}(\hat X)$, and $x \in \hat X$ not in the support of $v$, if $z$ is in a sufficiently small neighborhood of $x$, we have $z$ not in the support of $q_{\ell}^x v$. Hence if $A(x,y)$ is the kernel of $A$, then for all such $z$, we have
$$
[A q_{\ell}^x v](z) = \int_{\hat X} A(z,y) q_{\ell}(x,y) v(y) \hat m(y).
$$
It follows that one can differentiate under the integral, and obtain
$$
T_{j,\ell} v(x) = \int_{\hat X} [(\hat \nabla_b)_x^j A(x,y)] q_{\ell}(x,y) v(y) \hat m(y).
$$
Hence the kernel of $T_{j,\ell}$ is given by 
$$
T_{j,\ell} (x,y) = [(\hat \nabla_b)_x^j A(x,y)] q_{\ell}(x,y).
$$
When $j = \ell + 1$, this kernel satisfies the differential inequalities
$$
|(\hat \nabla_b)_x^{\alpha_1} (\hat \nabla_b)_y^{\alpha_2} T_{j,\ell} (x,y)| \lesssim \vartheta(x,y)^{-4-|\alpha|},
$$
since $A$ is a smoothing operator of order 1, and $q_{\ell}$ vanishes to order $\ell$ along the diagonal. It follows that the kernel of $T_{\ell+1,\ell}$ satisfies the differential inequalities of a smoothing operator of order 0.

Finally, we verify the cancellation property of $T_{\ell+1,\ell}$. To do so, suppose $\phi$ is a normalized bump function in some ball $B(x_0,r_0) \subset \hat X$. Then for any $x \in B(x_0,r_0)$ and $s \geq 0$, we have $(\hat \nabla_b)_x^s q_{\ell}^x \phi$ being $r_0^{\ell-s}$ times a normalized bump function in $B(x_0,r_0)$. Thus by cancellation property of $A$, we have 
$$
\|(\hat \nabla_b)^{j+r-s} [A (\hat \nabla_b)_x^s q_{\ell}^x \phi]\|_{L^{\infty}(B(x_0,r))} \lesssim r_0^{1-(j+r-\ell)}
$$
for all $j$, $r$ and $s$ with $j \geq 0$, $r \geq s$. In particular, evaluating at $x \in B(x_0,r_0)$, we have
$$
\left|\left.(\hat \nabla_b)_y^{j+r-s} \right|_{z=x} [A  (\hat \nabla_b)_x^s q_{\ell}^x \phi](z)\right| \lesssim r_0^{1-(j+r-\ell)}.
$$
Hence by (\ref{eq:Tjlderiv}), we see that 
$$
|(\hat \nabla_b)_x^r T_{j,\ell} \phi(x)| \lesssim r_0^{1-(j+r-\ell)}.
$$
Since this is true for all $x \in B(x_0,r_0)$, we obtain
$$
\|(\hat \nabla_b)_x^r T_{j,\ell} \phi\|_{L^{\infty}(B(x_0,r_0))} \lesssim r_0^{1-(j+r-\ell)}.
$$
When $j = \ell + 1$, this gives the desired cancellation property of order $0$ for $T_{\ell+1,\ell}$. Similarly, one can prove the desired cancellation property of $T_{\ell+1,\ell}^*$. Hence $T_{\ell+1,\ell}$ is a smoothing operator of order 0.
\end{proof}

\begin{lem} \label{lem:nonsmoothforTjl}
Suppose $v \in C^{\infty}(X) \cap L^{\infty}(\hat X)$. Then for $x \in X$, one can still define $T_{j,\ell} v(x)$ by (\ref{eq:Tjldef}), and formula (\ref{eq:Tjlderiv}) continues to hold for all $x \in X$.
\end{lem}

\begin{proof}
Suppose $v \in C^{\infty}(X) \cap L^{\infty}(\hat X)$, and $x \in X$. Let $\eta \in C^{\infty}_0(X)$ be such that $\eta \equiv 1$ in a neighborhood of $x$, and write $w_1 = \eta v$, $w_2 = (1-\eta)v$. Then $w_1 \in C^{\infty}(\hat X)$, and $T_{j,\ell}w_1(x)$ is given by (\ref{eq:Tjldef}) with $v$ replaced by $w_1$. Furthermore, $w_2$ is identically zero near $x$. Thus by pseudolocality of $A$, we have $[A q_{\ell}^x w_2](z)$ being $C^{\infty}$ for $z$ near $x$, and one can define $T_{j,\ell} w_2(x)$ by (\ref{eq:Tjldef}) with $v$ replaced by $w_2$. Since $T_{j,\ell} v = T_{j,\ell} w_1 + T_{j,\ell} w_2$, it follows that one can define $T_{j,\ell} v(x)$ by (\ref{eq:Tjldef}). 

In order to differentiate $T_{j,\ell}v$ at $x$, it suffices to differentiate $T_{j,\ell}w_1$ and $T_{j,\ell}w_2$ at $x$. One can differentiate $T_{j,\ell}w_1$ using (\ref{eq:Tjlderiv}). To differentiate $T_{j,\ell}w_2$ at $x$, note that since $x$ is not in the kernel of $w_2$, for any $z$ in a small neighborhood of $x$, we have
$$
T_{j,\ell}w_2(z) = \int_{\hat X} [(\hat \nabla_b)_z^j A(z,y)] q_{\ell}(z,y) w_2(y) \hat m(y).
$$
Differentiating under the integral using the dominated convergence theorem, one sees that the derivatives of $T_{j,\ell}w_2$ at $x$ satisfies
$$
(\hat \nabla_b)_x^r T_{j,\ell}w_2(x) = \sum_{s=0}^r \binom{r}{s} \left. (\hat \nabla_b)_z^{j+r-s} \right|_{z = x} [A (\hat \nabla_b)_x^{s} q_{\ell}^x w_2](z).
$$
Together, one concludes that the derivatives of $T_{j,\ell}v(x)$ is given by (\ref{eq:Tjlderiv}) in our current case as well.
\end{proof}

\subsection{Proof of Lemma~\ref{lem:hrhokderiv}}

We now move on to the proof of Lemma~\ref{lem:hrhokderiv}. To do so, we fix a neighborhood $U$ of $p$, and fix a frame $\hat Z_1$, $\hat Z_{\bar 1}$, $T$ of the complexified tangent bundle in $U$ as in Section~\ref{subsect:def}. Write $X_1$ and $X_2$ for the real and imaginary parts of $\hat Z_1$. One can now define normal coordinates centered at any point $x \in U$: for $y$ sufficiently close to $x$, there exists a unique $w \in \mathbb{R}^3$ such that if $\gamma(t)$ is the integral curve of $w_1 X_1 + w_2 X_2 + w_3 T$ with $\gamma(0) = x$, then $\gamma(1) = y$. In that case we write $y = x \exp(w)$, or equivalently $w = \Theta(x,y)$. Note in this case, 
$$\vartheta(x,y) \simeq |w_1|+|w_2|+|w_3|^{1/2}.$$
Also, in this normal coordinate system, $X_1$ and $X_2$ then takes the form
$$
X_1 = \frac{\partial}{\partial w_1} + 2 w_2 \frac{\partial}{\partial w_3} + O^1 \frac{\partial}{\partial w_1} + O^1 \frac{\partial}{\partial w_2} + O^2 \frac{\partial}{\partial w_3},
$$
$$
X_2 = \frac{\partial}{\partial w_2} - 2 w_1 \frac{\partial}{\partial w_3} + O^1 \frac{\partial}{\partial w_1} + O^1 \frac{\partial}{\partial w_2} + O^2 \frac{\partial}{\partial w_3}
$$
where $O^k$ are functions that vanishes to non-isotropic order $\geq k$ at $w = 0$. It follows that $(\hat \nabla_b)_x^j \Theta(x,y)^{\alpha}$ vanishes to non-isotropic order $\|\alpha\|-j$ along the diagonal $y = x$, i.e.
$$
|(\hat \nabla_b)_x^j \Theta(x,y)^{\alpha}| \lesssim \vartheta(x,y)^{\|\alpha\|-j}
$$
if $j \leq \|\alpha\|$.

Now fix $h \in \mathcal{E}(\hat \rho^1)$. For $x \in U$, let $h^x$ be the function of $w$ defined by $h^x(w) = h(x \exp(w))$. Let $P^x_k$ be the Taylor polynomial of $h^x$ at $w = 0$ up to non-isotropic order $k$. We sometimes think of $P^x_k$ as a function of $y$. Then
$$
P^x_k(y) = \sum_{\|\alpha\| \leq k} \frac{1}{\alpha!} (\partial_w^{\alpha} h^x)(0) [\Theta(x,y)]^{\alpha},
$$
where $\|\alpha\|:=|\alpha_1|+|\alpha_2|+2|\alpha_3|$ if $\alpha$ is the multi-index $(\alpha_1,\alpha_2,\alpha_3)$. One can show, by reduction to the ordinary Taylor's theorem, that 
\begin{equation} \label{eq:nonisoTaylor}
|h(y) - P^x_{k-1}(y)| \leq C \hat \rho(x)^{1-k} \vartheta(x,y)^k \quad \text{if $\vartheta(x,y) < \frac{1}{4} \hat \rho(x)$}.
\end{equation}
We will make crucial use of this in the proof of the following lemma:

\begin{lem} \label{lem:diffunderintegral}
Suppose $A$ is a smoothing operator of order 1, and $h \in \mathcal{E}(\hat \rho^1)$. If $v \in C^{\infty}(X) \cap L^{\infty}(\hat X)$, then for any $k \geq 1$, we have
$$
\left\| \hat \rho^k(x) \left. (\hat \nabla_b)_z^k \right|_{z=x} [A(h-P^x_{k-1})v](z) \right\|_{L^{\infty}(\hat m(x))} \lesssim \|v\|_{L^{\infty}(\hat m)}
$$
where $P_{k-1}^x(y)$ is defined as above.
\end{lem}

\begin{proof}
Fix $x \in X$, and write $r := \hat \rho(x) / 8$. Let $\varepsilon \in (0,r)$, and write
$$
1 = (1-\varphi_r) + (\varphi_r - \varphi_{\varepsilon}) + \varphi_{\varepsilon}
$$
where $\varphi_r$ and $\varphi_{\varepsilon}$ are normalized bump function on $B(x,2r)$ and $B(x,2\varepsilon)$ respectively, with $\varphi_r(y) \equiv 1$ on $B(x,r)$, $\varphi_{\varepsilon}(y) \equiv 1$ on $B(x,\varepsilon)$. Then
$$
h-P_{k-1}^x = (h-P_{k-1}^x) (1-\varphi_r) + (h-P_{k-1}^x) (\varphi_r - \varphi_{\varepsilon}) + (h-P_{k-1}^x) \varphi_{\varepsilon},
$$
and we estimate the contribution of each of these three terms to
$$
\hat \rho^k(x) \left. (\hat \nabla_b)_z^k \right|_{z=x} [A(h-P^x_{k-1})v](z).
$$ 
Let's call the above contributions $I$, $II$ and $III$. We will show that $I$ and $II$ are bounded by $C \|v\|_{L^{\infty}(\hat m)}$ uniformly in $x$ and $\varepsilon$, while $III$ tends to 0 as $\varepsilon \to 0$. These obviously imply the desired conclusion in our lemma.

Now note $I$ and $II$ can be computed using the kernel of $A$: we have
$$
I =  \hat \rho^k(x) \int_{\hat X} (1-\varphi_r(y)) [(\hat \nabla_b)_x^k A(x,y)] [h(y)-P_x^{k-1}(y)] v(y) \hat m(y) 
$$
and
$$
II =  \hat \rho^k(x) \int_{\hat X} (\varphi_r(y)-\varphi_{\varepsilon}(y)) [(\hat \nabla_b)_x^k A(x,y)] [h(y)-P_x^{k-1}(y)] v(y) \hat m(y).
$$
$I$ can then be estimated by breaking up $h-P_{k-1}^x$ into $h$ and the individual terms in $P_{k-1}^x$: 
the term involving $h$ only is bounded by
\begin{align*}
&\hat \rho^k(x) \int_{\vartheta(y,x) \geq r} |(\hat \nabla_b)_x^k A(x,y) h(y) v(y)| \hat m(y) \\ 
\leq  & \hat \rho^k(x) \int_{\vartheta(y,x) \geq r} \vartheta(x,y)^{-3-k} \hat m(y) \|v\|_{L^{\infty}(\hat m)} \\ \lesssim &
\begin{cases}
 r \|v\|_{L^{\infty}(\hat m)} \quad \text{if $k > 1$} \\
 r \log r \|v\|_{L^{\infty}(\hat m)} \quad \text{if $k = 1$}
\end{cases}\\
\lesssim & \|v\|_{L^{\infty}(\hat m)}.
\end{align*}
Also, the term involving $P^x_{k-1}$ can be bounded by
\begin{align*}
& \hat \rho^k(x) \int_{\vartheta(y,x) \geq r} |[(\hat \nabla_b)_x^k A(x,y)] P^x_{k-1}(y) v(y)| \hat m(y) \\
\leq & \sum_{\|\alpha\| \leq k-1} \hat \rho^k(x) \int_{\vartheta(y,x) \geq r} \vartheta(x,y)^{-3-k} \hat \rho^{1-\|\alpha\|}(x) \vartheta(x,y)^{\|\alpha\|} \hat m(y) \|v\|_{L^{\infty}(\hat m)}.
\end{align*}
The terms when $\|\alpha\|= k-1$ can be bounded by 
$$
\hat \rho^2(x) \int_{\vartheta(y,x) \geq r} \vartheta(x,y)^{-4} \hat m(y) \|v\|_{L^{\infty}(\hat m)} \lesssim r^2 \log r \|v\|_{L^{\infty}(\hat m)} \lesssim \|v\|_{L^{\infty}(\hat m)}.
$$
The terms when $\|\alpha\| < k-1$ can be bounded by
\begin{align*}
&\hat \rho(x)^{k+1-\|\alpha\|} \int_{\vartheta(x,y) \geq r} \vartheta(x,y)^{-4+\|\alpha\|-(k-1)} \hat m(y) \|v\|_{L^{\infty}(\hat m)} \\
\lesssim & \rho(x)^{k+1-\|\alpha\|} r^{\|\alpha\|-(k-1)} \|v\|_{L^{\infty}(\hat m)}  \\
\lesssim & r^2 \|v\|_{L^{\infty}(\hat m)}  \lesssim  \|v\|_{L^{\infty}(\hat m)}.
\end{align*}
This shows that $|I| \lesssim \|v\|_{L^{\infty}(\hat m)}$ uniformly in $x$, as desired.

Next, $II$ can be bounded using (\ref{eq:nonisoTaylor}), uniformly in $x$ and $\varepsilon$:
\begin{align*}
& \hat \rho^k(x) \int_{\vartheta(y,x) < r} |[(\hat \nabla_b)_x^k A(x,y)] (h-P^x_{k-1})(y) v(y)| \hat m(y) \\ 
\leq &\hat \rho^k(x) \int_{\vartheta(y,x) < r} \vartheta(x,y)^{-3-k} r^{1-k} \vartheta(x,y)^k \hat m(y) \|v\|_{L^{\infty}(\hat m)} \\
\lesssim & r^2 \|v\|_{L^{\infty}(\hat m)} \lesssim \|v\|_{L^{\infty}(\hat m)}.
\end{align*}
Thus it remains to show that $III$ tends to 0 as $\varepsilon \to 0$.

To do so, note that there is some constant $C_v$ (possibly depending on many derivatives of $v$) such that $C_v^{-1} \varepsilon^{-k} r^{k} \varphi_{\varepsilon} (h-P^x_{k-1}) v$ is a normalized bump function in $B(x,2\varepsilon)$. Thus since $A$ is smoothing of order 1, by the cancellation conditions of $A$, we have
$$
|III| \leq r^k \left\| (\hat \nabla_b)^k [A \varphi_{\varepsilon} (h-P^x_{k-1}) v] \right\|_{L^{\infty}(B(x,2\varepsilon))} \leq C_v \varepsilon,
$$
which tends to 0 as $\varepsilon \to 0$. This completes the proof of the current lemma.
\end{proof}

We are now ready to prove Lemma~\ref{lem:hrhokderiv}. In fact we will prove the following slight generalization:

\begin{lem} \label{lem:hrhokderiv2}
For any non-negative integer $k$, if $A_1$, $A_2$, $\dots$, $A_{k+1}$ are smoothing operators of order $1$, and $h \in \mathcal{E}(\hat \rho^1)$ is supported in a sufficiently small neighborhood of $p$, then for any function $u \in C^{\infty}(X) \cap L^{\infty}(\hat X)$, we have 
$$ 
\|\hat \rho^k \hat \nabla_b^k (A_{k+1}h A_k h \dots A_1 h) u\|_{L^{\infty}(\hat m)} \lesssim_k \|u\|_{L^{\infty}(\hat m)}.
$$ 
\end{lem}

For simplicity, we will write below $S_k$ for the operator $A_k h \dots A_1 h$. We then need to bound $\|\hat \rho^k \hat \nabla_b^k S_{k+1} u \|_{L^{\infty}(\hat m)}$ for $u \in C^{\infty}(X) \cap L^{\infty}(\hat X)$.

\begin{proof}
We proceed by induction on $k$. When $k = 0$, one just need to notice that $\|A_1 h u\|_{L^{\infty}(\hat m)} \lesssim \|u\|_{L^{\infty}(\hat m)}$, which holds since $A_1$ and $h$ both preserves $L^{\infty}(\hat m)$. Suppose now the proposition has been proved up to $k-1$ for some positive integer $k$. In other words, we assume
\begin{equation} \label{eq:inducthypomain}
\|\hat \rho^i \hat \nabla_b^i S_{i+1} u\|_{L^{\infty}(\hat m)} \lesssim_i \|u\|_{L^{\infty}(\hat m)}.
\end{equation}
for all $0 \leq i \leq k-1$. Let $u \in C^{\infty}(X) \cap L^{\infty}(\hat X)$, and write 
\begin{align*}
\hat \rho^k \hat \nabla_b^k S_{k+1} u(x) = \hat \rho^k \hat \nabla_b^k A_{k+1} h S_k u(x) = V_1(x) + V_2(x)
\end{align*}
where
$$
V_1(x) = \hat \rho^k(x) \left. (\hat \nabla_b)_z^k \right|_{z=x} [A_{k+1} (h-P^x_{k-1}) S_k u](z)
$$
and
$$
V_2(x) = \hat \rho^k(x) \left. (\hat \nabla_b)_z^k \right|_{z=x} [A_{k+1} P^x_{k-1} S_k u](z).
$$
We want to show that both $V_1(x)$ and $V_2(x)$ are bounded by $C \|u\|_{L^{\infty}(\hat m)}$. 

First $V_1(x)$ is like an error term, that can be estimated by Lemma~\ref{lem:diffunderintegral}. In fact, applying it to $v = S_k u$, which is in $C^{\infty}(X) \cap L^{\infty}(\hat X)$, we have
$$ 
\left\|\hat \rho^k(x) \left. (\hat \nabla_b)_z^k \right|_{z = x} [A_{k+1} (h-P^x_{k-1}) (S_k u)](z) \right\|_{L^{\infty}(\hat m (x))} \leq C \|S_k u\|_{L^{\infty}(\hat m)},
$$ 
and the latter is bounded trivially by $C \|u\|_{L^{\infty}(\hat m)}$. Thus it remains to estimate the main term $V_2(x)$.

Now write $\Theta_x(y) = \Theta(x,y)$. Then
\begin{align*}
V_2(x) 
= & \sum_{\|\alpha\| \leq k-1} \frac{1}{\alpha!} (\partial_w^{\alpha} h^x)(0) \hat \rho^k(x) \left. (\hat \nabla_b)_z^k \right|_{z=x} [A_{k+1} \Theta_x^{\alpha} S_k u](z).
\end{align*}
Note that on the right hand side of this sum, $\left. (\hat \nabla_b)_y^k \right|_{z=x} [A_{k+1} \Theta_x^{\alpha} S_k u](z)$ is of the form $T_{k,\ell} S_k u$ if $\ell = \|\alpha\|$, where $T_{k,\ell}$ is defined as in (\ref{eq:Tjldef}), with $A = A_{k+1}$ and $q_{\ell}(x,y) = \Theta(x,y)^{\alpha}$. Thus 
$$
|V_2(x)| \lesssim \sum_{\ell = 0}^{k-1} \hat \rho^{k-\ell+1}(x) |T_{k,\ell} (S_k u)(x)|.
$$
We will prove, by induction, that when $1 \leq j \leq k$, and $0 \leq \ell \leq j-1$, if $T_{j,\ell}$ is any operator of the form (\ref{eq:Tjldef}), we have
\begin{equation} \label{eq:inducthypojl}
\left\| \hat \rho^{j-\ell-1}(x) T_{j,\ell}(S_k u)(x)  \right\|_{L^{\infty}(\hat m(x))} \lesssim \|u\|_{L^{\infty}(\hat m)},
\end{equation}
and that the bound is uniform in the choice of $A$ and $q_{\ell}$, as long as they are suitably normalized. Assume this for the moment. Then applying it to the case $j = k$, with $T_{k,\ell}$ defined using $A = A_{k+1}$ and $q_{\ell}(x,y) = \Theta(x,y)^{\alpha}$, we have
$$
|V_2(x)| \lesssim \sum_{\ell = 0}^{k-1} \hat \rho^{2}(x) \|u\|_{L^{\infty}(\hat m)} \lesssim \|u\|_{L^{\infty}(\hat m)}.
$$
This will complete the proof of the current lemma.

It remains to prove (\ref{eq:inducthypojl}). To do so, we proceed by induction on $j$. First, when $j = 1$ (hence $\ell = 0$), note that $T_{1,0}$ is a smoothing operator of order 0 by Lemma~\ref{lem:Tl+1lsmoothing}. Hence 
$$
T_{1,0} S_k u = (T_{1,0} A_k) h S_{k-1} u = \tilde{A} h S_{k-1} = \tilde{S}_1 S_{k-1} u,
$$ 
where $\tilde{A} := T_{1,0} A_k$ is smoothing of order 1 by Theorem~\ref{thm2.2}, and $\tilde{S}_1 = \tilde{A} h$. By our induction hypothesis (\ref{eq:inducthypomain}) with $i = 0$, we see that 
$$
\|T_{1,0} S_k u(x)\|_{L^{\infty}(\hat m(x))} = \|\tilde{S}_1 [S_{k-1} u](x)\|_{L^{\infty}(\hat m(x))} \lesssim \|S_{k-1} u\|_{L^{\infty}(\hat m)} \lesssim \|u\|_{L^{\infty}(\hat m)}.
$$
Hence (\ref{eq:inducthypojl}) holds when $j = 1$. 

Next, assume that (\ref{eq:inducthypojl}) has been proved up to $j-1$, for some $2 \leq j \leq k$. In other words, we fix $j$ with $2 \leq j \leq k$, and assume that we have verified already
\begin{equation} \label{eq:inducthypojlassumed}
\left\| \hat \rho^{a-b-1}(x) T_{a,b}(S_k u)(x)  \right\|_{L^{\infty}(\hat m(x))} \lesssim \|u\|_{L^{\infty}(\hat m)}
\end{equation}
for all $a$ and $b$ with $1 \leq a \leq j-1$ and $0 \leq b \leq a-1$. We want to prove the same statement when $a = j$. So we fix $\ell$ with $0 \leq \ell \leq j-1$. First, $S_k u \in C^{\infty}(X) \cap L^{\infty}(\hat X)$ since $u$ is as such. Thus we can apply Lemma~\ref{lem:nonsmoothforTjl} and (\ref{eq:Tjlderiv}) to $v = S_k u$: in fact, if $T_{j,\ell}$ is defined so that $T_{j,\ell}v(x) = \left. (\hat \nabla_b)_z^j \right|_{z = x} [A q_{\ell}^x v](z)$, then by (\ref{eq:Tjlderiv}), we have
\begin{align} \notag
T_{j,\ell} (S_k u)(x) =& \hat \nabla_b^{j-\ell-1} [T_{\ell+1,\ell} (S_k u)](x)\\
&\quad - \sum_{s = 1}^{j-\ell-1} \binom{j-\ell-1}{s} \left. (\hat \nabla_b)_z^{j-s} \right|_{z = x} [A (\hat \nabla_b)_x^s q_{\ell}^x S_k u](z). \label{eq:Tjlest}
\end{align}
The first term on the right hand side can then be written as
$$
\hat \nabla_b^{j-\ell-1} [\tilde{A} h S_{k-1} u](x) = \hat \nabla_b^{j-\ell-1} [\tilde{A} h S_{j-\ell-1}] (S_{k-j+\ell} u)(x) 
$$
where $\tilde{A} := T_{\ell+1,\ell} A_k$ is smoothing of order 0 by Lemma~\ref{lem:Tl+1lsmoothing} and Theorem~\ref{thm2.2}. Hence by our induction hypothesis (\ref{eq:inducthypomain}), with $i = j-\ell-1$ (note $0 \leq i \leq k-1$ under our assumptions so (\ref{eq:inducthypomain}) applies), we have 
\begin{equation} \label{eq:rhoj-1main2}
|\hat \rho^{j-\ell-1}(x) \hat \nabla_b^{j-\ell-1} [T_{\ell+1,\ell} (S_k u)](x)| \lesssim \|S_{k-j+\ell} u\|_{L^{\infty}(\hat m)} \lesssim \|u\|_{L^{\infty}(\hat m)}.
\end{equation}
To bound the rest of the sum on the right hand side of (\ref{eq:Tjlest}), note that $\left. (\hat \nabla_b)_z^{j-s} \right|_{z = x} [A (\hat \nabla_b)_x^s q_{\ell}^x S_k u](z)$ is of the form $T_{j-s,\max\{\ell-s,0\}} S_k u(x)$. This is because $(\hat \nabla_b)_x^s q_{\ell}(x,y)$ is smooth on $\hat X \times \hat X$, and vanishes on the diagonal to order $\max\{\ell-s,0\}$. Hence 
\begin{align}
& \left| \hat \rho^{j-\ell-1}(x) \sum_{s=1}^{j-\ell-1} \binom{j-\ell-1}{s} \left. (\hat \nabla_b)_z^{j-s} \right|_{z = x} [A (\hat \nabla_b)_x^s q_{\ell}^x S_k u](z) \right| \notag \\
\lesssim & \sum_{s=1}^{j-\ell-1} \hat \rho^{\max\{s-\ell,0\}}(x) \hat \rho^{j-s-\max\{\ell-s,0\}-1}(x) |T_{j-s,\max\{\ell-s,0\}} S_k u(x)| \label{eq:rhoj-1error2} \\
\lesssim & \sum_{s=1}^{j-\ell-1} \hat \rho^{\max\{s-\ell,0\}}(x) \|u\|_{L^{\infty}(\hat m)} \lesssim \|u\|_{L^{\infty}(\hat m)}, \notag
\end{align} 
the second inequality following from our induction hypothesis (\ref{eq:inducthypojlassumed}) with $a = j-s$, $b = \max\{\ell-s,0\}$ (note that then $1 \leq a \leq j-1$ and $0 \leq b \leq  j-s-1$ in the above sum over $s$, so (\ref{eq:inducthypojlassumed}) applies for this $a$ and $b$). Combining (\ref{eq:rhoj-1main2}) and (\ref{eq:rhoj-1error2}), we see that
$$
|\rho^{j-\ell-1}(x) T_{j,\ell} S_k u (x)| \lesssim \|u\|_{L^{\infty}(\hat m)}
$$
uniformly in $x$, as desired. This completes our proof of (\ref{eq:inducthypojl}), and hence the proof of the current lemma.
\end{proof}

\section{Construction of the CR function $\psi$} \label{sect:CR}

The goal in this section is to show that there exists a function $\psi\in C^\infty(\hat X)$ such that $\hat\ddbar_b\psi=0$ and near $p$, we have $\psi=\abs{z}^2+it+R$, $R=\varepsilon(\hat\rho^4)$ (see Theorem~\ref{t-ysmipc}). It is via this $\psi$ that we reduce our problem from the non-compact manifold $X$ to the compact manifold $\hat X$, as was explained in the introduction. 

Until further notice, we work in some small neighbourhood of $p$. Put 
\[\ol Z^0_1=\frac{\pr}{\pr\ol z}+iz\frac{\pr}{\pr t}.\]
First, we need 

\begin{lem}\label{l-ysmipcI}
For any monomial $a\ol z^{\alpha}z^{\beta}t^\gamma$, $a\in\Complex$, $\alpha, \beta, \gamma\in\mathbb N_0:=\mathbb N\bigcup\set{0}$, with $\alpha+\beta+2\gamma\geq 3$, we can find a polynomial $f=\sum^m_{j=1}d_j\ol z^{\alpha_j}z^{\beta_j}t^{\gamma_j}$, $\alpha_j, \beta_j, \gamma_j\in\mathbb N_0$, $d_j\in\Complex$, $j=1,\ldots,m$, such that $\ol Z^0_1f=az^{\alpha}\ol z^{\beta}t^{\gamma}$ and
\begin{equation}\label{e2-ysmipcI} 
\begin{split}
&\mbox{$d_j=0$ if $\alpha_j+\beta_j+2\gamma_j\neq\alpha+\beta+2\gamma+1$},\\
&\mbox{$\abs{{\rm Re\,}f(z,t)}\leq c\abs{z}^2(\abs{z}+\abs{t})$ in a neighbourhood of $p$,}\\
&\mbox{where $c>0$ is a constant}.
\end{split}
\end{equation}
\end{lem} 

\begin{proof} 
We proceed by induction over $\gamma$. First we assume that $\gamma=0$. Given a monomial $a\ol z^{\alpha}z^{\beta}$, $a\in\Complex$, $\alpha, \beta\in\mathbb N_0$, with $\alpha+\beta\geq 3$. Put $f=\frac{a}{\alpha+1}\ol z^{\alpha+1}z^\beta$. It is easy to see that $\ol Z^0_1f=a\ol z^{\alpha}z^{\beta}$ and \eqref{e2-ysmipcI} hold.  Let $\gamma\geq1$. Given a monomial $a\ol z^{\alpha}z^{\beta}t^\gamma$, $a\in\Complex$, $\alpha, \beta,\gamma\in\mathbb N_0$, with $\alpha+\beta+2\gamma\geq 3$. First, we assume that $\alpha=\beta=0$. By the induction assumption,
we can find $f_1=\sum^m_{j=1}d_j\ol z^{\alpha_j}z^{\beta_j}t^{\gamma_j}$, $\alpha_j, \beta_j, \gamma_j\in\mathbb N_0$, $d_j\in\Complex$, $j=1,\ldots,m$, such that 
\[\ol Z^0_1f_1=-ia\gamma\abs{z}^2t^{\gamma-1}+i\ol a\gamma z^2t^{\gamma-1}\] 
and \eqref{e2-ysmipcI} hold. Put 
\[f=a\ol zt^\gamma-\ol azt^\gamma+f_1.\] 
It is not difficult to check that $\ol Z^0_1f=at^\gamma$ and \eqref{e2-ysmipcI} hold. Now, we assume that $\alpha+\beta\geq1$. 
By the induction assumption,
we can find $f_1=\sum^m_{j=1}d_j\ol z^{\alpha_j}z^{\beta_j}t^{\gamma_j}$, $\alpha_j, \beta_j, \gamma_j\in\mathbb N_0$, $d_j\in\Complex$, $j=1,\ldots,m$, such that 
\[\ol Z^0_1f_1=-i\frac{a\gamma}{\alpha+1}\ol z^{\alpha+1}z^{\beta+1}t^{\gamma-1}\] 
and \eqref{e2-ysmipcI} hold. Put 
\[f=\frac{a}{\alpha+1}\ol z^{\alpha+1}z^\beta t^{\gamma}+f_1.\] 
It is not difficult to check that $\ol Z^0_1f=a\ol z^\alpha z^\beta t^\gamma$ and \eqref{e2-ysmipcI} hold.

The lemma follows.
\end{proof}

We say that $g$ is a quasi-homogeneous polynomial of degree $d\in\mathbb N_0$ if $g$ is the finite sum \[g=\sum_{\alpha+\beta+2\gamma=d,\alpha,\beta,\gamma\in\mathbb N_0}c_{\alpha,\beta,\gamma}\ol z^\alpha z^\beta t^\gamma,\quad c_{\alpha,\beta,\gamma}\in\Complex.\]
We need 

\begin{prop}\label{p-ysmipcI}
There exists a function $\varphi\in C^\infty(\hat X)$ such that ${\rm Re\,}\varphi\geq0$ on $\hat X$, $\hat\ddbar_b\varphi$ vanishes to infinite order at $p$ and near $p$, we have
\[\varphi(z,t)=(2\pi)(\abs{z}^2+it)+S,\ \ S=\varepsilon(\hat\rho^6).\]
\end{prop}

\begin{proof}
We assume that the local coordinates $(z,t)$ defined on a small open set $W\subset X$ of $p$. 
From Lemma~\ref{l-ysmipcI} and \eqref{s1-e5b}, it is not difficult to see that we can find $f_j(z,t)$, $j=6,7,\ldots$, 
where for each $j$, $f_j$ is a quasi-homogeneous polynomial of degree $j$, 
such that 
\begin{equation}\label{e2-ysmiIb}
\hat Z_{\ol 1}\bigr((2\pi)(\abs{z}^2+it)+\sum^m_{j=6}f_j(z,t)\bigr)\in\varepsilon(\hat\rho^{m+3}),\ \ m=6,7,\ldots,
\end{equation}
and for each $j=6,7,\ldots$, 
\begin{equation}\label{e2-ysmiI}
\mbox{$\abs{{\rm Re\,} f_j(z,t)}\leq c_j\abs{z}^2(\abs{z}+\abs{t})$ on $W_j\subset W$, $c_j>0$ is a constant},
\end{equation} 
where $W_j$ is an open set, for each $j=6,7,\ldots$. 
Take $\phi(z,t)\in C^\infty_0(\Complex\times\Real,\ol\Real_+)$ so that $\phi(z,t)=1$ if $\abs{z^2}+\abs{t}\leq\frac{1}{2}$ 
and $\phi(z,t)=0$ if $\abs{z^2}+\abs{t}\geq1$. For each $j=6,7,\ldots$, take $\epsilon_j>0$ be a small constant ($\epsilon_j \sim 2^{-j}$ will do) so that 
${\rm Supp\,}\phi(\frac{z}{\epsilon_j},\frac{t}{\epsilon^2_j})\subset W_j$, 
\begin{equation}\label{e2-ysmiII}
\left| \phi(\frac{z}{\epsilon_j},\frac{t}{\epsilon^2_j}){\rm Re\,}f_j(z,t) \right| <2^{-j}\abs{z}^2
\end{equation}
and for all $\alpha,\beta,\gamma\in\mathbb N_0$, $\alpha+\beta+2\gamma<j$, we have 
\begin{equation}\label{e2-ysmiIII}
\left\| \pr^\alpha_{\ol z}\pr^\beta_z\pr^\gamma_t\bigr(\phi(\frac{z}{\epsilon_j},\frac{t}{\epsilon^2_j})f_j(z,t)\bigr) \right\|_{L^{\infty}} <2^{-j}.
\end{equation} 
On $W$, we put
\[\varphi_1(z,t)=(2\pi)(\abs{z}^2+it)+\sum^{\infty}_{j=6}\phi(\frac{z}{\epsilon_j},\frac{t}{\epsilon^2_j})f_j(z,t).\] 
From \eqref{e2-ysmiIII}, we can check that $\varphi_1(z,t)$ is well-defined as a smooth function on $W$ and 
for all $\alpha,\beta,\gamma\in\mathbb N_0$, $\alpha+\beta+2\gamma=d$, $d\geq6$, we have 
\[\pr^\alpha_{\ol z}\pr^\beta_z\pr^\gamma_t\varphi_1|_{(0,0)}=\pr^\alpha_{\ol z}\pr^\beta_z\pr^\gamma_tf_d|_{(0,0)}.\]
Combining this with \eqref{e2-ysmiIb}, we conclude that $\hat\ddbar_b\varphi_1$ vanishes to infinite order 
at $p$. Moreover, from \eqref{e2-ysmiII}, we have 
\[{\rm Re\,}\varphi_1(z,t)\geq\abs{z^2}(2\pi-\sum^{\infty}_{j=6}2^{-j})>\frac{1}{2}\abs{z}^2.\]
Thus, ${\rm Re\,}\varphi_1\geq0$ on $W$. Take $\chi\in C^\infty_0(W)$, $\chi\geq0$ and $\chi=1$ near $p$ and put 
$\varphi=\chi\varphi_1\in C^\infty(X)$. Then, $\varphi$ satisfies the claim of this proposition. The proposition follows.
\end{proof} 

Let $W$ be a small neighbourhood of $p$ such that $\abs{\varphi(z,t)}\geq\hat\rho(z,t)^2$ on $W$, where $\varphi$ is as 
in Proposition~\ref{p-ysmipcI}. Take $\chi\in C^\infty_0(W,\ol\Real_+)$ so that $\chi=1$ in some small 
neighbourhood of $p$. Put 
\[\tau:=\frac{\chi}{\varphi}.\]
It is easy to check that $\tau$ is well-defined as an element in $\mathscr D'(\hat X)$. Since $\hat\ddbar_b\varphi$ vanishes to infinite order at $p$, we have $\hat\ddbar_b\tau\in C^\infty(\hat X)$.
Thus, 
\[\hat\Box_b\tau\in C^\infty(\hat X).\]
Now 
\begin{equation}\label{eq:dual}
I = \hat \Pi + \hat N \hat \Box_b \quad \text{on $\mathcal{D}'(\hat X)$},
\end{equation} 
since $I = \hat \Pi + \hat \Box_b \hat N$ on $L^2(\hat m)$, and since the operators involved are self-adjoint and pseudolocal. Hence we obtain
\begin{equation}\label{e-ysmiI}
\tau=\hat\Pi\tau+\hat N\hat\Box_b\tau=\hat\Pi\tau-F,
\end{equation}
where $F\in C^\infty(\hat X)$. Thus, 
\begin{equation}\label{e-ysmiII}
\hat\ddbar_b(\tau+F)=0\ \ \mbox{on $\mathscr D'(\hat X)$}.
\end{equation} 
Take $C_0>0$ be a large constant so that ${\rm Re\,}F+C_0>0$ on $\hat X$. Since ${\rm Re\,}\varphi\geq0$, we conclude that 
\begin{equation}\label{e-ysmiIII}
{\rm Re\,}(\tau+F+C_0)>0\ \ \mbox{on $\hat X$}.
\end{equation} 
Put 
\[\psi=\frac{1}{\tau+F+C_0}.\] 
From \eqref{e-ysmiIII}, we know that $\tau+F+C_0\neq0$ on $X$. Since $\tau+F+C_0\in C^\infty(X)$, we conclude that $\psi\in C^\infty(X)$. Now, we study the behaviour of $\psi$ near $p$. Let $W'\Subset W$ be a small neighbourhood of $p$ such that $\chi=1$ on $W'$ and $\abs{\varphi(F+C_0)}<1$ on $W'$. Then, on $W'$, 
\begin{equation}\label{e-ysmiIV}
\psi=\frac{1}{\tau+F+C_0}=\frac{1}{\frac{1}{\varphi}+F+C_0}=\frac{\varphi}{1+\varphi(F+C_0)}\in C^\infty(W').
\end{equation} 
Thus, 
\begin{equation}\label{e-ysmiV}
\psi\in C^\infty(\hat X).
\end{equation}
Moreover, from \eqref{e-ysmiIV}, we can check that near $p$, 
\begin{equation}\label{e-ysmiVI}
\psi=(2\pi)(\abs{z}^2+it)+R,\ \ R\in\varepsilon(\hat\rho^4). 
\end{equation} 

\begin{lem} \label{l-ysmipcII}
We have 
\[\hat\ddbar_b\psi=0\ \ \mbox{on $\hat X$}.\]
\end{lem}

\begin{proof}
Put $h:=\tau+F+C_0$, so that $\psi = h^{-1}$, and take any $g\in\Omega^{0,1}_0(X)$. We have 
\begin{equation}\label{e-ysmiVII}
(\hat\ddbar_b\psi\ |\ g)_{\hat m,\hat\theta}=-(\hat\ddbar_bh\ |\ \ol h^{-2}g)_{\hat m,\hat\theta}=0
\end{equation}
since $\hat\ddbar_bh=0$ in the sense of distribution. Thus, $\hat\ddbar_b\psi=0$ on $X$. Since $\psi\in C^\infty(\hat X)$, we conclude that $\hat\ddbar_b\psi=0$ on $\hat X$. The lemma follows.
\end{proof}

From \eqref{e-ysmiIII}, \eqref{e-ysmiV}, \eqref{e-ysmiVI} and Lemma~\ref{l-ysmipcII}, we obtain the main result of this section 

\begin{thm}\label{t-ysmipc}
There is a smooth function $\psi\in C^\infty(\hat X)$ such that $\hat\ddbar_b\psi=0$ on $\hat X$, $\psi\neq0$ on $X$, ${\rm Re\,}\psi\geq0$ on $\hat X$ and near $p$, we have 
\begin{equation}\label{e-ysmiVIII}
\psi(z,t)=(2\pi)(\abs{z}^2+it)+R,\ \ R=\varepsilon(\hat\rho^4).
\end{equation}
\end{thm} 

In the study of the positive $p$-mass theorem (see~\cite{CMY}), one needs to find a special $CR$ function of specific growth rate on $\hat X$. More precisely, in \cite{CMY}, one needs to find a CR function $g\in C^\infty(X)$ with 
\[g=\frac{z}{\abs{z}^2+it}+g_1\]
near $p$, where $g_1\in\mathcal{E}(\hat\rho^0)$. By using the proof of Theorem~\ref{t-ysmipc}, we can construct a such CR function

\begin{thm} \label{t-crho}
There is a function $g\in C^\infty(X)\bigcap\mathscr D'(\hat X)$ such that $\hat\ddbar_bg=0$ and 
\[g=\frac{1}{\psi}(z+r),\]
where $\psi$ is as in Theorem~\ref{t-ysmipc} and $r\in C^\infty(\hat X)$, $r=\varepsilon(\hat\rho^2)$.
\end{thm} 

\begin{proof}
We can repeat the proof of Proposition~\ref{p-ysmipcI} with minor change and conclude that there is a function
$\Td r\in C^\infty(\hat X)$ with $\Td r=\varepsilon(\hat\rho^5)$ such that $\hat\ddbar_b(z+\tilde r)$ vanishes to infinite order at $p$. Put 
\[\Td g=\frac{1}{\psi}(z+\Td r).\] 
Since $\hat\ddbar_b(z+\Td r)$ vanishes to infinite order at $p$ and $\hat\ddbar_b\psi=0$, we have 
\begin{equation}\label{e-crhoI}
\hat\ddbar_b\Td g\in C^\infty(\hat X).
\end{equation} 
Thus, 
\[\hat\Box_b\Td g\in C^\infty(\hat X).\]
By (\ref{eq:dual}), we obtain
\[\Td g=\hat\Pi\Td g+\hat N\hat\Box_b\Td g=\hat\Pi\Td g-f,\]
where $f\in C^\infty(\hat X)$. Thus, 
\[\hat\ddbar_b(\Td g+f)=0\ \ \mbox{on $\mathscr D'(\hat X)$}.\]
Put $g=\Td g+f$. Then $\hat\ddbar_b g$=0 and we can check that $g=\frac{1}{\psi}(z+\Td r+\psi f)=\frac{1}{\psi}(z+r)$, where 
$r=\Td r+\psi f\in C^\infty(\hat X)$, $r=\varepsilon(\hat\rho^2)$. The theorem follows. 
\end{proof} 

\section{The relation between $\Box_{b,1}$ and $\Td\Box_b$}  \label{sect5}

Having constructed our CR function $\psi$, we can proceed as in Section \ref{subsect:outlineofproof}, and constuct the intermediate Kohn Laplacian $\Td \Box_b$.  We refer the reader to that section for the details of this construction. The goal of the current section is to reduce the study of our operator of interest, namely $\Box_{b,1}$, to the study of $\Td\Box_b$. This is done by establishing (\ref{eq:Boxb1totildeBoxb1}) and (\ref{eq:Boxb1totildeBoxb2}) in Section \ref{subsect:outlineofproof}.

First, we recall that
$
m_1 =\abs{\psi}^{-2}\Td m.
$
Thus,
\begin{equation}\label{e-crysmiII}
\begin{split}
&\mbox{$u\in L^2(m_1)$ if and only if $\psi^{-1}u\in L^2(\Td m)$},\\
&\mbox{$u\in L^2_{(0,1)}(m_1,\hat\theta)$ if and only if $\psi^{-1}u\in L^2_{(0,1)}(\Td m,\theta)$},
\end{split}
\end{equation}
and 
\begin{equation}\label{e-crysmiIII}
\begin{split}
&\norm{u}_{m_1}=\norm{\psi^{-1}u}_{\Td m}, \ \ \norm{v}_{m_1,\hat\theta}=\norm{\psi^{-1}v}_{\Td m,\hat\theta}, \forall u\in L^2(m_1), v\in L^2_{(0,1)}(m_1,\hat\theta), \\
&(u_1\ |\ u_2)_{m_1}=(\psi^{-1}u_1\ |\ \psi^{-1}u_2)_{\Td m},\ \ \forall u_1, u_2\in L^2(m_1),\\
&(v_1\ |\ v_2)_{m_1,\hat\theta}=(\psi^{-1}v_1\ |\ \psi^{-1}v_2)_{\Td m,\hat\theta},\ \ \forall v_1, v_2\in L^2_{(0,1)}(m_1,\hat\theta).
\end{split}
\end{equation} 
Let
\begin{equation}\label{e-lopcI}
\ol{\pr}^{*,f}_{b,1}:\Omega^{0,1}(X)\To C^\infty(X)
\end{equation}
be the formal adjoint of $\ddbar_b$ with respect to $(\,\cdot\,|\,\cdot\,)_{m_1}$, $(\,\cdot\,|\,\cdot\,)_{m_1,\hat\theta}$. Let also 
\begin{equation}\label{e-lopcIb}
\Td{\ol{\pr}}^{*,f}_{b}:\Omega^{0,1}(X)\To C^\infty(X)
\end{equation}
be the formal adjoint of $\ddbar_b$ with respect to $(\,\cdot\,|\,\cdot\,)_{\Td m}$, $(\,\cdot\,|\,\cdot\,)_{\Td m,\hat\theta}$. Then

\begin{lem}\label{l-lopcI}
For $v\in\Omega^{0,1}(X)$, we have 
\begin{equation}\label{e-lopcII}
\ol{\pr}^{*,f}_{b,1}v=\psi\Td{\ol{\pr}}^{*,f}_{b}(\psi^{-1}v).
\end{equation}
\end{lem}

\begin{proof}
Let $h\in C^\infty_0(X)$, $g\in\Omega^{0,1}_0(X)$. We have 
\[\begin{split}
(\ddbar_bh\,|\,g)_{m_1,\hat\theta}&=(\ddbar_bh\,|\,g\abs{\psi}^{-2})_{\Td m,\hat\theta}=(\ddbar_b(\psi^{-1}h)\,|\,\psi^{-1} g)_{\Td m,\hat\theta}\\
&=(\psi^{-1}h\,|\,\Td{\ol{\pr}}^{*,f}_{b}(\psi^{-1}g))_{\Td m}=(h\,|\,\psi\Td{\ol{\pr}}^{*,f}_{b}(\psi^{-1}g))_{m_1,\hat\theta}.
\end{split}\]
\eqref{e-lopcII} follows.
\end{proof}

The next lemma clarifies the relation between $\ddbar_{b,1}$ and $\Td\ddbar_{b}$:

\begin{lem}\label{l-lopcII}
We have 
\begin{equation}\label{e-lopcIII}
\mbox{$u\in{\rm Dom\,}\ddbar_{b,1}$ if and only if $\psi^{-1}u\in{\rm Dom\,}\Td\ddbar_{b}$}.
\end{equation} 

Moreover, 
\begin{equation}\label{e-lopcIIIb}
\ddbar_{b,1}u=\psi\Td\ddbar_b(\psi^{-1}u),\ \ \forall u\in{\rm Dom\,}\ddbar_{b,1}.
\end{equation}
\end{lem}

\begin{proof} 
Let $u\in{\rm Dom\,}\ddbar_{b,1}$. Then, there is a $h\in L^2_{(0,1)}(m_1,\hat\theta)$ such that 
\begin{equation}\label{e-lopcIV}
(h\ |\ \alpha)_{m_1,\hat\theta}=(u\ |\ \ol{\pr}^{*,f}_{b,1}\alpha)_{m_1},\ \ \forall \alpha\in\Omega^{0,1}_0(X).
\end{equation}
Note that $\ddbar_{b,1}u:=h$. From \eqref{e-lopcII} and \eqref{e-lopcIV}, it is easy to see that
\begin{equation}\label{e-lopcV}
\begin{split}
(\psi^{-1}h\ |\ g)_{\Td m,\hat\theta}&=(h\ |\ \psi g)_{m_1,\hat\theta}=(u\ |\ \ol{\pr}^{*,f}_{b,1}(\psi g))_{m_1}\\
&=(\psi^{-1}u\ |\ \psi^{-1}\ol{\pr}^{*,f}_{b,1}(\psi g))_{\Td m}\\
&=(\psi^{-1}u\ |\ \Td{\ol{\pr}}^{*,f}_bg)_{\Td m},\ \ \forall g\in\Omega^{0,1}_0(X).
\end{split}
\end{equation} 
Since $\psi^{-1}h$ in $L^2_{(0,1)}(\Td m,\hat\theta)$, from \eqref{e-lopcV}, we conclude that $\psi^{-1}u\in{\rm Dom\,}\Td\ddbar_b$ and 
\[\Td\ddbar_b(\psi^{-1}u)=\psi^{-1}h=\psi^{-1}\ddbar_{b,1}u.\]

We have proved that if $u\in{\rm Dom\,}\ddbar_{b,1}$ then $\psi^{-1}u\in{\rm Dom\,}\Td\ddbar_{b}$ and \[\psi\Td\ddbar_b(\psi^{-1}u)=\ddbar_{b,1}u.\]

We can repeat the procedure above and conclude that if $v\in{\rm Dom\,}\Td\ddbar_{b}$ then $\psi v\in{\rm Dom\,}\ddbar_{b,1}$ and $\psi^{-1}\ddbar_{b,1}(\psi v)=\Td\ddbar_bv$. The lemma follows.
\end{proof} 

We have a corresponding lemma about the relation between $\ddbar_{b,1}^*$ and $\Td\ddbar_{b}^*$:

\begin{lem}\label{l-lopcIII}
We have 
\begin{equation}\label{e-hoamI}
\mbox{$u\in{\rm Dom\,}\ol{\pr}^*_{b,1}$ if and only if $\psi^{-1}u\in{\rm Dom\,}\Td{\ol{\pr}}^*_{b}$}.
\end{equation} 

Moreover, 
\begin{equation}\label{e-hoamIII}
\ol{\pr}^*_{b,1}u=\psi\Td{\ol{\pr}}^*_{b}(\psi^{-1}u),\ \ \forall u\in{\rm Dom\,}\ol{\pr}^*_{b,1}.
\end{equation}
\end{lem} 

\begin{proof} 
Let $u\in{\rm Dom\,}\ol{\pr}^*_{b,1}$. We claim that $\psi^{-1}u\in{\rm Dom\,}\Td{\ol{\pr}}^*_{b}$ and $\Td{\ol{\pr}}^*_{b}(\psi^{-1}u)=\psi^{-1}\ol{\pr}^*_{b,1}u$. Put $\ol{\pr}^*_{b,1}u=h\in L^2(m_1)$. By definition, we have 
\begin{equation}\label{e-hobmI}
(\ddbar_{b,1}g\ |\ u)_{m_1,\hat\theta}=(g\ |\ h)_{m_1},\ \ \forall g\in{\rm Dom\,}\ddbar_{b,1}. 
\end{equation} 
From \eqref{e-crysmiIII}, Lemma~\ref{l-lopcII} and \eqref{e-hobmI}, we have 
\begin{equation}\label{e-hobmII}
\begin{split}
(\Td\ddbar_bf\ |\ \psi^{-1}u)_{\Td m,\hat\theta}&=(\psi\Td\ddbar_bf\ |\ u)_{m_1,\hat\theta}
=(\ddbar_{b,1}(\psi f)\ |\ u)_{m_1,\hat\theta}\\
&=(\psi f\ |\ h)_{m_1}=(f\ |\ \psi^{-1}h)_{\Td m},\ \ \forall f\in{\rm Dom\,}\Td\ddbar_b. 
\end{split}
\end{equation}
Thus, $\psi^{-1}u\in{\rm Dom\,}\Td{\ol{\pr}}^*_{b}$ and $\Td{\ol{\pr}}^*_{b}(\psi^{-1}u)=\psi^{-1}h=\psi^{-1}\ol{\pr}^*_{b,1}u$. 

We can repeat the procedure above and conclude that if $v\in{\rm Dom\,}\Td{\ol{\pr}}^*_{b}$ then $\psi v\in{\rm Dom\,}\ol{\pr}^*_{b,1}$ and $\ol{\pr}^*_{b,1}(\psi v)=\psi\Td{\ol{\pr}}^*_{b}v$. The lemma follows.
\end{proof}

Combining the above, we obtain (\ref{eq:Boxb1totildeBoxb1}) and (\ref{eq:Boxb1totildeBoxb2}):

\begin{thm}\label{t-hobmI}
We have 
\[\mbox{$u\in{\rm Dom\,}\Box_{b,1}$ if and only if $\frac{u}{\psi}\in{\rm Dom\,}\Td\Box_b$}\]
and 
\[\Box_{b,1}u=\psi\Td\Box_b(\frac{u}{\psi}),\ \ \forall u\in{\rm Dom\,}\Box_{b,1}.\]
\end{thm} 

Thus $\Box_{b,1}$ will have closed range in $L^2(m_1)$, if and only if $\tilde{\Box}_b$ has closed range in $L^2(\tilde{m})$. We will prove the latter in the next two sections, and that will establish Theorem~\ref{thm2}.

\section{The relation between $\Td \Box_b$ and $\hat \Box_b$} \label{sect:TdBoxbtohatBoxb}

In the last section, we saw how the solution of $\Box_{b,1}$ is reduced to the solution of the intermediate operator $\Td\Box_b$. In this section, we see how the latter could be further reduced to solving $\hat{\Box_b}$, which we introduced in Section~\ref{subsect:outlineofproof}. In particular, we will prove (\ref{eq:Boxbhattotilde1}) and (\ref{eq:Boxbhattotilde2}) there.

First, note that 
\begin{equation}\notag
L^2(\Td m)=L^2(\hat m),\ \ L^2_{(0,1)}(\Td m,\hat\theta)=L^2_{(0,1)}(\hat m,\hat\theta).
\end{equation}
In fact, from the expansions (\ref{eq:Gpexpand}) of $G_p$ and (\ref{e-ysmiVIII}) of $\psi$, we see that near $p$,
\begin{equation}\label{miII}
\frac{\hat m}{\Td m}= G_p^{-2} |\psi|^{-2} = 1+a(z,t),\ \ a(z,t)\in\mathcal{E}(\hat\rho^2).
\end{equation}
Let
\begin{equation}\notag
\hat{\ol{\pr}}^{*,f}_{b}:\Omega^{0,1}(X)\To C^\infty(X)
\end{equation}
be the formal adjoint of $\ddbar_b$ with respect to $(\,\cdot\,|\,\cdot\,)_{\hat m}$, $(\,\cdot\,|\,\cdot\,)_{\hat m,\hat\theta}$. Then 
$$
\Td{\ol{\pr}}^{*,f}_b=\frac{\hat m}{\Td m}\hat{\ol{\pr}}^{*,f}_b\frac{\Td m}{\hat m},
$$
so from the above, we see that 
\begin{equation}\label{e-hoIa}
\Td{\ol{\pr}}^{*,f}_b=\hat{\ol{\pr}}^{*,f}_b+g,\quad \text{for some $g\in\mathcal{E}(\hat\rho,T^{0,1}\hat X)$}.
\end{equation} 
(Here we think of the $(0,1)$ vector $g$ as an element on the dual space of $\Lambda^{0,1}T^* \hat X$.) We then have the following lemma about the relation between $\Td \ddbar_b$ and $\hat \ddbar_b$:

\begin{lem}\label{l-hoI} 
We have $\Td\ddbar_b=\hat\ddbar_b$. That is, $${\rm Dom\,}\Td\ddbar_b={\rm Dom\,}\hat\ddbar_b, \quad \text{and} \quad \Td\ddbar_bu=\hat\ddbar_bu$$ for all $u$ in the common domain of definition.
\end{lem}

\begin{proof}
Let $u\in{\rm Dom\,}\Td\ddbar_b$. We claim that $u\in{\rm Dom\,}\hat\ddbar_b$ and $\Td\ddbar_bu=\hat\ddbar_bu$. By definition, there is a $h\in L^2_{(0,1)}(\Td m,\hat\theta)$ such that 
\begin{equation}\label{e-hoIab}
(h\ |\ \alpha)_{\Td m,\hat\theta}=(u\ |\ \Td{\ddbar_b}^{*,f}\alpha)_{\Td m},\ \ \forall \alpha\in\Omega^{0,1}_0(X).
\end{equation}
Note that $h=\Td\ddbar_bu$. From \eqref{e-hoIa} and \eqref{e-hoIab}, we have 
\begin{equation}\label{e-hoIac}
\begin{split}
(h\ |\ \gamma)_{\hat m,\hat\theta}&=(h\ |\ \frac{\hat m}{\Td m}\gamma)_{\Td m,\hat\theta}=(u\ |\ \Td{\ol{\pr}}^{*,f}_b(\frac{\hat m}{\Td m}\gamma))_{\Td m,\hat\theta}=(u\ |\ \frac{\hat m}{\Td m}\hat{\ol{\pr}}^{*,f}_b\gamma)_{\Td m,\hat\theta}\\
&=(u\ |\ \hat{\ol{\pr}}^{*,f}_b\gamma)_{\hat m,\hat\theta},\ \ \forall \gamma\in\Omega^{0,1}_0(X).
\end{split}\end{equation} 

Take $\alpha \in\Omega^{0,1}(\hat X)$. Let 
\[D_r:=\set{x=(x_1,x_2,x_3)\in\Real^3;\, \abs{x}^2:=\abs{x_1}^2+\abs{x_2}^2+\abs{x_3}^2<r},\ \ r>0,\] 
be a small ball. We identify $D_r$ with a neighbourhood of $p$. Take $\chi\in C^\infty_0(D_r,\ol\Real_+)$, $\chi=1$ on $D_{\frac{r}{2}}:=\set{x\in\Real^3;\, \abs{x}^2<\frac{r}{2}}$. Take $\epsilon>0$, $\epsilon$ small and put $\alpha_{\epsilon}=(1-\chi(\frac{x}{\epsilon}))\alpha \in\Omega^{0,1}_0(X)$. Then, 
\[\begin{split}
&\norm{\alpha_\epsilon-\alpha}^2_{\hat m,\hat\theta}=\int \abs{\alpha}^2_{\hat\theta}\abs{\chi(\frac{x}{\epsilon})}^2\hat m\leq C\int_{\abs{x}\leq\epsilon r}\hat m\To0\ \ \mbox{as $\epsilon\To0$},\\
&\norm{\hat{\ol{\pr}}^{*,f}_{b}(\alpha_\epsilon-\alpha)}^2_{\hat m}\leq\int \abs{\hat{\ol{\pr}}^{*,f}_b \alpha}^2_{\hat m}\abs{\chi(\frac{x}{\epsilon})}^2\hat m
+\int\abs{\alpha}^2_{\hat\theta}\abs{\frac{1}{\epsilon}(Z_1\chi)(\frac{x}{\epsilon})}^2\hat m\\
&\leq\frac{C_1}{\epsilon^2}\int_{\abs{x}\leq\epsilon r}\hat m\To0\ \ \mbox{as $\epsilon\To0$},
\end{split}\] 
where $C>0$, $C_1>0$ are constants independent of $\epsilon$. We conclude that there exist $\alpha_j\in\Omega^{0,1}_0(X)$, $j=1,2,\ldots$, such that 
\[\lim_{j\To\infty}\norm{\alpha_j-\alpha}_{\hat m,\hat\theta}=0,\ \ \lim_{j\To\infty}\norm{\hat{\ol{\pr}}^{*,f}_b(\alpha_j-\alpha)}_{\hat m}=0.\]
Combining this with \eqref{e-hoIac}, we have 
\[(h\ |\ \alpha)_{\hat m,\hat\theta}=\lim_{j\To\infty}(h\ |\ \alpha_j)_{\hat m,\hat\theta}=\lim_{j\To\infty}(u\ |\ \hat{\ol{\pr}}^{*,f}_b \alpha_j)_{\hat m}=(u\ |\ \hat{\ol{\pr}}^{*,f}_b \alpha)_{\hat m}.\] 
Thus, $u\in{\rm Dom\,}\hat\ddbar_b$ and $\hat\ddbar_bu=\Td\ddbar_bu=h$. 

We have proved that if $u\in{\rm Dom\,}\Td\ddbar_b$ then $u\in{\rm Dom\,}\hat\ddbar_b$ and $\hat\ddbar_bu=\Td\ddbar_bu$. 

We can repeat the procedure above and conclude that if $v\in{\rm Dom\,}\hat\ddbar_b$ then $v\in{\rm Dom\,}\Td\ddbar_b$ and $\Td\ddbar_bv=\hat\ddbar_bv$. The lemma follows.
\end{proof}

Next, we want to understand the relation between $\Td \ddbar_b^*$ and $\hat \ddbar_b^*$. To do so, we need the following lemma:

\begin{lem}\label{l-hocmI}
Let $v\in{\rm Dom\,}\hat\ddbar_b$. Then, $\frac{\Td m}{\hat m}v\in{\rm Dom\,}
\hat\ddbar_b$ and 
\[\hat\ddbar_b(\frac{\Td m}{\hat m}v)=\frac{\Td m}{\hat m}\hat\ddbar_bv-\frac{\Td m}{\hat m}g^{*}v\]
where $g^*$ is the $(0,1)$ form dual to the $(0,1)$ vector $g$.
\end{lem}

\begin{proof} 
For all $\alpha\in\Omega^{0,1}_0(X)$, we have 
\begin{equation}\notag
\begin{split}
(\frac{\Td m}{\hat m}v\ |\ \hat{\ol{\pr}}^{*,f}_b\alpha)_{\hat m}&=(\frac{\Td m}{\hat m}v\ |\ (\Td{\ol{\pr}}^{*,f}_b-g)\alpha)_{\hat m}\ \ \mbox{here we used \eqref{e-hoIa}}\\
&=(v\ |\ (\Td{\ol{\pr}}^{*,f}_b-g)\alpha)_{\Td m}\\
&=(\hat\ddbar_bv\ | \alpha)_{\Td m,\hat\theta}-(g^{*}v\ |\ \alpha)_{\Td m,\hat\theta}\ \ \mbox{here we used Lemma~\ref{l-hoI}}\\
&=(\frac{\Td m}{\hat m}\hat\ddbar_bv-\frac{\Td m}{\hat m}g^{*}v\ |\ \alpha)_{\hat m,\hat\theta}.
\end{split}
\end{equation} 

We can now repeat the procedure in the proof of Lemma~\ref{l-hoI} and conclude that 
\[(\frac{\Td m}{\hat m}v\ |\ \hat{\ol{\pr}}^{*,f}_bh)_{\hat m}=(\frac{\Td m}{\hat m}\hat\ddbar_bv-\frac{\Td m}{\hat m}g^{*}v\ |\ h)_{\hat m,\hat\theta},\ \ \forall h\in\Omega^{0,1}(\hat X).\] 
The lemma follows.
\end{proof} 

We can now prove:

\begin{lem}
We have ${\rm Dom\,}\hat{\ol{\pr}}^*_b={\rm Dom\,}\Td{\ol{\pr}}^*_b$ and 
\begin{equation}\label{e-hocmII} 
\Td{\ol{\pr}}^*_bu=\hat{\ol{\pr}}^*_bu+gu,\ \ \forall u\in{\rm Dom\,}\hat{\ol{\pr}}^*_b={\rm Dom\,}\Td{\ol{\pr}}^*_b.
\end{equation} 
\end{lem} 

\begin{proof} 
Let $u\in{\rm Dom\,}\hat{\ol{\pr}}^*_b$. We claim that $u\in{\rm Dom\,}\Td{\ol{\pr}}^*_b$ and $\Td{\ol{\pr}}^*_bu=\hat{\ol{\pr}}^*_bu+gu$. 
From Lemma~\ref{l-hoI} and Lemma~\ref{l-hocmI}, we can check that for every $v\in{\rm Dom\,}\Td\ddbar_b$,  
\begin{equation}\notag
\begin{split}
(v\ |\ (\hat{\ol{\pr}}^*_b+g)u)_{\Td m}&=(\frac{\Td m}{\hat m}v\ |\ \hat{\ol{\pr}}^*_bu)_{\hat m}+(v\ |\ gu)_{\Td m}\\
&=(\hat\ddbar_b(\frac{\Td m}{\hat m}v)\ |\ u)_{\hat m,\hat\theta}+(g^{*}v\ |\ u)_{\Td m,\hat\theta}\\
&=(\frac{\Td m}{\hat m}\hat\ddbar_b v-\frac{\Td m}{\hat m}g^{*}v\ |\ u)_{\hat m,\hat\theta}+(\frac{\Td m}{\hat m}g^{*}v\ |\ u)_{\hat m}\\
&=(\hat\ddbar_bv\ |\ u)_{\Td m,\hat\theta}=(\Td\ddbar_bv\ |\ u)_{\Td m,\hat\theta}.
\end{split}
\end{equation}
Thus, $u\in{\rm Dom\,}\Td{\ol{\pr}}^*_b$ and $\Td{\ol{\pr}}^*_bu=\hat{\ol{\pr}}^*_bu+gu$. 

Similarly, for $u\in{\rm Dom\,}\Td{\ol{\pr}}^*_b$, we can repeat the procedure above and conclude that  $u\in{\rm Dom\,}\hat{\ol{\pr}}^*_b$.
The lemma follows.
\end{proof} 

It follows from the above that (\ref{eq:Boxbhattotilde1}) and (\ref{eq:Boxbhattotilde2}) holds:

\begin{thm} \label{thm:TdBoxbtohatBoxb}
$$
{\rm Dom\,}\tilde \Box_{b} = {\rm Dom\,}\hat \Box_b,
$$
and
$$
\Td \Box_b u= \hat \Box_b u + g \hat \ddbar_b u ,\ \ \forall u\in{\rm Dom\,}\tilde \Box_{b}.
$$
\end{thm}

This allows one to understand solutions to $\Td \Box_b$ via solutions to $\hat \Box_b$, as we will see below.

\section{Proof of Theorem~\ref{thm2}} \label{sect:closedrange}

In this section, we will see that $\Td \Box_b$ and $\Box_{b,1}$ have closed ranges in $L^2(\Td m)$ and $L^2(m_1)$ respectively. The analogous property for $\hat \Box_b$ is well-known by the CR embeddability of $\hat X$ into $\mathbb{C}^N$. 

\begin{thm}\label{t-hoI}
The operator 
\[\Td\ddbar_b:{\rm Dom\,}\Td\ddbar_b\subset L^2(\Td m)\To L^2_{(0,1)}(\Td m, \hat\theta)\] 
has closed range.
\end{thm}

\begin{proof}
One simply notes that since $\hat X$ is CR embeddable in some $\mathbb{C}^N$, by the result of \cite{Koh86}, $\hat\ddbar_b \colon \text{Dom} \hat \ddbar_b \subset L^2(\hat m) \to L^2_{(0,1)}(\hat m, \hat\theta)$ has closed range in $L^2_{(0,1)}(\hat m,\hat\theta)$. By the identity of $\hat \ddbar_b$ with $\Td \ddbar_b$ as in Lemma~\ref{l-hoI}, it follows that $\Td\ddbar_b$ has closed range in $L^2_{(0,1)}(\Td m,\hat\theta)$. 
\end{proof}

It is now a standard matter to prove:

\begin{thm}\label{t-hoII}
The operator 
\[\Td\Box_b:{\rm Dom\,}\Td\Box_b\subset L^2(\Td m)\To L^2(\Td m)\] 
has closed range.
\end{thm}

\begin{proof} 
By Theorem~\ref{t-hoI}, there is a constant $c>0$ such that 
\begin{equation}\label{e-hoII} 
\norm{\Td\ddbar_bu}^2_{\Td m,\hat\theta}\geq c\norm{u}^2_{\Td m},\ \ \forall u\bot{\rm Ker\,}\Td\ddbar_b.
\end{equation}

Let $f\in{\rm Dom\,}\Td\Box_b\bigcap({\rm Ker\,}\Td\Box_b)^{\bot}$. It is not difficult to see that ${\rm Ker\,}\Td\Box_b={\rm Ker\,}\Td\ddbar_b$. Thus, $f\in{\rm Dom\,}\Td\ddbar_b\bigcap({\rm Ker\,}\Td\ddbar_b)^{\bot}$. From this observation and \eqref{e-hoII}, we have 
\[\norm{\Td\Box_bf}_{\Td m}\norm{f}_{\Td m}\geq(\Td\Box_bf\ |\ f)_{\Td m}=(\Td\ddbar_bf\ |\ \Td\ddbar_bf)_{\Td m,\hat\theta}\geq c\norm{f}^2_{\Td m},\]
where $c>0$ is the constant as in \eqref{e-hoII}. Thus, 
\begin{equation}\label{e-lopc523III}
\norm{\Td\Box_bf}_{\Td m}\geq c\norm{f}_{\Td m},\ \ \forall f\in{\rm Dom\,}\Td\Box_{b}\bigcap({\rm Ker\,}\Td\Box_{b})^{\bot},
\end{equation}
where $c>0$ is the constant as in \eqref{e-hoII}. From \eqref{e-lopc523III}, it is easy to see that $\Td\Box_b$ has closed range. 
The theorem follows.
\end{proof} 

Now we use the closed range property of $\Td\ddbar_b$ to prove the same for $\ddbar_{b,1}$.

\begin{thm}\label{t-lopcI}
The operator 
\[\ddbar_{b,1}:{\rm Dom\,}\ddbar_{b,1}\subset L^2(m_1)\To L^2_{(0,1)}(m_1,\hat\theta)\] 
has closed range.
\end{thm}

\begin{proof}
Let $f_j\in{\rm Dom\,}\ddbar_{b,1}$, $j=1,2,\ldots$, $\ddbar_{b,1}f_j=g_j\in L^2_{(0,1)}(m_1,\hat\theta)$, $j=1,2,\ldots$. We assume that 
there is a function $g\in L^2_{(0,1)}(m_1,\hat\theta)$ such that $\lim_{j\To\infty}\norm{g_j-g}_{m_1,\hat\theta}=0$. 
We are going to show that $g\in{\rm Ran\,}\ddbar_{b,1}$. From \eqref{e-lopcIII}, \eqref{e-lopcIIIb} and \eqref{e-crysmiIII}, we see that 
\begin{equation}\label{e-lopc523I}
\begin{split}
&\psi^{-1}f_j\in{\rm Dom\,}\Td\ddbar_{b},\ \ j=1,2,\ldots,\\
&\Td\ddbar_{b}(\psi^{-1}f_j)=\psi^{-1}g_j\in L^2_{(0,1)}(\Td m,\hat\theta),\ \ j=1,2,\ldots,\\
&\lim_{j\To\infty}\norm{\psi^{-1}g_j-\psi^{-1}g}_{\Td m,\hat\theta}=0.
\end{split}
\end{equation} 
Since $\Td\ddbar_{b}$ has closed range, we can find $\Td h\in{\rm Dom\,}\Td\ddbar_{b}$ such that $\Td\ddbar_{b}\Td h=\psi^{-1}g$. Put $h=\psi\Td h\in L^2(m_1)$. From \eqref{e-lopcIII} and \eqref{e-lopcIIIb}, we see that $h\in{\rm Dom\,}\ddbar_{b,1}$ and $\ddbar_{b,1}h=g$. Thus,  
$g\in{\rm Ran\,}\ddbar_{b,1}$. The theorem follows. 
\end{proof}

From Theorem~\ref{t-lopcI}, we can repeat the proof of Theorem~\ref{t-hoII} and conclude the proof of Theorem~\ref{thm2}.

\section{Proof of Theorem~\ref{thm3}} 

In the last section, we have seen that $\Box_{b,1}$ and $\Td \Box_b$ have closed ranges in $L^2$. Thus one can define the partial inverses $N$ and $\Td N$ of $\Box_{b,1}$ and $\Td \Box_b$ respectively (c.f. Section~\ref{subsect:def}). Furthermore, we write $\Pi$ and $\Td \Pi$ for the Szeg\"o projections, which are orthogonal projections onto ${\rm Ker\,}\Td\Box_b$ and ${\rm Ker\,}\Box_{b,1}$ respectively, as in Sections~\ref{subsect:strategy} and \ref{subsect:outlineofproof}. Our goal is to understand $N$ and $\Pi$, as in Theorem~\ref{thm3}. But from \eqref{e-crysmiII}, \eqref{e-crysmiIII} and Theorem~\ref{t-hobmI}, we obtain
$$
\Box_{b,1} \psi\Td N\frac{1}{\psi} + \psi\Td\Pi\frac{1}{\psi} = I \quad \mbox{on $L^2(m_1)$}.
$$
Thus we obtain (\ref{eq:soltoTdsol}), namely
$$
N=\psi\Td N\frac{1}{\psi} \quad \text{and} \quad \Pi=\psi\Td\Pi\frac{1}{\psi} .
$$
The analysis of $N$ and $\Pi$ then reduces to the analysis of $\Td N$ and $\Td \Pi$; in fact, to prove Theorem~\ref{thm3}, it suffices to prove instead  (\ref{eq:Pimapprop}) and (\ref{eq:Nmapprop}): 

\begin{thm}\label{t-msmapepa}
$\Td \Pi$ and $\Td N$ extend as continuous operators
$$
\Td \Pi \colon \mathcal{E}(\hat \rho^{-4+\delta}) \to \mathcal{E}(\hat \rho^{-4+\delta})
$$
$$
\Td N \colon \mathcal{E}(\hat \rho^{-4+\delta}) \to \mathcal{E}(\hat \rho^{-2+\delta})
$$
for every $0 < \delta < 2$.
\end{thm}  
We will achieve this by reducing to the analogous properties of $\hat N$ and $\hat\Pi$, which we proved in Section~\ref{sect2}.

The starting point is the following lemma:

\begin{lem}\label{l-copI}
On $L^2(\Td m)$, we have 
\begin{align}
&\Td\Pi(I+\hat R)=\hat\Pi,\label{e-copI}\\
&\Td N(I+\hat R)=(I-\Td\Pi) \hat N,\label{e-copIlkmiII}
\end{align}
where $\hat R = g \hat \ddbar_b \hat N$.
\end{lem} 

\begin{proof}
First, we know that $\hat N:L^2(\hat m)\To {\rm Dom\,}\hat\Box_b = {\rm Dom\,}\Td\Box_b$. 
From Theorem~\ref{thm:TdBoxbtohatBoxb}, we can check that
$$
\Td\Box_b\hat N+\hat\Pi=\hat\Box_b\hat N+\hat\Pi+\hat R=I+\hat R.
$$
From this, we have 
\begin{equation}\label{almiIX}
\Td\Pi(I+\hat R)=\Td\Pi(\Td\Box_b\hat N+\hat\Pi)=\Td\Pi\hat\Pi.
\end{equation}
On the other hand, we have 
\begin{equation}\label{almiX}
\hat\Pi=(\Td N\Td\Box_b+\Td\Pi)\hat\Pi=\Td\Pi\hat\Pi.
\end{equation}
From \eqref{almiIX} and \eqref{almiX}, we get \eqref{e-copI}.

Now, from Theorem~\ref{thm:TdBoxbtohatBoxb} again, we have
\begin{equation}\notag
\begin{split}
\hat N&=(\Td N\Td\Box_b+\Td\Pi)\hat N \\
&=(\Td N\hat\Box_b+\Td\Pi)\hat N+\Td N\hat R\\
&=\Td N(I-\hat\Pi)+\Td\Pi\hat N+\Td N\hat R\\
&=\Td N+\Td\Pi\hat N+\Td N\hat R;
\end{split}
\end{equation} 
in the last line we used that fact that $\Td N\hat\Pi=0$. \eqref{e-copIlkmiII} then follows, and we are done.
\end{proof}

We now extend the definitions of $\Td \Pi$ and $\Td N$ to $\mathcal{E}(\hat \rho^{-4+\delta})$, $0 < \delta < 2$. The problem is that one does not have a good inverse for $(I+\hat R)$ in \eqref{e-copI} and \eqref{e-copIlkmiII}. The key then is to rewrite (\ref{e-copI}) and (\ref{e-copIlkmiII}) as
\begin{equation} \label{eq:TdPitohatPi2}
\Td \Pi = \hat \Pi - \Td \Pi \hat R,
\end{equation}
\begin{equation} \label{eq:TdNtohatN2}
\Td N = (I-\Td \Pi) \hat N - \Td N \hat R.
\end{equation}
Note that $\hat R$ extends to a continuous operator
$$
\hat R \colon \mathcal{E}(\hat \rho^{-4+\delta}) \to \mathcal{E}(\hat \rho^{-2+\delta}) \subset L^2(\Td m)
$$
for every $0 < \delta < 2$, since $\hat R = g \hat \ddbar_b \hat N$, and $\hat N$ satisfies the analogous property. Thus $\Td \Pi \hat R$ maps $\mathcal{E}(\hat \rho^{-4+\delta})$ continuously to $L^2(\Td m)$. Since $\hat \Pi$ maps $\mathcal{E}(\hat \rho^{-4+\delta})$ continuously to $\mathcal{E}(\hat \rho^{-4+\delta})$, by (\ref{eq:TdPitohatPi2}), we have extended the domain of definition of $\Td \Pi $ to $\mathcal{E}(\hat \rho^{-4+\delta})$. Furthermore, $\Td N \hat R$ maps $\mathcal{E}(\hat \rho^{-4+\delta})$ continuously to $L^2(\Td m)$, and $\hat N$ maps $\mathcal{E}(\hat \rho^{-4+\delta})$ continuously to $\mathcal{E}(\hat \rho^{-2+\delta}) \subset L^2(\Td m)$. Thus together with the continuity of $\Td \Pi$ on $L^2(\Td m)$, from (\ref{eq:TdNtohatN2}), we see that $\Td N$ extends as a continuous map $\mathcal{E}(\hat \rho^{-4+\delta}) \to L^2(\Td m)$.

To proceed further, let's write $\hat{\Pi}^{*,\Td m}$, $\hat{N}^{*,\Td m}$ and $\hat R^{*,\Td m}$ for the adjoints of $\hat \Pi$, $\hat N$ and $\hat R$ with respect to the inner product of $L^2(\Td m)$. We note that
$$
\hat \Pi^{*,\Td m} = \frac{\hat m}{\Td m} \hat \Pi \frac{\Td m}{\hat m}, 
$$
$$
\hat N^{*,\Td m} = \frac{\hat m}{\Td m} \hat N \frac{\Td m}{\hat m}, 
$$
$$
\hat R^{*,\Td m} = \frac{\hat m}{\Td m} \hat N \hat \ddbar_b^* (g^{*} \frac{\Td m}{\hat m}), 
$$
where $\Td m / \hat m := G_p^2 |\psi|^2$ is the density of $\Td m$ with respect to $\hat m$, and similarly $\hat m / \Td m := G_p^{-2} |\psi|^{-2}$. Here $g^*$ is the $(0,1)$ form dual to $g$. Since $\Td m / \hat m$, $\hat m / \Td m \in \mathcal{E}(\hat \rho^0)$, one can show that
\begin{equation} \label{eq:Pi*2}
\hat \Pi^{*,\Td m} \colon \mathcal{E}(\hat \rho^{-4+\delta}) \to \mathcal{E}(\hat \rho^{-4+\delta})
\end{equation}
\begin{equation} \label{eq:N*2}
\hat N^{*,\Td m} \colon \mathcal{E}(\hat \rho^{-4+\delta}) \to \mathcal{E}(\hat \rho^{-2+\delta})
\end{equation}
\begin{equation} \label{eq:R*2}
\hat R^{*,\Td m} \colon \mathcal{E}(\hat \rho^{-4+\delta}) \to \mathcal{E}(\hat \rho^{-2+\delta})
\end{equation}
for every $0 < \delta < 2$; these are easy consequences of the analogous properties of $\hat \Pi$, $\hat N$ and $\hat R$. The problem is that it is not clear that $(I+\hat{R}^{*,\Td m})$ is invertible on $\mathcal{E}(\hat \rho^{-4+\delta})$; if it is, then we can invoke (\ref{eq:TdtohatPiadjoint}) and (\ref{eq:TdtohatNadjoint}) and our proof of Theorem~\ref{thm3} would be much easier. In order to get around this problem, we introduce a cut-off $\chi$, as was explained in Section~\ref{subsect:outlineofproof}; we will prove

\begin{thm}\label{t-mspII}
Let $\chi\in C^\infty(\hat X)$ with $\chi=1$ near $p$. Then  
\[(1-\chi)\Td N, (1-\chi)\Td\Pi:\mathcal{E}(\hat\rho^{-4+\delta})\To C^\infty_0(X)\]
are continuous for $0<\delta<2$.
\end{thm}

\begin{thm}\label{thm:I+hatRchi-1}
If the support of $\chi \in C^\infty(\hat X)$ is a sufficiently small neighborhood of $p$, then $(I + \hat{R}^{*,\Td m} \chi)$ is invertible on $L^2(\Td m)$, and extends to a linear map 
$$ 
(I + \hat{R}^{*,\Td m} \chi)^{-1} \colon \mathcal{E}(\hat \rho^{-4+\delta}) \to \mathcal{E}(\hat \rho^{-4+\delta})
$$ 
for every $0 < \delta < 4$.
\end{thm}

Assuming these for the moment. Then one can finish the proof of Theorem~\ref{t-msmapepa} using (\ref{eq:TdtohatPiadjointchi}) and (\ref{eq:TdtohatNadjointchi}) as explained in Section~\ref{subsect:outlineofproof} shortly after these identities. Theorem~\ref{thm3} then follows from (\ref{eq:soltoTdsol}) and Theorem~\ref{t-msmapepa}. We omit the details.

We now turn to the proofs of the theorems above. 

\begin{proof}[Proof of Theorem~\ref{thm:I+hatRchi-1}]
First, observe that
$$
I + \hat{R}^{*,\Td m} \chi = \frac{\hat m}{\Td m} (I + \hat N \hat \ddbar_b^* \chi g^* ) \frac{\Td m}{\hat m},
$$
and $\Td m/\hat m$, $\hat m/\Td m \in \mathcal{E}(\hat \rho^0)$. Thus it suffices to prove that $I + \hat N \hat \ddbar_b^* \chi g^*$ is invertible on $L^2(\Td m)$, and extends to a linear map 
$$ 
(I + \hat N \hat \ddbar_b^* \chi g^* )^{-1} \colon \mathcal{E}(\hat \rho^{-4+\delta}) \to \mathcal{E}(\hat \rho^{-4+\delta})
$$ 
for every $0 < \delta < 4$. But this follows from Theorem~\ref{t-slalkeyI}, since $\chi g^* \in \mathcal{E}(\hat \rho, \Lambda^{0,1} T^*\hat X)$ has compact support in a sufficiently small neighborhood of $p$. Thus we are done.
\end{proof}

\begin{proof}[Proof of Theorem~\ref{t-mspII}]
We need to recall the Kohn Laplacian $\hat \Box_{b,\varepsilon} := \hat \ddbar_{b,\varepsilon}^* \hat \ddbar_{b,\varepsilon}$ with respect to the volume form $m_{\varepsilon} := \eta_{\varepsilon} \hat m + (1-\eta_{\varepsilon}) \Td m$ as described in Section~\ref{subsect:outlineofproof}. First
$$ 
\frac{\hat m_{\varepsilon}}{\Td m}= \eta_{\varepsilon} G_p^{-2} |\psi|^{-2} + (1-\eta_{\varepsilon}) =  \eta_{\varepsilon} (1+a) + (1-\eta_{\varepsilon}) = 1+ \eta_{\varepsilon} a.
$$ 
The upshot is that $\eta_{\varepsilon} a \in \mathcal{E}(\hat \rho^2, T^{0,1} \hat X)$ has compact support near $p$. Thus if we follow the construction in Section~\ref{sect:TdBoxbtohatBoxb}, there will exist some $g_{\varepsilon} \in \mathcal{E}(\hat \rho^1, T^{0,1} \hat X)$ (possibly non-smooth near $p$) such that 
\begin{equation} \label{eq:Boxbhattotildeepep}
\Td \Box_b u= \hat \Box_{b,\varepsilon} u + g_{\varepsilon} \hat \ddbar_b u ,\ \ \forall u\in{\rm Dom\,}\tilde \Box_{b};
\end{equation} 
in addition $g_{\varepsilon}$ will be compactly supported in the support of $\eta_{\varepsilon}$. It is known that $\hat \Box_{b,\varepsilon}$ has closed range in $L^2(\hat m_{\varepsilon})$. Thus one can define the partial inverse $\hat N_{\varepsilon}$ of $\hat \Box_{b,\varepsilon}$, as well as the Szeg\"o projection $\hat \Pi_{\varepsilon}$ onto the kernel of $\hat \Box_{b,\varepsilon}$, such that
$$
\hat \Box_{b,\varepsilon} \hat N_{\varepsilon} + \hat \Pi_{\varepsilon} = I.
$$
One can then repeat the proof of Lemma~\ref{l-copI}, and show that 
\begin{equation} \label{eq:TdPitohatPiep}
\Td \Pi (I + \hat{R}_{\varepsilon})  = \hat \Pi_{\varepsilon}
\end{equation}
\begin{equation} \label{eq:TdNtohatNep}
\Td N (I + \hat{R}_{\varepsilon}) =  (I - \Td \Pi) \hat N_{\varepsilon}
\end{equation}
on $L^2(\Td m)$, where 
$$
\hat{R}_{\varepsilon} := g_{\varepsilon} \hat \ddbar_{b,\varepsilon} \hat N_{\varepsilon}.
$$
Now we write $\hat \Pi_{\varepsilon}^{*,\Td m}$, $\hat N_{\varepsilon}^{*,\Td m}$ and $\hat{R}_{\varepsilon}^{*,\Td m}$ for the adjoints of $\hat \Pi_{\varepsilon}$, $\hat N_{\varepsilon}$ and $\hat R_{\varepsilon}$ with respect to $L^2(\Td m)$. By taking adjoints of (\ref{eq:TdPitohatPiep}) and (\ref{eq:TdNtohatNep}), and multiplying by $(1-\chi)$, it follows that
\begin{equation} \label{eq:1-chiTdPi2}
(1-\chi) \Td \Pi = (1-\chi) \hat \Pi_{\varepsilon}^{*,\Td m} - (1-\chi) \hat{R}_{\varepsilon}^{*,\Td m} \Td \Pi,
\end{equation}
\begin{equation} \label{eq:1-chiTdN2}
(1-\chi) \Td N = (1-\chi) \hat N_{\varepsilon}^{*,\Td m} (I - \Td \Pi) - (1-\chi) \hat{R}_{\varepsilon}^{*,\Td m} \Td N.
\end{equation}
But
$$
(1-\chi) \hat{R}_{\varepsilon}^{*,\Td m} = \frac{\hat m_{\varepsilon}}{\Td m} (1-\chi) \hat N_{\varepsilon} \hat \ddbar_{b,\varepsilon}^* g_{\varepsilon}^* \frac{\Td m}{\hat m_{\varepsilon}},
$$
and if $\varepsilon$ is chosen sufficiently small (so that the support of $g_{\varepsilon}$ is disjoint from that of $1-\chi$), then $(1-\chi) \hat N_{\varepsilon} \hat \ddbar_{b,\varepsilon}^* g_{\varepsilon}^*$ is an infinitely smoothing pseudodifferential operator, by pseudolocality of $\hat N_{\varepsilon}$. Hence the last term of (\ref{eq:1-chiTdPi2}), and also the last term of (\ref{eq:1-chiTdN2}), map $\mathcal{E}(\hat \rho^{-4+\delta})$ into $C^{\infty}_0(X)$. Since for every $0 < \varepsilon < 1$ and every $0 < \delta < 2$, 
$$
\hat \Pi_{\varepsilon}^{*,\Td m} = \frac{\hat m_{\varepsilon}}{\Td m}\hat \Pi_{\varepsilon} \frac{\Td m}{\hat m_{\varepsilon}} \colon \mathcal{E}(\hat \rho^{-4+\delta}) \to \mathcal{E}(\hat \rho^{-4+\delta}),
$$ 
and 
$$
\hat N_{\varepsilon}^{*,\Td m} = \frac{\hat m_{\varepsilon}}{\Td m}\hat N_{\varepsilon} \frac{\Td m}{\hat m_{\varepsilon}} \colon \mathcal{E}(\hat \rho^{-4+\delta}) \to \mathcal{E}(\hat \rho^{-2+\delta}),
$$ 
it follows from (\ref{eq:1-chiTdPi2}) that $(1-\chi) \Td \Pi$ maps $\mathcal{E}(\hat\rho^{-4+\delta})$ continuously into $C^\infty_0(X)$ as desired. This then implies the corresponding result for $(1-\chi) \Td N$ by (\ref{eq:1-chiTdN2}), and we are done.
\end{proof}

\section{Proof of Theorem~\ref{thm4}} \label{sect:app}

In this section, we will complete the proof of Theorem~\ref{thm4}. To begin with, we have the following lemma:

\begin{lem} \label{lem:PiErho0}
If $\alpha \in \mathcal{E}(\hat \rho^0, \Lambda^{0,1} T^* \hat X)$, then $\ddbar_{b,1}^* \alpha \in \mathcal{E}(\hat \rho^{-1})$, and 
$$
\Pi \ddbar_{b,1}^* \alpha = 0.
$$
\end{lem}

\begin{proof}
Given $\alpha \in \mathcal{E}(\hat \rho^0, \Lambda^{0,1} T^* \hat X)$, and any $\gamma \in (0,1)$, there exists a sequence of smooth and compactly supported $\alpha_j \in \Omega_0^{0,1}(X)$ such that 
$$
\alpha_j \to \alpha \quad \text{in $\mathcal{E}(\hat \rho^{-\gamma}, \Lambda^{0,1} T^* \hat X)$.}
$$
Then
$$
\ddbar_{b,1}^* \alpha_j \to \ddbar_{b,1}^* \alpha \quad \text{in $\mathcal{E}(\hat \rho^{-1-\gamma}, \Lambda^{0,1} T^* \hat X)$,}
$$
so by continuity of $\Pi$ on $\mathcal{E}(\hat \rho^{-1-\gamma}, \Lambda^{0,1} T^* \hat X)$, we have 
$$
\Pi \ddbar_{b,1}^* \alpha = \lim_{j \to \infty} \Pi \ddbar_{b,1}^* \alpha_j.
$$
But the right hand side here is zero, since 
$$
\Pi \ddbar_{b,1}^* = 0 \quad \text{on ${\rm Dom\,}\ddbar^*_{b,1}$},
$$
and $\alpha_j\in{\rm Dom\,}\ddbar^*_{b,1}$ for all $j$. Hence we are done.
\end{proof}

Now to prove of Theorem~\ref{thm4}, suppose $f = G_p^2 F$ and $F = \Box_b \Td \beta$ as in Section~\ref{subsect:strategy}, where
$$
\Td\beta=\beta_0+\beta_1,\quad\beta_0 = \chi(z,t) \frac{i \bar z}{|z|^2-it} \in \mathcal{E}(\hat \rho^{-1}), \quad \beta_1 \in \mathcal{E}(\hat \rho^1).
$$ 
Then $f = \Box_{b,1} \Td \beta$. Our goal is to compute $\Pi f = \Pi \ddbar_{b,1}^* (\ddbar_{b,1} \Td \beta)$. The problem is that $\ddbar_{b,1} \Td \beta$ is not in $\mathcal{E}(\hat \rho^0, \Lambda^{0,1} T^* \hat X)$; otherwise we could simply apply the above lemma to conclude. Nevertheless, we will write $\ddbar_{b,1} \Td \beta$ as the sum of a main term and an error, where the error is in $\mathcal{E}(\hat \rho^0, \Lambda^{0,1} T^* \hat X)$, and the $\ddbar_{b,1}^*$ of the main term can be approximated in $\mathcal{E}(\hat \rho^{-1})$ by the $\ddbar_{b,1}^*$ of some forms in $\mathcal{E}(\hat \rho^0, \Lambda^{0,1} T^* \hat X)$. Then we can conclude using the lemma, and the continuity of $\Pi$ on $\mathcal{E}(\hat \rho^{-1})$.

To begin with, note that
$$
\ddbar_{b,1} \Td \beta = \chi \ddbar_{b,1} \frac{i \bar z}{|z|^2-it} + (\ddbar_{b,1} \chi)  \frac{i \bar z}{|z|^2-it}  + \ddbar_{b,1} \beta_1,
$$
and the last two terms are in $\mathcal{E}(\hat \rho^0, \Lambda^{0,1} T^* \hat X)$. Thus the key is to compute the first term. Suppose we pick a local section $\hat Z_1$ of $T^{1,0} \hat X$ near $p$, with $\langle \hat Z_1 | \hat Z_1 \rangle_{\hat \theta} = 1$, such that $\hat Z_1$ admits the expansion (\ref{s1-e5b}) in CR normal coordinates $(z,t)$ near $p$. We also write $\hat Z_{\bar 1}$ for $\overline{\hat Z_1}$, and $\hat Z^{\bar 1}$ for the dual $(0,1)$ form of $\hat Z_{\bar 1}$. Then by the expansion (\ref{s1-e5b}) of $\hat Z_1$, we have
$$
\chi \ddbar_{b,1} \frac{i \bar z}{|z|^2-it} 
= \chi \hat Z_{\bar 1} \left( \frac{i \bar z}{|z|^2-it} \right) \hat Z^{\bar 1} 
= -i \chi \frac{|z|^2+it}{(|z|^2-it)^2} \hat Z^{\bar 1} + \text{error},
$$
where the error is in $\mathcal{E}(\hat \rho^0, \Lambda^{0,1} T^* \hat X)$. Thus
$$
\ddbar_{b,1} \Td \beta = -i \chi \frac{|z|^2+it}{(|z|^2-it)^2} \hat Z^{\bar 1}  + \text{error},
$$
where the first term is in $\mathcal{E}(\hat \rho^{-2}, \Lambda^{0,1} T^* \hat X)$, and the error is in $\mathcal{E}(\hat \rho^0, \Lambda^{0,1} T^* \hat X)$. 

To proceed further, we write $$m_1 = G_p^2 \hat \theta \wedge d \hat \theta = 2i v(z,t) dt \wedge d\bar z \wedge dz$$ for some function $v(z,t)$ near $p$. Then by the expansion (\ref{eq:Gpexpand}) of $G_p$ and the expansion (\ref{eq:hthetaexpansion}) of $\hat \theta$, we have
$$
v(z,t) = \frac{1}{4\pi^2 \hat \rho^4} + \frac{A}{\pi \hat \rho^2} + \text{error in $\mathcal{E}(\hat \rho^{-1})$}.
$$
Now we define a $(0,1)$ form
$$
\alpha_0 := -2\pi i \chi v^{-1} \overline{\psi}^{-3} \hat Z^{\bar 1}
$$
near $p$. Then using the expansion (\ref{e-ysmiVIII}) of $\psi$, we get 
$$
-i \chi \frac{|z|^2+it}{(|z|^2-it)^2} \hat Z^{\bar 1} = \alpha_0 + \text{error in $\mathcal{E}(\hat \rho^0, \Lambda^{0,1} T^* \hat X)$}.
$$
It follows that 
$$
\ddbar_{b,1} \Td \beta = \alpha_0 + E,
$$
where $\alpha_0$ is the main term in $\mathcal{E}(\hat \rho^{-2}, \Lambda^{0,1} T^* \hat X)$, and $E$ is an error term in $\mathcal{E}(\hat \rho^0, \Lambda^{0,1} T^* \hat X)$.

Recall our goal was to compute $\Pi f = \Pi \ddbar_{b,1}^* (\ddbar_{b,1} \Td \beta).$ But by Lemma~\ref{lem:PiErho0}, $$\Pi \ddbar_{b,1}^* E = 0.$$ Thus it suffices to compute $\Pi \ddbar_{b,1}^* \alpha_0$. Now define, for $\varepsilon > 0$,
$$
\alpha_{\varepsilon} := -2\pi i \chi v^{-1} \overline{\psi}^{-2} (\overline{\psi} + \varepsilon)^{-1} \hat Z^{\bar 1}.
$$
Then 
\begin{equation} \label{eq:gamma0reg}
\alpha_{\varepsilon} \in \mathcal{E}(\hat \rho^0, \Lambda^{0,1} T^* \hat X)
\end{equation} 
for all $\varepsilon > 0$, since by Theorem~\ref{t-ysmipc}, $\text{Re } \psi \geq 0$, which implies $\overline{\psi} + \varepsilon \ne 0$ on $X$. By (\ref{s1-e5b}), there exists some $s \in \mathcal{E}(\hat \rho^3)$ such that
$$
\ddbar_{b,1}^* (h \hat Z^{\bar 1}) = (-\hat Z_1 - \frac{\hat Z_1 v}{v} + s) h 
$$
for all $h \in C^{\infty} (X)$. But $\hat Z_1 \overline{\psi} = 0$ on $X$. Thus
$$
\ddbar_{b,1}^* \alpha_0  = - 2 \pi i s \chi v^{-1} \overline{\psi}^{-3} + 2 \pi i (\hat Z_1 \chi) v^{-1}  \overline{\psi}^{-3}. 
$$
Similarly,
$$
\ddbar_{b,1}^* \alpha_{\varepsilon} = - 2 \pi i s \chi v^{-1} \overline{\psi}^{-2} (\overline{\psi} + \varepsilon)^{-1} + 2 \pi i (\hat Z_1 \chi) v^{-1} \overline{\psi}^{-2} (\overline{\psi} + \varepsilon)^{-1}. 
$$
It follows that
$$
\ddbar_{b,1}^* \alpha_{\varepsilon} \to \ddbar_{b,1}^* \alpha_0 \quad \text{in $\mathcal{E}(\hat \rho^{-1})$}. 
$$
From (\ref{eq:gamma0reg}) and Lemma~\ref{lem:PiErho0}, we then have
$$
\Pi  \ddbar_{b,1}^* \alpha_0 = \lim_{\varepsilon \to 0} \Pi \ddbar_{b,1}^* \alpha_{\varepsilon}  = 0.
$$
This completes the proof of Theorem~\ref{thm4}.

\section{Appendix 1: The Green's function of the conformal Laplacian}

In order to apply our results to the positive $p$-mass theorem in~\cite{CMY}, one needs to check that the Green function $G_p$ for $-4\triangle_b+R$ at $p$ satisfies (\ref{eq:Gpexpand0}) under the assumption that the Tanaka-Webster curvature $R$ is positive on $X$. This can be done by using an argument similar to the one in Theorem~\ref{thm2.3}. 

We recall that $\triangle_b$ denotes the sublaplacian on $X$. It was shown in section 5 of~\cite{CMY} that $G_p$ has the form: $G_p=\frac{1}{2\pi\hat\rho^2}+\omega$, where $\omega\in C^1(\hat X)$ and $\omega$ satisfies the equation
\begin{equation}\label{e-gemimapaIbos}
(-4\triangle_b+R)\omega=\Td g,\ \ \Td g\in\mathcal{E}(\hat\rho^0).
\end{equation}
It is obvious that $G_p$ satisfies (\ref{eq:Gpexpand0}) if $\omega-\omega(p)\in\mathcal{E}(\hat\rho)$. We are going to prove that $\omega-\omega(p)\in\mathcal{E}(\hat\rho)$. 

First we extend $-4\triangle_b+R$ to 
\[-4\triangle_b+R:{\rm Dom\,}(-4\triangle_b+R)\subset L^2(\hat m)\To L^2(\hat m)\] 
in the standard way. Note that $-4\triangle_b+R$ is subelliptic, self-adjoint and  $-4\triangle_b+R$ has $L^2$ closed range. Since the Tanaka-Webster curvature $R$ is positive on $X$, it is easy to see that $-4\triangle_b+R:{\rm Dom\,}(-4\triangle_b+R)\subset L^2(\hat m)\To L^2(\hat m)$ is injective. Let $H:L^2(\hat m)\To{\rm Dom\,}(-4\triangle_b+R)$ be the inverse of $-4\triangle_b+R$. We have 
\begin{equation}\label{e-gemimapaIIbos}
\begin{split}
&(-4\triangle_b+R)H=I\ \ \mbox{on $L^2(\hat m)$},\\
&H(-4\triangle_b+R)=I\ \ \mbox{on ${\rm Dom\,}(-4\triangle_b+R)$}. 
\end{split}
\end{equation}
We can repeat the $L^2$ estimates of Kohn and show that for every $k\in\mathbb N_0$, there is a 
constant $c_k>0$ such that 
\begin{equation}\label{e-gemimapaIIIbos}
\norm{\hat\nabla^{k+2}_bHu}_{\hat m}\leq c_k\norm{\hat\nabla^{k}_bu}_{\hat m},\ \ \forall u\in C^\infty(\hat X), 
\end{equation}
and the distribution kernel $H(x,y)$ of $H$ is $C^\infty$ away from the diagonal. 

Now, we claim that $H$ is a smoothing operator of order $2$. Let $B(x,r)$ be a small ball and let $(x_1,x_2,x_3)$ be local coordinates on $B(x,r)$. We first observe that for any smooth function $f$, we have 
\begin{equation}\label{e-gemimapaIVbos}
\begin{split}
\norm{f(x)}_{L^\infty(B(x,r))}&\leq c_0\int_{B(x,2r)}\abs{\frac{\pr^3f}{\pr x_1\pr x_2\pr x_3}(x)}\hat m(x)\\
&\leq c_1\int_{B(x,2r)}\abs{\hat\nabla^4_b f(x)}\hat m(x),
\end{split}
\end{equation}
where $c_0>0$ and $c_1>0$ are constants independent of $f$ and $r$. Let $\phi$ be a normalized bump function in the ball $B(x,r)$ and let $k\in\mathbb N_0$. From  \eqref{e-gemimapaIIIbos} and \eqref{e-gemimapaIVbos}, we have 
\begin{equation}\label{e-gemimapaVbos}
\begin{split}
\norm{\hat\nabla^{k}_bH\phi}_{L^\infty(B(x,r))}&\leq c_1\int_{B(x,2r)}\abs{(\hat\nabla^{k+4}_bH\phi)(x)}\hat m(x)\\
&\leq c_2r^2\norm{\hat\nabla^{k+4}_bH\phi}_{\hat m}\leq\Td c_kr^2\norm{\hat\nabla^{k+2}_b\phi}_{\hat m}\leq \hat c_kr^{2-k},
\end{split}
\end{equation}
where $c_1>0$, $c_2>0$, $\Td c_k>0$ and $\hat c_k>0$ are constants independent of $r$, $\phi$ and $x$. Thus, $H$ satisfies the cancellation property for a smoothing operator of order 2. From \eqref{e-gemimapaVbos} and \eqref{e-gemimapaIIIbos}, we can repeat the methods as in Christ~\cite{Ch88I}, ~\cite{Ch88II} and Koenig~\cite{Koe02} and conclude that for all multi-indices $\alpha_1$, $\alpha_2$, we have, 
\begin{equation}\label{e-gemimapaVIbos}
\abs{(\nabla^{\alpha_1}_b)_x(\nabla^{\alpha_2}_b)_yH(x,y)}\leq C_{\alpha}\vartheta(x,y)^{-2-|\alpha|}, \forall (x,y)\in\hat X\times\hat X,\ \ x\neq y,
\end{equation} 
where $C_{\alpha}>0$ is a constant. This shows that $H$ is a smoothing operator of order 2.

Now, we are ready to prove that $\omega-\omega(p)\in\mathcal{E}(\hat\rho)$. From \eqref{e-gemimapaIbos} and \eqref{e-gemimapaIIbos}, we have $\omega=H\Td g$. Fix $k\in\mathbb N_0$ and fix a point $x_0 \neq p$, $x_0$ is in some small neighbourhood $W$ of $p$. Let $r=\frac{1}{4}\vartheta(x_0,p)$, and $\eta$ be a normalized bump function supported in $B(x_0,r)$, with $\eta=1$ on $B(x_0,r/2)$. Then, 
\[
\begin{split}
&|\hat{\nabla}_b^{k+1}\omega(x_0)|=
|\hat{\nabla}_b^{k+1}H\Td g(x_0)|\\
& \leq |\hat{\nabla}_b^{k+1}H (\eta\Td g)(x_0)| + |\hat{\nabla}_b^{k+1}H((1-\eta)\Td g)(x_0)|
\end{split}\]
and $\eta\Td g(x)$ is a normalized bump function on $B(x_0,r)$. So by \eqref{e-gemimapaVbos}, we see that 
\begin{equation}\label{e-gemimapaVI-Ibos}
|\hat{\nabla}_b^{k+1}H(\eta \Td g)(x_0)|\leq C_kr^{1-k},
\end{equation}
where $C_k>0$ is a constant independent of $x_0$ and $r$. By using \eqref{e-gemimapaVIbos}, $\hat{\nabla}_b^{k+1}H((1-\eta)\Td g)(x_0)$ can be estimated by writing out the integral directly: 
\[\hat{\nabla}_b^{k+1}H((1-\eta)\Td g)(x_0) = \int (\hat{\nabla}_b^{k+1}H)(x_0,y) (1-\eta)(y)\Td g(y)\hat m(y),\]
this integral is dominated by 
\begin{equation}\label{e-gemimapaVIIbos}
D_k\int_{\vartheta(y,x_0)\geq\frac{1}{2}r}\vartheta(y,x_0)^{-3-k}\hat m(y)\leq E_kr^{-k}\int_{\hat X}\vartheta(y,x_0)^{-3}\hat m(y)\leq F_kr^{-k},
\end{equation}
where $D_k>0$, $E_k>0$ and $F_k>0$ are constants independent of the point $x_0$ and $r$. From \eqref{e-gemimapaVI-Ibos} and \eqref{e-gemimapaVIIbos}, we conclude that $\hat\nabla_b\omega\in\mathcal{E}(\hat\rho^0)$ and hence $\omega-\omega(p)\in\mathcal{E}(\hat\rho)$. 

\section{Appendix 2: Subelliptic estimates for $\ddbar_b$}

In this appendix we present a proof of Proposition~\ref{p-keyest} and \ref{p-keyest2}. As is well-known, the crux of the matter is to prove a normalized subelliptic estimate on a unit cube, and rescale to a ball of radius $r$. It is this normalized subelliptic estimate we will focus on below.

Suppose on $\mathbb{R}^3$, $T = \frac{\partial}{\partial x_3}$, and on the cube $Q_2 :=(-2,2)^3$, there is a (complex) vector field $Z$ such that $[Z, \Zbar] = -iT + bZ + \overline{b} \Zbar$ and $[Z,T] = cZ + d\Zbar + eT$. Fix a sequence of positive numbers $c_k$. The only assumptions we make on $b,c,d$ and $e$ are that they are $C^{\infty}$ on $Q_2$, and that their $C^k$ norms are bounded by $c_k$ for all $k$. 

We also need to assume the following condition on $Z$: Write $Z = \sum_{i=1}^3 A_i(x) \frac{\partial}{\partial x_i}$ on $Q_2$. Then the only assumption we make on the $A_i$'s is that $|A_i(x)| \leq 1$. 

Note that $Z$ and $\Zbar$ are only defined on $Q_2$. They will never hit any function that is not supported in $Q_2$, whereas $T$ could hit a function that is defined on all of $\mathbb{R}^3$.

Suppose also that we have a smooth contact form $\theta$ on $Q_2$, so that $\theta(T) = 1$, and $\theta(Z) = \theta(\Zbar) = 0$ on $Q_2$. One then has a measure $\m$ on $Q_2$.
We also assume that the formal adjoint of $Z$ with respect to $L^2(\m)$ on $Q_2$ is given by $-\Zbar + a$ for some $C^{\infty}$ function $a$, where again the only assumption on $a$ is that its $C^k$ norm is bounded by $c_k$ for all $k$. Similarly for the adjoint of $\Zbar$. We also assume that on $Q_2$, $\m = \rho^2 dx$ where $dx = dx_1 dx_2 dx_3$ is the Lebesgue measure on $\mathbb{R}^3$ and $\rho $ is a positive smooth function on $Q_2$. The only assumptions on $\rho $ are that $c_0^{-1} \leq \rho  \leq 1$ and that its $C^1$ norm is bounded by $c_1$. There are no other assumptions on $\theta$.

We shall also fix two functions $\eta$, $\teta$ such that they are $C^{\infty}_c$ with support in $Q_2$, identically equal to 1 on $Q_1:=(-1,1)^3$, and $\teta \equiv 1$ on a neighborhood of the support of $\eta$.

Write $\nabla_b u$ for $(Zu, \Zbar u)$. We claim the following proposition:

\begin{prop} \label{prop:normalizedsub}
For all functions $u \in C^{\infty}(Q_2)$ and $k \geq 1$, we have 
$$
\|\nabla_b^k (\eta u)\| \leq C_k \left(\|\nabla_b^{k-1} \Zbar (\teta u)\| + \|\teta u\| + \|\teta v\|\right)
$$
where $v$ is any solution to the equation $Zv = u$ on $Q_2$, all norms are $L^2(\m)$ norms, and $C_k$ depends only on the chosen sequence $c_k$ and on $\eta$, $\teta$ (but not otherwise on the vector fields, the coefficients $a, b, c, d, e, A_i$ or $\theta$).
\end{prop}

We remark that if we have $Zv + \alpha v = u$ instead of $Zv = u$, where $\alpha$ is a fixed $C^{\infty}$ function on $Q_2$, then the above theorem still holds. See the end of this section for a discussion about that. Propositions~\ref{p-keyest} and \ref{p-keyest2} now follows easily by a well-known rescaling procedure. We omit the details.

To prove Proposition~\ref{prop:normalizedsub}, recall on $\mathbb{R}^3$ we have the Lebesgue measure $dx = dx_1 dx_2 dx_3$, and there is the Fourier transform defined by $\widehat{u}(\xi) = \int u(x) e^{-2\pi i x \xi} dx$. Let $\Psi^+$ be a smooth function of $\xi$, such that it is (Euclidean) homogeneous of degree 0 outside the unit ball $\{|\xi| \leq 1\}$, equal to 1 on $\{ \xi_3 > 2\varepsilon_0 ( |\xi_1| + |\xi_2| ) \}$ there, and equal to 0 on both $\{\xi_3 < \varepsilon_0 (|\xi_1| + |\xi_2|)\} \cap \{|\xi| > 1\}$ and $\{|\xi| \leq \frac{1}{2}\}$. Here $\varepsilon_0$ is a small positive absolute constant to be chosen. Let $\Psi^-(\xi) = \Psi^+(\xi_1,\xi_2,-\xi_3)$, and $\Psi^0 = 1-\Psi^+-\Psi^-$. Let $\Lambda^+, \Lambda^0, \Lambda^-$ be the Fourier multipliers corresponding to $\Psi^+, \Psi^0, \Psi^-$ respectively. For instance, if $U$ is a function on $\mathbb{R}^3$, then $\widehat{\Lambda^+ U}(\xi) := \Psi^+(\xi) \widehat{U}(\xi)$. Then $\Lambda^+ + \Lambda^0 + \Lambda^-$ is the identity operator on $L^2(dx)$. 

We shall also fix a sequence of $C^{\infty}$ functions $\Psi^+_0, \Psi^+_1, \Psi^+_2, \dots$ of $\xi$, with $\Psi^+_0 = \Psi^+$, such that each $\Psi^+_k$ is (Euclidean) homogeneous of degree 0 outside the unit ball $\{|\xi| \leq 1\}$, equal to 1 on a neighborhood of the support of $\Psi^+_{k-1}$, and equal to 0 on both $\{\xi_3 < (1/2)\varepsilon_0 (|\xi_1| + |\xi_2|)\} \cap \{|\xi| > 1\}$ and $\{|\xi| \leq \frac{1}{4}\}$. We shall denote the corresponding Fourier multipliers $\Lambda^+_k$. Note $\Lambda^+_k \Lambda^+_{k-1} = \Lambda^+_{k-1}$ for all $k \geq 1$.

In addition, we fix a Fourier multiplier operator $\tilde{\Lambda}^-$, with symbol $\tilde{\Psi}^-$ that is supported on $\{\xi_3 < 0\}$, such that $\tilde{\Lambda}^- \Lambda^- = \Lambda^-$.

Finally, we fix a sequence of $C^{\infty}_c$ functions $\eta_0, \eta_1, \eta_2, \dots$, with $\eta_0 = \eta$, such that each $\eta_k$ has support in $Q_2$, $\eta_{k+1} \equiv 1$ on a neighborhood of the support of $\eta_k$ for all $k$, and $\teta \equiv 1$ on the support of $\eta_k$ for all $k$.

Now suppose we are given $u \in C^{\infty}(Q_2)$. We write $u^+$ for $\eta_1 \Lambda^+ (\eta u)$, and similarly $u^0$ and $u^-$. Note that then $\eta u = u^+ + u^0 + u^-$. 

\subsection{Estimate for $u^-$}

First we prove that for any $k \geq 1$,
\begin{equation}\label{eq:u-}
\|\nabla_b^k u^-\| \leq C_k \left(\|\nabla_b^{k-1} \Zbar u^-\| + \|\eta u\|\right).
\end{equation}

A useful lemma is the following:

\begin{lem} \label{lem:pdou-}
For every $k \geq 1$, there exist Euclidean pseudodifferential operators $S_{-1}$ and $S_{-k}$, smoothing of orders 1 and $k$ respectively, so that 
$$
u^- = \tilde{\Lambda}^- u^- + S_{-1} u^- + S_{-k} (\eta u).
$$
\end{lem}

Here a Euclidean pseudodifferential operator is said to be smoothing of order $k$, if its symbol is in the H\"ormander class $S^{-k}_{1,0}$.

\begin{proof}
To see this, write
\begin{align}
u^- 
&= \eta_1 \Lambda^- \eta u \notag \\
&= \eta_1 \tilde{\Lambda}^- \eta_2 \Lambda^- \eta u + \eta_1 \tilde{\Lambda}^- (1-\eta_2) \Lambda^- \eta u \notag \\
&= \tilde{\Lambda}^- \eta_1 \Lambda^- \eta u + [\eta_1, \tilde{\Lambda}^-] \eta_2 \Lambda^- \eta u + \eta_1 \tilde{\Lambda}^- (1-\eta_2) \Lambda^- \eta u. \label{eq:u-lemma}
\end{align}
The last term here is $S_{-k} (\eta u)$ for some (Euclidean) pseudodifferential operator that is smoothing of order $k$, because one can pick some $C^{\infty}_c$ function $\zeta$ such that $\zeta \equiv 1$ on the support of $\eta$ and $\eta_1 \equiv 1$ on the support of $\zeta$; this is possible because $\eta_1 \equiv 1$ on a neighborhood of the support of $\eta$. Then writing $\eta u$ as $\zeta \eta u$, and commuting the $\zeta$ past $\Lambda^-$ to hit $(1-\eta_2)$, the last term above is just
$$
\eta_1 \tilde{\Lambda}^- (1-\eta_2) [[[\Lambda^-, \zeta], \zeta],\dots,\zeta] (\eta u) = S_{-k} (\eta u).
$$
With the same choice of $\zeta$, the second term in (\ref{eq:u-lemma}) can be written as
$$
[\eta_1, \tilde{\Lambda}^-] \eta_2 \eta_1 \Lambda^- \eta u + [\eta_1, \tilde{\Lambda}^-] \eta_2 (1-\eta_1) \Lambda^- \eta u = [\eta_1, \tilde{\Lambda}^-] u^- + [\eta_1, \tilde{\Lambda}^-] \eta_2 (1-\eta_1) \Lambda^- \eta u,
$$
and by the same argument above, the last term here is a pseudodifferential operator of order $-k$ acting on $\eta u$. It follows that this second term in (\ref{eq:u-lemma}) is of the form $S_{-1} u^- + S_{-k} (\eta u)$. Finally, the first term in (\ref{eq:u-lemma}) is just $\tilde{\Lambda}^- u^-$. This completes our proof of this lemma.
\end{proof}

Also, to prove (\ref{eq:u-}), it suffices to prove that 
\begin{equation}\label{eq:u-sum}
\|\nabla_b^k u^-\| \leq C_k \left(\sum_{l=0}^{k-1} \|\nabla_b^l \Zbar u^-\| + \|\eta u\|\right),
\end{equation}
in view of the following interpolation inequality:

\begin{lem} \label{lem:interpolate}
If $k \geq 1$, then $$\|\nabla_b^l \Zbar u\|^2 \leq C_k \left(\|\nabla_b^{k-1} \Zbar u\|^2 + \|u\|^2\right)$$ for all $0 \leq l \leq k-1$, for any function $u$ that is smooth and compactly supported in $Q_2$.
\end{lem}

\begin{proof}
One proves, by induction on $l$ beginning at $l=0$, that for any $\varepsilon > 0$, there exists $C_{k,\varepsilon}$ such that
$$
\|\nabla_b^l \Zbar u\|^2 \leq \varepsilon \sum_{j=0}^{k-1} \|\nabla_b^j \Zbar u\|^2 + C_{k,\varepsilon} \|u\|^2
$$
for all $0 \leq l \leq k-2$. The key is that 
\begin{align*}
(\nabla_b^l \Zbar u, \nabla_b^l \Zbar u) &= (\nabla_b^{l-1} \Zbar u, \nabla_b^{l+1} \Zbar u) + O(\|\nabla_b^{l-1} u\|\|\nabla_b^l u\|) \\
&\leq \varepsilon (\|\nabla_b^{l+1} \Zbar u\|^2 + \|\nabla_b^{l} \Zbar u\|^2) + C_{l,\varepsilon} \|\nabla_b^{l-1} \Zbar u\|^2.
\end{align*}
Once this is established, the lemma follows easily by summing over $l$.
\end{proof}

Now, to prove (\ref{eq:u-sum}), we proceed by induction on $k$. 

When $k = 1$, it suffices to bound $\|Z u^-\|^2$. Denote $(\cdot, \cdot)$ the inner product in $L^2(\m)$. Then
\begin{align*}
&\|Z u^-\|^2 \\
=& -(\Zbar Z u^-,  u^-) + O(\|u^-\| \|\nabla_b u^-\|)\\
=& -(Z \Zbar u^-,  u^-) - (iTu^-, u^-) + O(\|u^-\| \|\nabla_b u^-\|)\\
=& \|\Zbar u^-\|^2 - (iTu^-, u^-) + O(\|u^-\| \|\nabla_b u^-\|).
\end{align*}
Now $(iT u^-, u^-) = \langle \rho iT u^-, \rho u^- \rangle$, where $\langle \cdot, \cdot \rangle$ is the $L^2$ inner product with respect to the Lebesgue measure $dx$. Hence 
$$
(iTu^-, u^-) = \langle iT \rho u^-, \rho u^-\rangle + O(\|u^-\|^2).
$$
Also, by the above lemma, 
\begin{align*}
\rho u^- 
&= \tilde{\Lambda}^- \rho u^- + [\rho, \tilde{\Lambda}^-] u^- + \rho S_{-1} u^- + \rho S_{-1}(\eta u)\\
&= \tilde{\Lambda}^- \rho u^- + S_{-1} u^- + S_{-1}(\eta u).
\end{align*}
So
\begin{align*}
(iTu^-, u^-) =& \langle iT\tilde{\Lambda}^- \rho u^-, \tilde{\Lambda}^- \rho u^- \rangle + \langle iT\tilde{\Lambda}^- \rho u^-, S_{-1} u^- \rangle \\
& \quad+ \langle iT\tilde{\Lambda}^- \rho u^-, S_{-1} (\eta u) \rangle + O(\|u^-\|^2) + O(\|\eta u\|^2).
\end{align*}
But the second and third terms is $O(\|u^-\|^2) + O(\|\eta u\|^2)$ (one just needs to integrate by parts in $T$ and let $T$ fall on $S_{-1}$), and the first term is $$\langle iT \tilde{\Lambda}^- \rho u^-, \tilde{\Lambda}^- \rho u^- \rangle = \int -2\pi \xi_3 |\tilde{\Psi}^-(\xi)|^2 |\widehat{\rho u^- }(\xi)|^2 d\xi$$
which is non-negative since $\xi_3 < 0$ on the support of $\tilde{\Psi}^-$.
Hence altogether 
$$
\|Z u^-\|^2 \leq C \|\Zbar u^-\|^2 + O(\|u^-\|^2) + O(\|\eta u\|^2) + O(\|u^-\| \|\nabla_b u^-\|),
$$
and using $\|u^-\|\|\nabla_b u^-\| \leq \delta \|\nabla_b u^-\|^2 + \delta^{-1} \|u^-\|^2$ and $\|u^-\|^2 \leq C \|\eta u\|^2$, we get
$$\|Z u^-\|^2 \leq C (\|\Zbar u^-\|^2 + \|\eta u\|^2)$$
as desired.

Next, suppose (\ref{eq:u-sum}) has been proved for $k-1$ for some $k \geq 2$. We prove the same estimate for $k$. To do so, we first prove that for all $0 \leq m \leq \lfloor \frac{k}{2} \rfloor$ and all $\varepsilon > 0$, there exists $C_{\varepsilon}$ such that 
\begin{equation}\label{eq:u-m}
\|T^m \nabla_b^{k-2m} u^- \|^2 \leq \varepsilon \|\nabla_b^k u^-\|^2 + C_{\varepsilon} \left( \|\nabla_b^{k-1}\Zbar u^-\|^2 + \sum_{l=0}^{k-1} \|\nabla_b^l u^-\|^2 + \|\eta u\|^2 \right).
\end{equation}
In fact the desired inequality (\ref{eq:u-}) for $k$ follows readily from the above inequality when $m = 0$.

To prove (\ref{eq:u-m}), we proceed in two steps. First, we prove that for all $0 \leq m \leq \lfloor \frac{k}{2} \rfloor$ and all $\varepsilon > 0$, there exists $C_{\varepsilon}$ such that
\begin{equation}\label{eq:u-msecond}
\|T^m Z^{k-2m} u^- \|^2 \leq \varepsilon \|\nabla_b^k u^-\|^2 + C_{\varepsilon} \left( \|\nabla_b^{k-1}\Zbar u^-\|^2 + \sum_{l=0}^{k-1} \|\nabla_b^l u^-\|^2 + \|\eta u\|^2 \right).
\end{equation}
Next, we prove by induction on $m$, beginning from $m$ that is as large as possible, that (\ref{eq:u-m}) holds.

In the first step, there are two cases: either $k-2m = 0$ (which occurs only when $k$ is even), or $k-2m \geq 1$. 

In the first case, we need to estimate $\|T^m u^-\|^2$. Now $m \geq 1$, and
\begin{align*}
&\|T^m u^-\|^2 \\ 
=& (T^m u^-, i(Z \Zbar - \Zbar Z) T^{m-1} u^-) + O(\|T^m u^-\| \|\nabla_b^{k-1} u^-\|) \\
=& -(T^m u^-, i\Zbar Z T^{m-1} u^-) + (T^m u^-, iZ T^{m-1} \Zbar u^-) + O(\|T^m u^-\| \sum_{l=0}^{k-1} \|\nabla_b^{l} u^-\|)\\
=& -(T^m u^-, i\Zbar Z T^{m-1} u^-) + O(\|T^m u^-\| \|\nabla_b^{k-1} \Zbar u^-\|) + O(\|T^m u^-\| \sum_{l=0}^{k-1} \|\nabla_b^{l} u^-\|).
\end{align*}
But 
\begin{align*}
&(T^m u^-, i\Zbar Z T^{m-1} u^-) \\
=& (iT (Z T^{m-1} u^-), Z T^{m-1} u^-) + O(\|T^m u^-\| \|\nabla_b^{k-1} u^-\|) + O( \|\nabla_b^{k-1} u^-\|^2)\\
=& \langle iT (\rho Z T^{m-1} u^-), \rho Z T^{m-1} u^- \rangle + O(\|T^m u^-\| \|\nabla_b^{k-1} u^-\|) + O( \|\nabla_b^{k-1} u^-\|^2).
\end{align*}
Also by Lemma~\ref{lem:pdou-}, 
\begin{align*}
&\rho Z T^{m-1} u^- \\
=& \tilde{\Lambda}^- (\rho Z T^{m-1} u^-) + [\rho Z T^{m-1}, \tilde{\Lambda}^-] u^- + \rho Z T^{m-1} S_{-1}(u^-) + \rho Z T^{m-1} S_{-k}(\eta u)\\
=& \tilde{\Lambda}^- (\rho Z T^{m-1} u^-) + S_{-1} \sum_{l=0}^{k-1} \nabla_b^l u^- + S_{-1}(\eta u).
\end{align*}
Hence
\begin{align*}
&\langle iT (\rho Z T^{m-1} u^-), \rho Z T^{m-1} u^- \rangle\\
=& \langle iT \tilde{\Lambda}^- (\rho Z T^{m-1} u^-), \tilde{\Lambda}^- (\rho Z T^{m-1} u^-) \rangle + \langle iT \tilde{\Lambda}^- (\rho Z T^{m-1} u^-), S_{-1} \sum_{l=0}^{k-1} \nabla_b^l u^- \rangle \\ & \quad + \langle iT \tilde{\Lambda}^- (\rho Z T^{m-1} u^-), S_{-1}(\eta u) \rangle + O(\sum_{l=0}^{k-1} \|\nabla_b^l u^-\|^2) + O(\|\eta u\|^2),
\end{align*}
where the first term is non-negative, and the second and third terms are $O(\sum_{l=0}^{k-1} \|\nabla_b^l u^-\|^2) + O(\|\eta u\|^2)$ after integrating by parts in $T$. Altogether, we get
$$
\|T^m u^-\|^2 \leq \varepsilon \|\nabla_b^k u^-\|^2 + C_{\varepsilon} \left( \|\nabla_b^{k-1}\Zbar u^-\|^2 + \sum_{l=0}^{k-1} \|\nabla_b^l u^-\|^2 + \|\eta u\|^2 \right) 
$$
as desired.

Next, in the second case, we need to estimate $\|T^m Z^{k-2m} u^-\|^2$ when $k-2m \geq 1$. The strategy is the same as the one when we dealt with the case when $k=1$ and $m=0$. One observes that
\begin{align*}
&\|T^m Z^{k-2m} u^-\|^2 \\
=& -(T^m \Zbar Z^{k-2m} u^-, T^m Z^{k-2m-1} u^-) + O(\sum_{l=0}^{k-1} \|\nabla_b^l u^-\| \|\nabla_b^k u^-\|) + O(\sum_{l=0}^{k-1} \|\nabla_b^l u^-\|^2)  \\
=& -(T^m Z^{k-2m} \Zbar u^-, T^m Z^{k-2m-1} u^-) - (k-2m) (iT (T^m Z^{k-2m-1} u^-), T^m Z^{k-2m-1} u^-) \\
&+ O(\sum_{l=0}^{k-1} \|\nabla_b^l u^-\| \|\nabla_b^k u^-\|) +O(\sum_{l=0}^{k-1} \|\nabla_b^l u^-\|^2) \\
=& \, \|T^m Z^{k-2m-1} \Zbar u^-\|^2 - [2(k-2m)-1] \langle iT(\rho T^m Z^{k-2m-1} u^-), \rho T^m Z^{k-2m-1} u^- \rangle \\
& + O(\sum_{l=0}^{k-1} \|\nabla_b^l u^-\| \|\nabla_b^k u^-\|) + O(\sum_{l=0}^{k-1} \|\nabla_b^l u^-\|^2).
\end{align*}
Now $[2(k-2m)-1] > 0$, and by Lemma~\ref{lem:pdou-}, 
\begin{align*}
&\rho T^m Z^{k-2m-1} u^- \\
=& \tilde{\Lambda}^- (\rho T^m Z^{k-2m-1} u^-) + [\rho T^m Z^{k-2m-1}, \tilde{\Lambda}^-] u^- + \rho T^m Z^{k-2m-1} S_{-1}(u^-) + \rho T^m Z^{k-2m-1} S_{-k}(\eta u)\\
=& \tilde{\Lambda}^- (\rho T^m Z^{k-2m-1} u^-) + S_{-1} \sum_{l=0}^{k-1} \nabla_b^l u^- + S_{-1}(\eta u).
\end{align*}
Hence
\begin{align*}
&\langle iT (\rho T^m Z^{k-2m-1} u^-), \rho T^m Z^{k-2m-1} u^- \rangle\\
=& \langle iT \tilde{\Lambda}^- (\rho T^m Z^{k-2m-1} u^-), \tilde{\Lambda}^- (\rho T^m Z^{k-2m-1} u^-) \rangle + \langle iT \tilde{\Lambda}^- (\rho T^m Z^{k-2m-1} u^-), S_{-1} \sum_{l=0}^{k-1} \nabla_b^l u^- \rangle \\ & \quad + \langle iT \tilde{\Lambda}^- (\rho T^m Z^{k-2m-1} u^-), S_{-1}(\eta u) \rangle + O(\sum_{l=0}^{k-1} \|\nabla_b^l u^-\|^2) + O(\|\eta u\|^2),
\end{align*}
where the first term is non-negative, and the second and third terms are $O(\sum_{l=0}^{k-1} \|\nabla_b^l u^-\|^2) + O(\|\eta u\|^2)$ after integrating by parts in $T$. Altogether,
$$
\|T^m Z^{k-2m} u^-\|^2 \leq \varepsilon \|\nabla_b^k u^-\|^2 + C_{\varepsilon} \left( \|\nabla_b^{k-1}\Zbar u^-\|^2 + \sum_{l=0}^{k-1} \|\nabla_b^l u^-\|^2 + \|\eta u\|^2 \right) 
$$
as desired. This finishes our first step in proving (\ref{eq:u-m}).

Now to complete the proof of (\ref{eq:u-m}), we proceed by induction on $m$, beginning with $m$ that is as big as possible. In that case $k-2m$ is either 0 or 1. 
Both cases follow right away by what we have proved above in the first step. Now we prove the inequality (\ref{eq:u-m}) for $m$, assuming the inequality has been proved for all strictly bigger $m$'s. Then we need to estimate $\|T^m \nabla_b^{k-2m} u^-\|^2$. Consider $T^m \nabla_b^{k-2m} u^-$. If all the $\nabla_b$'s are $Z$, then this follows again from what we have proved above. If one of the $\nabla_b$'s is $\Zbar$, then one only needs to commute the $\Zbar$ all the way through the other $\nabla_b$'s to get $T^m \nabla_b^{k-2m-1} \Zbar u^-$, up to an error that either has fewer $\nabla_b$ derivatives, or an error of the form $T^{m+1} \nabla_b^{k-2m-2} u^-$. For example,
\begin{align*}
& \|T^m \Zbar \nabla_b^{k-2m-1} u^-\|^2 \\
=& \|T^m \nabla_b^{k-2m-1} \Zbar u^-\|^2 + O(\|T^{m+1} \nabla_b^{k-2m-2} u^-\|^2) + O(\sum_{l=0}^{k-1} \|\nabla_b^{l} u^-\|^2).
\end{align*}
The first error term can then be estimated by our induction hypothesis on $m$. Hence
$$
\|T^m \Zbar \nabla_b^{k-2m-1} u^-\|^2 \leq \varepsilon \|\nabla_b^k u^-\|^2 + C_{\varepsilon} \left( \|\nabla_b^{k-1}\Zbar u^-\|^2 + \sum_{l=0}^{k-1} \|\nabla_b^l u^-\|^2 + \|\eta u\|^2 \right).
$$
This completes the proof of (\ref{eq:u-m}), and thus the proof of (\ref{eq:u-}).

\subsection{Estimate for $u^0$}

Next we prove that for any $k \geq 1$ and any $\varepsilon > 0$,
\begin{equation} \label{eq:u0}
\|\nabla_b^k u^0\| \leq \|\nabla_b^{k-1} \Zbar u^0\|^2 + \varepsilon \|\nabla_b^k (\eta u)\|^2 + C_{k,\varepsilon} \sum_{l=0}^{k-1} \|\nabla_b^{l} (\eta u)\|.
\end{equation}
We proceed by induction on $k$ exactly as before. 

When $k = 1$, we only need to estimate $|(iT u^0, u^0)|$. But $$|(iT u^0, u^0)| \leq \varepsilon \|Tu^0\|^2 + \varepsilon^{-1} \|\eta u\|^2,$$ and
$$
\|T u^0\|^2
\leq  \|T \Lambda^0 (\rho  \eta u)\|^2 + C\|\eta u\|^2.
$$
Taking Fourier transform,
\begin{align*}
&\|T \Lambda^0 (\rho  \eta u)\|^2 \\
\leq & \int (1 + 2 \varepsilon_0^2 (|\xi_1|^2 + |\xi_2|^2)) |\widehat{\Lambda^0 (\rho \eta u)}(\xi)|^2 d\xi\\
\leq & 2 \varepsilon_0^2 \left( \left\| \frac{\partial}{\partial x_1} \Lambda^0 (\rho \eta u) \right \|^2 + \left\| \frac{\partial}{\partial x_2} \Lambda^0 (\rho \eta u) \right \|^2 \right ) +  \|\eta u\|^2 \\
\leq & 2 \varepsilon_0^2 \left( \left\| \frac{\partial}{\partial x_1} \eta_1 \Lambda^0 (\rho \eta u) \right \|^2 + \left\| \frac{\partial}{\partial x_2} \eta_1 \Lambda^0 (\rho \eta u) \right \|^2 \right ) +  \|\eta u\|^2.
\end{align*}
(The last line follows since $\rho \eta u = \eta_1 \rho  \eta u$ and one can commute the $\eta_1$ past $\Lambda^0$ to obtain a better error.)

Now the key observation is that on $Q_2$, $\frac{\partial}{\partial x_1}$ and $\frac{\partial}{\partial x_2}$ can be written as linear combinations of $Z$, $\Zbar$ and $T$ with coefficients that are bounded by an absolute constant. In fact we only need to bound the coefficients of the the inverse of the matrix whose first column is $(A_1,A_2,A_3)$, the second column is the conjugate of the first, and the third column is (0,0,1). From $\m(Z,\Zbar,T) = d\theta(Z,\Zbar) = -\theta(Z,\Zbar) = \theta(iT) = i$, $\m = \rho ^2 dx$, and $\rho  \leq 1$, we have $|dx(Z,\Zbar,T)| \geq 1$, i.e. the determinant of the matrix to be inverted is bounded below by 1. Together with the assumed bounds on the $A_i$'s, we obtain our key observation. 

Hence, continuing from above, 
$$
\|T \Lambda^0 (\rho  \eta u)\|^2
\leq  C_0 \varepsilon_0^2 \left( \left\| \nabla_b \Lambda^0 (\rho \eta u) \right \|^2 + \left\| T \Lambda^0 (\rho \eta u) \right \|^2 \right) +  \|\eta u\|^2,
$$
which implies 
\begin{align*}
&\|T \Lambda^0 (\rho  \eta u)\|^2 \\
\leq & C \left\| \nabla_b \Lambda^0 (\rho \eta u) \right \|^2 + \|\eta u\|^2\\
\leq & C (\left\| \nabla_b (\eta u) \right \|^2 + \|\eta u\|^2)
\end{align*}
if $\varepsilon_0$ was chosen to be sufficiently small. One then completes the proof of the case $k=1$ as before.

Next, we prove by induction on $m$ the following for any $\varepsilon > 0$:
\begin{equation}\label{eq:u0m}
\|T^m \nabla_b^{k-2m} u^0 \|^2 \leq C_{\varepsilon} \left( \|\nabla_b^{k-1}\Zbar u^0\|^2 + \sum_{l=0}^{k-1} \|\nabla_b^l (\eta u)\|^2\right) + \varepsilon \|\nabla_b^k (\eta u)\|^2.
\end{equation}

First, suppose $k-2m = 0$. Then we need only estimate $|(iT ZT^{m-1} u^0, ZT^{m-1} u^0)|$. One certainly has $$|(iT ZT^{m-1} u^0, ZT^{m-1} u^0)| \leq \varepsilon \|TZT^{m-1}u^0\|^2 + \varepsilon^{-1} \sum_{l=0}^{k-1} \|\nabla_b^{l} (\eta u)\|^2.$$ To estimate $\|TZT^{m-1}u^0\|^2$, we write $TZT^{m-1}u^0 = T\Lambda^0 ZT^{m-1} (\eta u) + T[ZT^{m-1}\eta_1, \Lambda^0](\eta u)$. The second term is bounded by $$\sum_{l=0}^k \|\nabla_b^l (\eta u)\|.$$ The first term satisfies
$$
\|T\Lambda^0 ZT^{m-1} (\eta u)\|^2 \leq \|\nabla_b^{k-1}(\eta u)\|^2 + C_0 \varepsilon_0^2 \|\nabla_b \Lambda^0 ZT^{m-1} (\eta u)\|^2 + C_0 \varepsilon_0^2 \|T\Lambda^0 ZT^{m-1} (\eta u)\|^2,
$$
from which it follows that
\begin{align*}
&\|T\Lambda^0 ZT^{m-1} (\eta u)\|^2\\
 \leq & \|\nabla_b^{k-1}(\eta u)\|^2 + C \|\nabla_b \Lambda^0 ZT^{m-1} (\eta u)\|^2 \\
 \leq & C \|\nabla_b^{k-1}(\eta u)\|^2 + C \|\nabla_b^k (\eta u)\|^2 \\
\end{align*}
by our choice of $\varepsilon_0$. Hence
$$
|(iT ZT^{m-1} u^0, ZT^{m-1} u^0)| \leq \varepsilon \|\nabla_b^k (\eta u)\|^2 + C_{\varepsilon}  \sum_{l=0}^{k-1} \|\nabla_b^l (\eta u)\|^2 
$$
and one finishes the proof for the case $k-2m = 0$ as before.

Next, when $k-2m=1$, we need to estimate $|(iT T^m u^0, T^m u^0)|$. One certainly has $$|(iT T^{m} u^0, T^{m} u^0)| \leq \varepsilon \|T^{m+1}u^0\|^2 + \varepsilon^{-1} \sum_{l=0}^{k-1} \|\nabla_b^{l} (\eta u)\|^2.$$ To estimate $\|T^{m+1}u^0\|^2$, we write $T^{m+1}u^0 = T\Lambda^0 T^{m} (\eta u) + T[T^{m}\eta_1, \Lambda^0](\eta u)$. The second term is bounded by $$\sum_{l=0}^k \|\nabla_b^l (\eta u)\|.$$ The first term satisfies
$$
\|T\Lambda^0 T^{m} (\eta u)\|^2 \leq \|\nabla_b^{k-1}(\eta u)\|^2 + C_0 \varepsilon_0^2 \|\nabla_b \Lambda^0 T^{m} (\eta u)\|^2 + C_0 \varepsilon_0^2 \|T\Lambda^0 T^{m} (\eta u)\|^2,
$$
from which it follows that
\begin{align*}
&\|T\Lambda^0 T^{m} (\eta u)\|^2\\
 \leq & \|\nabla_b^{k-1}(\eta u)\|^2 + C \|\nabla_b \Lambda^0 T^{m} (\eta u)\|^2 \\
 \leq & C \|\nabla_b^{k-1}(\eta u)\|^2 + C \|\nabla_b^k (\eta u)\|^2.
\end{align*}
Hence
$$
|(iT T^m u^0, T^m u^0)| \leq \varepsilon \|\nabla_b^k (\eta u)\|^2 + C_{\varepsilon} \sum_{l=0}^{k-1} \|\nabla_b^l (\eta u)\|^2
$$
and one finishes the proof for the case $k-2m = 1$ as before.

Now we prove (\ref{eq:u0m}) for $m$, assuming that the statement has been proved for all larger $m$'s. We then estimate $\|T^m \nabla_b^{k-2m} u^0\|^2$. If one of the $\nabla_b$ is $\Zbar$, we proceed exactly as before and commute the $\Zbar$ until it hits $u^0$. This proves the desired estimate with the induction hypothesis on $m$. If now all $\nabla_b$ are $Z$'s, then as before we only need to bound $|(iT T^m Z^{k-2m-1}u^0, T^m Z^{k-2m-1}u^0)|$. One certainly has $$|(iT T^{m} Z^{k-2m-1}u^0, T^{m}Z^{k-2m-1} u^0)| \leq \varepsilon \|T^{m+1}Z^{k-2m-1}u^0\|^2 + \varepsilon^{-1} \sum_{l=0}^{k-1} \|\nabla_b^{l} (\eta u)\|^2.$$ To estimate $\|T^{m+1}Z^{k-2m-1}u^0\|^2$, we write $T^{m+1}Z^{k-2m-1}u^0 = T\Lambda^0 T^{m}Z^{k-2m-1} (\eta u) + T[T^{m}Z^{k-2m-1}\eta_1, \Lambda^0](\eta u)$. The second term is bounded by $$\sum_{l=0}^k \|\nabla_b^l (\eta u)\|.$$ The first term satisfies
\begin{align*}
&\|T\Lambda^0 T^{m} Z^{k-2m-1}(\eta u)\|^2 \\
\leq &\|\nabla_b^{k-1}(\eta u)\|^2 + C_0 \varepsilon_0^2 \|\nabla_b \Lambda^0 T^{m} Z^{k-2m-1}(\eta u)\|^2 + C_0 \varepsilon_0^2 \|T\Lambda^0 T^{m}Z^{k-2m-1} (\eta u)\|^2,
\end{align*}
from which it follows that
\begin{align*}
&\|T\Lambda^0 T^{m}Z^{k-2m-1} (\eta u)\|^2\\
 \leq & \|\nabla_b^{k-1}(\eta u)\|^2 + C \|\nabla_b \Lambda^0 T^{m} Z^{k-2m-1}(\eta u)\|^2 \\
 \leq & C \|\nabla_b^{k-1}(\eta u)\|^2 + C \|\nabla_b^k (\eta u)\|^2.
\end{align*}
Hence
$$
|(iT T^m Z^{k-2m-1}u^0, T^m Z^{k-2m-1} u^0)| \leq \varepsilon \|\nabla_b^k (\eta u)\|^2 + C_{\varepsilon} \sum_{l=0}^{k-1} \|\nabla_b^l (\eta u)\|^2
$$
and one finishes the proof for this case as before.

\subsection{Estimate for $u^+$}

Now we turn to estimate $u^+$. Recall we introduced a sequence of cut-offs $\eta_0 = \eta, \eta_1, \eta_2, \dots$, and a sequence of Fourier multipliers $\Lambda^+_0 = \Lambda^+, \Lambda^+_1, \Lambda^+_2, \dots$, such that $\eta_k \eta_{k+1} = \eta_k$, and $\Lambda^+_k \Lambda^+_{k+1} = \Lambda^+_k$ for all $k \geq 0$. In fact the Fourier multiplier for $\Lambda^+_{k+1}$ is identically equal to 1 on a neighborhood of the support of that of $\Lambda^+_k$. Also, the cut-off function $\teta$ dominates all the $\eta_j$'s, in the sense that $\eta_j \teta = \eta_j$ for all $j$. We wrote $u^+ = \eta_1 \Lambda^+ \eta u$, and $v$ is any solution to $Zv = u$ on $Q_2$. The estimate we shall prove is
\begin{equation}\label{eq:u+}
\|\nabla_b^k u^+\| \leq C_k (\|\nabla_b^{k-1} \Zbar u^+\| + \sum_{l=0}^{k-1} \|\nabla_b^l (\eta_k u)\| + \|\eta_{k+1} v\|)
\end{equation}
for all $k \geq 1$.

To prove this, the first observation is the following:

\begin{lem} \label{lem:v+simple}
For all $k \geq 1$,
$$
\|\nabla_b^k v^+\| \leq C_k \left(\|\nabla_b^{k-1} Z v^+\| + \|\eta v\|\right),
$$
where $v^+ = \eta_1 \Lambda^+ \eta v$. 
\end{lem}

The proof of this inequality is the same as that of (\ref{eq:u-}), except that one reverses the role of $Z$ and $\Zbar$, and replaces $u^-$ by $v^+$. It does not make use of the fact that $v$ solves $Zv = u$. By the same token,

\begin{lem} \label{lem:v+}
For all $k \geq 1$ and all $j \geq 1$, 
$$
\|\nabla_b^k (\eta_j \Lambda^+_{j-1} \eta_{j-1} v)\| \leq C_{j,k} \left(\|\nabla_b^{k-1} Z (\eta_j \Lambda^+_{j-1} \eta_{j-1} v)\| + \|\eta_{j-1} v\|\right).
$$
\end{lem}

Another useful lemma is the following: 

\begin{lem} \label{lem:v+u+}
For any $k \geq 1$ and any $j \geq 1$, there exist pseudodifferential operators $S_0$ and $S_{-k}$, smoothing of order 0 and $k$ respectively, such that $$\eta_j \Lambda^+_{j-1} \eta_{j-1} u = Z (\eta_j \Lambda^+_j \eta_{j-1} v) + S_0 (\eta_{j+1} \Lambda^+_j \eta_j v) + S_{-k} (\eta_j v).$$ In particular, when $j = 1$,
$$u^+ = Zv^+ + S_0(\eta_2 \Lambda^+_1 \eta_1 v) + S_{-k} (\eta_1 v).$$
\end{lem}

\begin{proof}
For all $j \geq 1$,
\begin{align}
\eta_j \Lambda^+_{j-1} \eta_{j-1} u
&= \eta_j \Lambda^+_{j-1} \eta_{j-1} Zv \notag \\
&= \eta_j \Lambda^+_{j-1} \eta_{j-1} Z (\eta_j v) \notag \\
&= \eta_j \Lambda^+_{j-1} \eta_{j-1} Z (\eta_{j+1} \Lambda^+_j \eta_j v) + \eta_j \Lambda^+_{j-1} \eta_{j-1} Z (\eta_{j+1} (1-\Lambda^+_j) \eta_j v). \label{eq:vj+}
\end{align}
We shall argue that the second term on the last line is $S_{-k}(\eta_j v)$ for any $k \geq 1$.

Since the Fourier multiplier $\Psi^+_j$ of $\Lambda^+_j$ is identically 1 on a neighborhood of the support of $\Psi^+_{j-1}$, there exists a Fourier multiplier $\Psi^+_{j,0}$ such that $\Psi^+_j$ is identically 1 on the support of $\Psi^+_{j,0}$, and $\Psi^+_{j,0}$ is identically 1 on the support of $\Psi^+_{j-1}$. Writing $\Lambda^+_{j,0}$ for the Fourier multiplier operator corresponding to $\Psi^+_{j,0}$, we have
$$(1-\Lambda^+_j) = (1-\Lambda^+_{j,0})(1-\Lambda^+_j);$$ indeed $\Psi^+_j \equiv 1$ on the support of $\Psi^+_{j,0}$ implies that $\Lambda^+_{j,0} (1-\Lambda^+_j) = 0$. Putting this back in the second term (\ref{eq:vj+}), and commuting $1-\Lambda^+_{j,0}$ until it hits $\Lambda^+_{j-1}$, we get 
$$
\eta_j \Lambda^+_{j-1} (1-\Lambda^+_{j,0}) \eta_{j-1} Z (\eta_{j+1} (1-\Lambda^+_j) \eta_j v) + \eta_j \Lambda^+_{j-1} [\eta_{j-1} Z \eta_{j+1}, 1-\Lambda^+_{j,0}] (1-\Lambda^+_j) \eta_j v).
$$
The first term here is zero, since $$\Lambda^+_{j-1} (1-\Lambda^+_{j,0}) = 0;$$ the second term here is $$\eta_j \Lambda^+_{j-1} S_0 (1-\Lambda^+_j) \eta_j v.$$ Again writing $(1-\Lambda^+_j) = (1-\Lambda^+_{j,0})(1-\Lambda^+_j)$ and commuting $1-\Lambda^+_{j,0}$ until it hits $\Lambda^+_{j-1}$, we get that this is
$$
\eta_j \Lambda^+_{j-1} S_{-1} (1-\Lambda^+_j) \eta_j v.
$$
Repeating this argument, it is clear that we can make this $\eta_j \Lambda^+_{j-1} S_{-k} (1-\Lambda^+_j) \eta_j v$ for any $k$, and this is thus $S_{-k} (\eta_j v)$.

Next, the first term in (\ref{eq:vj+}) is
\begin{align*}
&\eta_j \Lambda^+_{j-1} \eta_{j-1} Z (\eta_{j+1} \Lambda^+_j \eta_j v)\\
=& Z \eta_j \Lambda^+_{j-1} \eta_{j-1} \eta_{j+1} \Lambda^+_j \eta_j v + [\eta_j \Lambda^+_{j-1} \eta_{j-1}, Z] (\eta_{j+1} \Lambda^+_j \eta_j v)\\
=& Z \eta_j \Lambda^+_{j-1} \eta_{j-1} \Lambda^+_j \eta_j v + S_0 (\eta_{j+1} \Lambda^+_j \eta_j v).
\end{align*}
By writing $\Lambda^+_j  = 1 - (1-\Lambda^+_j )$, the first term in the last line is equal to
\begin{align*}
Z (\eta_j \Lambda^+_{j-1} \eta_{j-1} v) - Z \eta_j \Lambda^+_{j-1} \eta_{j-1} (1-\Lambda^+_j) \eta_j v.
\end{align*}
We only need to argue now that the second term in the last line is $S_{-k} (\eta_j v)$ for any $k$. But we only need to adopt the strategy above again: writing  $(1-\Lambda^+_j) = (1-\Lambda^+_{j,0})(1-\Lambda^+_j)$ and commuting $1-\Lambda^+_{j,0}$ until it hits $\Lambda^+_{j-1}$, we get that this is
$S_{-k}( \eta_j v)$ for any $k$.
\end{proof}

It follows that

\begin{lem} \label{lem:lowerorderv+}
For all $k \geq 1$,
$$\sum_{l=0}^k \|\nabla_b^l (\eta_2 \Lambda^+_1 \eta_1 v)\| \leq C (\sum_{l=0}^{k-1} \|\nabla_b^l (\eta_k u)\| + \|\eta_{k+1} v\|).$$
\end{lem}

\begin{proof}
We shall prove by induction on $k$ that for all $j,k \geq 1$,
$$\sum_{l=0}^k \|\nabla_b^l (\eta_j \Lambda^+_{j-1} \eta_{j-1} v)\| \leq C (\sum_{l=0}^{k-1} \|\nabla_b^l (\eta_{j+k-2} u)\| + \|\eta_{j+k-1} v\|).$$
The case $j=2$ yields the current lemma. Assume this has been proved for $k-1$, and we prove the statement for $k$.
By Lemma~\ref{lem:v+}, $$\|\nabla_b^k (\eta_j \Lambda^+_{j-1} \eta_{j-1} v)\| \leq C (\|\nabla_b^{k-1} Z (\eta_j \Lambda^+_{j-1} \eta_{j-1} v)\| + \|\eta_{j-1} v\|).$$ Now by Lemma~\ref{lem:v+u+}, one has
$$
Z (\eta_j \Lambda^+_{j-1} \eta_{j-1} v) = \eta_j \Lambda^+_{j-1} \eta_{j-1} u + S_0(\eta_{j+1} \Lambda^+_j \eta_j v) + S_{-(k-1)}(\eta_j v).
$$
Hence one only needs to estimate $\|\nabla_b^{k-1}S_0(\eta_{j+1} \Lambda^+_j \eta_j v)\|$, which can be estimated by induction hypothesis since this involves fewer than $k$ derivatives on $\eta_{j+1} \Lambda^+_j \eta_j v$.
\end{proof}

As a result,

\begin{lem} \label{lem:highorderv+}
For all $k \geq 0$,
$$\|\nabla_b^{k+1} v^+ \| \leq C(\|\nabla_b^k u^+\| + \sum_{l=0}^{k-1} \|\nabla_b^l (\eta_k u)\| + \|\eta_{k+1} v\|).$$
\end{lem}

\begin{proof}
By Lemma~\ref{lem:v+simple}, $$\|\nabla_b^{k+1} v^+\| \leq C (\|\nabla_b^k Z v^+\| + \|\eta v\|).$$ Now by Lemma~\ref{lem:v+u+}, one has
$$
Z v^+ = u^+ + S_0(\eta_2 \Lambda^+_1 \eta_1 v) + S_{-k}(\eta_1 v).
$$
Hence one only needs to estimate $\|\nabla_b^k S_0(\eta_2 \Lambda^+_1 \eta_1 v)\|$, which can be estimated by Lemma~\ref{lem:lowerorderv+}.
\end{proof}

Now we prove (\ref{eq:u+}) by induction on $k$.

When $k = 1$, we need only estimate $\|Z u^+\|^2$. But by Lemma~\ref{lem:v+u+}, we have
$$u^+ = Zv^+ + S_0(\eta_2 \Lambda^+_1 \eta_1 v) + S_{-1} (\eta_1 v).$$
Hence
\begin{align*}
&\|Zu^+\|^2 \\
=& (ZZv^+, Zu^+) + O(\|Z \eta_2 \Lambda^+_1 \eta_1 v\| \|Zu^+\| ) + O(\|\eta_1 v\| \|Z u^+\|)\\
=& (ZZv^+, Zu^+) + O(\|\eta_1 u\| \|Zu^+\| ) + O(\|\eta_2 v\| \|Z u^+\|)) \quad \text{by Lemma~\ref{lem:lowerorderv+}} \\
=& -(Zv^+, Z \Zbar u^+) - (Zv^+, iT u^+) + O(\|Zv^+\|\|\nabla_b u^+\|) + O(\|\eta_1 u\| \|Zu^+\| ) + O(\|\eta_2 v\| \|Z u^+\|)\\
=& -(Zv^+, Z \Zbar u^+) - (Zv^+, iT u^+) + O(\|\eta_1 u\| \|\nabla_b u^+\| ) + O(\|\eta_2 v\| \|\nabla_b u^+\|) \quad \text{by Lemma~\ref{lem:highorderv+}} \\
=& (\Zbar Z v^+, \Zbar u^+) + (iT v^+, \Zbar u^+) + O(\|\nabla_b^2 v^+\| \|u^+\|) + O(\|\nabla_b v^+\|(\|\nabla_b u^+\|+\|u^+\|))\\
& + O(\|\eta_1 u\| \|\nabla_b u^+\| ) + O(\|\eta_2 v\| \|\nabla_b u^+\|) \\
=& O(\|\nabla_b^2 v^+\| \|\Zbar u^+\|) + O(\|\nabla_b^2 v^+\| \|u^+\|) + O(\|\nabla_b v^+\|(\|\nabla_b u^+\|+\|u^+\|)) \\
& + O(\|\eta_1 u\| \|\nabla_b u^+\| ) + O(\|\eta_2 v\| \|\nabla_b u^+\|) \\
\leq & C(\varepsilon \|\nabla_b^2 v^+\|^2 + \varepsilon^{-1} \|\Zbar u^+\|^2 + \varepsilon^{-1} \|u^+\|^2 + \varepsilon \|Zu^+\|^2 + \varepsilon^{-1} \|\nabla_b v^+\|^2 + \varepsilon^{-1} \|\eta_1 u\|^2 + \varepsilon^{-1} \|\eta_2 v\|^2).
\end{align*}
Absorbing $C\varepsilon \|Zu^+\|^2$ to the left hand side, we get
\begin{align*}
\|Zu^+\|^2 
\leq  C(\varepsilon \|\nabla_b^2 v^+\|^2 + \varepsilon^{-1} \|\Zbar u^+\|^2 + \varepsilon^{-1} \|\nabla_b v^+\|^2 + \varepsilon^{-1} \|\eta_1 u\|^2 + \varepsilon^{-1} \|\eta_2 v\|^2).
\end{align*}
Now by Lemma~\ref{lem:highorderv+}, one estimates $\|\nabla_b^2 v^+\|$ and $\|\nabla_b v^+\|$:
$$
\|\nabla_b^2 v^+\| \leq  C \left(\|\nabla_b u^+\| + \|\eta_1 u\| + \|\eta_2 v\|\right),
$$
$$
\|\nabla_b v^+\| \leq C (\|u^+\| + \|\eta_1 v\|) \leq C (\|\eta_1 u\| + \|\eta_2 v\|).
$$
Together, we get
$$
\|Zu^+\|^2 \leq C \left(\|\Zbar u^+\|^2 + \|\eta_1 u\| + \|\eta_2 v\|\right)
$$
as desired.

Next, to prove (\ref{eq:u+}) for a general $k$, we prove the following statement by induction on $m$ for all $0 \leq m \leq \lfloor \frac{k}{2} \rfloor$ and $\varepsilon > 0$:
\begin{equation}\label{eq:u+m}
\|T^m \nabla_b^{k-2m} u^+\| \leq \varepsilon \|\nabla_b^k u^+ \| + C_{\varepsilon} \left(\|\nabla_b^{k-1} \Zbar u^+\| + \sum_{l=0}^{k-1} \|\nabla_b^l (\eta_k u)\| + \|\eta_{k+1} v\|\right)
\end{equation}
In fact the case $m = 0$ readily implies (\ref{eq:u+}) for $k$.

Again we begin from the biggest possible value of $m$. Suppose $k-2m = 0$. Then we need to estimate $\|T^m u^+\|$. Now 
\begin{align*}
&\|T^m u^+\|^2 \\
=& (T^m Zv^+, T^m u^+) + O((\|\nabla_b^k(\eta_2 \Lambda^+_1 \eta_1 v)\| + \|\eta_1 v\|) \|T^m u^+\|) \\
=& -(T^m v^+, T^m \Zbar u^+) + O(\sum_{l=0}^k \|\nabla_b^l v^+\| \|T^m u^+\|) + O( \|T^m v^+\| \sum_{l=0}^k \|\nabla_b^l u^+\|) \\
& \quad + O((\sum_{l=0}^{k-1} \|\nabla_b^l(\eta_k u)\| + \|\eta_{k+1} v\|) \|T^m u^+\|) \\
=& (\nabla_b T^m v^+, \nabla_b T^{m-1} \Zbar u^+) + O(\sum_{l=0}^k \|\nabla_b^l v^+\| \|T^m u^+\|) + O( \|T^m v^+\| \sum_{l=0}^k \|\nabla_b^l u^+\|) \\
& \quad + O((\sum_{l=0}^{k-1} \|\nabla_b^l(\eta_k u)\| + \|\eta_{k+1} v\|) \|T^m u^+\|) 
\end{align*}
the last line following by writing one of the $T$'s in $T^m \Zbar u^+$ as a commutator and integrating by parts. Hence 
\begin{align*}
&\|T^m u^+\|^2 \\
\leq & \varepsilon \|\nabla_b^{k+1} v^+\|^2 + \varepsilon \|\nabla_b^k u^+\|^2 + C_{\varepsilon} ( \|\nabla_b^{k-1} \Zbar u^+\|^2 + \sum_{l=0}^k \|\nabla_b^l v^+\|^2 + \sum_{l=0}^{k-1} \|\nabla_b^l(\eta_k u)\|^2 + \|\eta_{k+1} v\|^2 ).
\end{align*}
Now we invoke the Lemma~\ref{lem:highorderv+} to estimate $\|\nabla_b^{k+1} v^+\|$ and $\sum_{l=0}^k \|\nabla_b^l v^+\|$.
Together,
$$
\|T^m u^+\|^2 \leq \varepsilon \|\nabla_b^k u^+\|^2 + C_{\varepsilon} \left( \|\nabla_b^{k-1} \Zbar u^+\|^2 + \sum_{l=0}^{k-1} \|\nabla_b^l (\eta_k u)\|^2 + \|\eta_{k+1} v\|^2 \right)
$$
which implies the desired estimate for $T^m u^+$.

Next, we estimate estimate $\|T^m Z^{k-2m} u^+\|$ for any $0 \leq m \leq \lfloor \frac{k}{2} \rfloor$, if $k - 2m > 0$. Then 
\begin{align*}
&\|T^m Z^{k-2m} u^+\|^2 \\
=& (T^m Z^{k-2m+1} v^+, T^m Z^{k-2m} u^+) + O((\|\nabla_b^k (\eta_2 \Lambda^+_1 \eta_1 v)\| + \|\eta_1 v\|) \| T^m Z^{k-2m} u^+\|) \\
=& -(T^m Z^{k-2m} v^+, T^m Z^{k-2m} \Zbar u^+) - (k-2m) (T^m Z^{k-2m} v^+, iT^{m+1}Z^{k-2m-1} u^+) \\& + O(\sum_{l=0}^{k} \|\nabla_b^l v^+\| \sum_{j=0}^k \|\nabla_b^j u^+\|) + O((\sum_{l=0}^{k-1} \|\nabla_b^l (\eta_k u)\| + \|\eta_{k+1} v\|)\| T^m Z^{k-2m} u^+\|)
\end{align*}
Now in the first term, we split off one $Z$ in $T^m Z^{k-2m} \Zbar$ and integrate by parts. Also, in the second term, we split off one $T$ in $T^{m+1} Z^{k-2m-1}$ and integrate by parts; then we split off one $Z$ in $T^m Z^{k-2m}$ and integrate by parts. We get
\begin{align*}
&\|T^m Z^{k-2m} u^+\|^2 \\
=& (\Zbar T^m Z^{k-2m} v^+, T^m Z^{k-2m-1} \Zbar u^+) + (k-2m)(iT^{m+1} Z^{k-2m-1} v^+, T^m Z^{k-2m} u^+)\\
& + O(\|\nabla_b^{k+1} v^+\|\sum_{l=0}^{k-1} \|\nabla_b^l u^+\|) +  O(\sum_{l=0}^{k} \|\nabla_b^l v^+\| \sum_{j=0}^k \|\nabla_b^l u^+\|) \\
& + O((\sum_{l=0}^{k-1} \|\nabla_b^l (\eta_k u)\| + \|\eta_{k+1} v\|)\| T^m Z^{k-2m} u^+\|) \\
\leq & O(\|\nabla_b^{k+1} v^+\| \|\nabla_b^{k-1} \Zbar u^+\|)  + O(\|\nabla_b^{k+1} v^+\|\|\sum_{l=0}^{k-1} \|\nabla_b^l u^+\|) \\&+  O(\sum_{l=0}^{k} \|\nabla_b^l v^+\| \sum_{j=0}^k \|\nabla_b^l u^+\|) + O((\sum_{l=0}^{k-1} \|\nabla_b^l (\eta_k u)\| + \|\eta_{k+1} v\|)\| T^m Z^{k-2m} u^+\|)  \\
\end{align*}
As a result,
\begin{align*}
&\|T^m Z^{k-2m} u^+\|^2 \\
\leq & \varepsilon \|\nabla_b^{k+1} v^+\|^2 + \varepsilon \|\nabla_b^k u^+\|^2 \\& + C_{\varepsilon} \left(\|\nabla_b^{k-1} \Zbar u^+\|^2  + \sum_{l=0}^{k} \|\nabla_b^l v^+\|^2 + \sum_{l=0}^{k-1} \|\nabla_b^l (\eta_k u)\|^2  + \|\eta_{k+1} v\|^2 \right)  \\
\end{align*}
and the desired estimate follows upon invoking Lemma~\ref{lem:highorderv+}. 

Now suppose we have proved (\ref{eq:u+m}) for all strictly bigger $m$'s, and want to prove the inequality for $m$. Then we need to estimate $\|T^m \nabla_b^{k-2m} u^+\|$. If one of the $\nabla_b$ is $\Zbar$, we commute until that $\Zbar$ hits $u^+$, obtaining an error that has more $T$ in it, which one can estimate by the induction hypothesis. Otherwise all $\nabla_b$ are $Z$'s, and the estimate follows from what we have proved above. This completes the proof of (\ref{eq:u+}).

Now putting (\ref{eq:u-}), (\ref{eq:u0}) and (\ref{eq:u+}) together, and remembering that $\eta u = u^- + u^0 + u^+$, we get
$$
\|\nabla_b^k (\eta u) \| \leq C_k \left( \|\nabla_b^{k-1} \Zbar (\eta u)\| + \sum_{l=0}^{k-1} \| \nabla_b^l (\eta_k u)\| + \|\eta_{k+1} v\| \right).
$$
But $\sum_{l=0}^{k-1} \|\nabla_b^l (\eta_k u)\|$ involves fewer than $k$ derivatives of $\eta_k u$, and can be estimated if we iterate the above inequality. In fact
$$
\sum_{l=0}^{k-1} \| \nabla_b^l (\eta_k u)\| \leq C \left(\sum_{l=0}^{k-2}\|\nabla_b^l \Zbar (\teta u)\| + \|\teta u\| + \|\teta v\|\right). 
$$
It follows that
$$
\|\nabla_b^k (\eta u) \| \leq C_k \left( \sum_{l=0}^{k-1} \| \nabla_b^l \Zbar (\teta u)\| + \|\teta v\| \right).
$$
Using the interpolation inequality in Lemma~\ref{lem:interpolate}, one obtains the desired inequality in Proposition~\ref{prop:normalizedsub}.

Finally, we come back to the remark we made after the statement of Proposition~\ref{prop:normalizedsub}. We remark that if we have $Zv + \alpha v = u$ instead of $Zv = u$, where $\alpha$ is a fixed $C^{\infty}$ function on $Q_2$, then the above theorem still holds. See the end of this section for a discussion about that. The key there is to observe that Lemma~\ref{lem:v+u+} above holds under this modified assumption as well. In fact, then 
\begin{align*}
\eta_j \Lambda^+_{j-1} \eta_{j-1} u
&= \eta_j \Lambda^+_{j-1} \eta_{j-1} (Z + \alpha) v \\
&= \eta_j \Lambda^+_{j-1} \eta_{j-1} (Z + \alpha) (\eta_j v) \\
&= \eta_j \Lambda^+_{j-1} \eta_{j-1} (Z + \alpha) (\eta_{j+1} \Lambda^+_j \eta_j v) + \eta_j \Lambda^+_{j-1} \eta_{j-1} (Z + \alpha) (\eta_{j+1} (1-\Lambda^+_j) \eta_j v).
\end{align*}
The $\alpha$ in the first term contributes only $S_0(\eta_{j+1}\Lambda^+_j \eta_j v)$, while the $\alpha$ in the last term contributes only $S_{-k-1}(\eta_j v)$.

\end{document}